\documentclass[reqno,10pt]{amsart}

\title[Anosov representations and dominated splittings]{Anosov representations and dominated splittings}

\author[J.~Bochi]{Jairo Bochi}

\author[R.~Potrie]{Rafael Potrie}

\author[A.~Sambarino]{Andr\'es Sambarino}

\thanks{J.B.\ was partially supported by projects Fondecyt 1180371 and Conicyt PIA ACT172001. R.P.\ was partially supported by CSIC-618 and FCE-135352. A.S.\ was partially supported by IFUM and CSIC group 618.}
\thanks{This preprint was published in \textit{J.\ Eur.\ Math.\ Soc.\ }21 (2019), no.\ 11, pp.\ 3343--3414. \doi{10.4171/JEMS/905}}

\usepackage[colorlinks = true,backref = page,hyperindex,breaklinks,bookmarksopen=false]{hyperref} % Links in the pdf. Backref option: each bibliographical entry denotes where it was cited
\usepackage{amssymb}
\usepackage{mathtools}      % improves amsmath
\usepackage{mathabx}        % makes many symbols (as $\leq$) more beautiful; must be loaded after mathtools 
\usepackage[bb = fourier,cal = euler,scr = rsfs]{mathalfa}	% selecting fancy math fonts
\usepackage{enumitem}       % special itemize with custom tags and labels that work 
\usepackage[english]{babel}
\usepackage{tikz}			% figures
\usepackage{tikz-cd}		% commutative diagrams
\usepackage[font = small]{caption}	% smaller captions for figures
\usepackage{nicefrac}
\usepackage[normalem]{ulem}

\usepackage{xcolor}
\hypersetup{
    colorlinks,
    linkcolor={red!50!black},
    citecolor={blue!50!black},
    urlcolor={blue!80!black}
}

\renewcommand*{\backref}[1]{}
\renewcommand*{\backrefalt}[4]{\quad \tiny
  \ifcase #1 (\textbf{NOT CITED.})%
  \or    (Cited on page~#2.)%
  \else   (Cited on pages~#2.)%
  \fi}

\makeatletter

\def\MRbibitem{\@ifnextchar[\my@lbibitem\my@bibitem}

\def\mybiblabel#1#2{\@biblabel{{\hyperref{http://www.ams.org/mathscinet-getitem?mr=#1}{}{}{#2}}}}

\def\myhyperanchor#1{\Hy@raisedlink{\hyper@anchorstart{cite.#1}\hyper@anchorend}}

\def\my@lbibitem[#1]#2#3#4\par{%
  \item[\mybiblabel{#2}{#1}\myhyperanchor{#3}\hfill]#4%
  \@ifundefined{ifbackrefparscan}{}{\BR@backref{#3}}%
  \if@filesw{\let\protect\noexpand\immediate% write to aux-file
    \write\@auxout{\string\bibcite{#3}{#1}}}\fi\ignorespaces%
}

\def\my@bibitem#1#2#3\par{%
  \refstepcounter\@listctr% standard tex item counter for the generic item number
  \item[\mybiblabel{#1}{\the\value\@listctr}\myhyperanchor{#2}\hfill]#3%
  \@ifundefined{ifbackrefparscan}{}{\BR@backref{#2}}%
  \if@filesw\immediate\write\@auxout% write to aux-file
    {\string\bibcite{#2}{\the\value\@listctr}}\fi\ignorespaces%
}

\makeatother
\newcommand{\xqedhere}[2]{%
  \rlap{\hbox to#1{\hfil\llap{\ensuremath{#2}}}}}

\newcommand{\Z}{\mathbb{Z}} \newcommand{\ZZ}{\Z}

\newcommand{\R}{\mathbb{R}} \newcommand{\RR}{\R}
\newcommand{\C}{\mathbb{C}}
\newcommand{\N}{\mathbb{N}}
\renewcommand{\P}{\mathbb{P}}

\newcommand{\cC}{\mathcal{C}}
\newcommand{\cD}{\mathcal{D}}\newcommand{\cE}{\mathcal{E}}
\newcommand{\cG}{\mathcal{G}}

\newcommand{\cO}{\mathcal{O}}

\newcommand{\cT}{\mathcal{T}}

\newcommand{\en}{\subset}
\newcommand{\eps}{\varepsilon}
\newcommand{\G}{\Gamma} %{\sf{\Gamma}}    % the group
\newcommand{\Gr}{\mathscr G}    % Grassmannian
\newcommand{\<}{\langle}
\renewcommand{\>}{\rangle}
\newcommand{\E}{\Sigma}
\newcommand{\g}{\gamma}

\newcommand{\om}{\omega}
\newcommand{\posgen}{\scr F^{(2)}}
\renewcommand{\t}{\theta}

\renewcommand{\L}{\Lambda}
\newcommand{\p}{{\mathsf{f}}}
\newcommand{\zz}{{\mathsf{Z}}}
\newcommand{\wk}{\check}
\renewcommand{\aa}{\underline{a}}
\newcommand{\dd}{\diamondsuit}

\newcommand{\scr}{\mathscr}
\renewcommand{\sf}[1]{{\mathsf{#1}}}

\newcommand{\cal}{\mathcal}
\renewcommand{\frak}{\mathfrak}

\DeclareMathOperator{\ii}{i}

\DeclareMathOperator{\plucker}{\L}

\DeclareMathOperator{\SL}{SL}
\DeclareMathOperator{\PSL}{PSL}
\DeclareMathOperator{\GL}{GL}

\DeclareMathOperator{\PGL}{PGL}

\DeclareMathOperator{\id}{id}
\DeclareMathOperator{\inte}{int}

\DeclareMathOperator{\traza}{Trace}
\DeclareMathOperator{\ad}{ad}
\DeclareMathOperator{\Ad}{Ad}

\DeclareMathOperator{\jac}{jac}

\newcommand{\mm}{\mathbf{m}}
\newcommand{\Id}{\mathrm{Id}}
\newcommand{\st}{\,\mathord{\colon}\,} % "such that" in the definition of a set

\newcommand{\cu}{\mathrm{cu}}
\newcommand{\cs}{\mathrm{cs}}
\newcommand{\grass}{\Gr}      %{\mathrm{Gr}}
\renewcommand{\angle}{\measuredangle}
\newcommand{\Wedge}{\mathsf{\Lambda}}  % Big wedge (for exterior powers)

\renewcommand{\epsilon}{\varepsilon}

\newcommand{\arxiv}[1]{Preprint \href{http://arxiv.org/abs/#1}{arXiv:{#1}}}
\newcommand{\doi}[1]{\href{http://dx.doi.org/#1}{doi:{#1}}}

\hypersetup{bookmarksdepth  =  3} % Depth of sections/subs... to have bookmark links in the pdf;
\numberwithin{equation}{section}     % Makes labeled equations easier to find.

\setlist[enumerate,1]{label = {\upshape(\roman*)},ref = \roman*}
\setlist[enumerate,2]{label = {\upshape(\alph*)},ref = \alph*}

\newtheorem{teo}{Theorem}[section]

\newtheorem{cor}[teo]{Corollary}
\newtheorem{lem}[teo]{Lemma}
\newtheorem{prop}[teo]{Proposition}

\theoremstyle{definition}
\newtheorem*{claim}{Claim}
\newtheorem{question}[teo]{Question}

\newtheorem{remark}[teo]{Remark}

\theoremstyle{remark}
\newtheorem{obs}[teo]{Remark}
\newtheorem*{ack}{Acknowledgements}

\begin{document}

\begin{abstract}
We provide a link between Anosov representations introduced by Labourie and dominated splitting of linear cocycles. This allows us to obtain equivalent characterizations for Anosov representations and to recover recent results due to Gu\'eritaud--Guichard--Kassel--Wienhard \cite{GGKW} and Kapovich--Leeb--Porti \cite{KLP2} by different methods. We also give characterizations in terms of multicones and cone-types inspired in the work of Avila--Bochi--Yoccoz \cite{ABY} and Bochi--Gourmelon \cite{BG}. Finally we provide a new proof of the higher rank Morse Lemma of Kapovich--Leeb--Porti \cite{KLP2}.
\end{abstract}

\maketitle

\tableofcontents

\section{Introduction}\label{s.intro}

The aim of this paper is to expose and exploit connections between the following two areas:
\begin{enumerate}[label = {\arabic*}.]
	\item linear representations of discrete groups;
	\item differentiable dynamical systems.
\end{enumerate}
More specifically, we show that \emph{Anosov representations} are closely related to \emph{dominated splittings}. This relation allows us to reobtain some results about Anosov representations, and to give new characterizations of them.

\medskip

Anosov representations where introduced by Labourie \cite{Labourie-AnosovFlows} in his study of the Hitchin component (\cite{Hitchin}). They provide a stable class of discrete faithful representations of word-hyperbolic groups into semi-simple Lie groups, that unifies examples of varying nature. Since then, Anosov representations have become a main object of study, being subject of various deep results (see for example \cite{GuichardWienhard}, the surveys \cite{BCS,Wienhard-Survey} and references therein). Recently, new characterizations of Anosov representations have been found by \cite{GGKW,KLP1,KLP2,KLP3}, these characterizations considerably simplify the definition.
It is now fairly accepted in the community that Anosov representations are a good generalization of convex co-compact groups to higher rank.

\medskip

In differentiable dynamical systems, 
the notion of \emph{hyperbolicity}, as introduced by Anosov and Smale \cite{Smale} plays a central role.
Early on, it was noted that weaker forms of hyperbolicity (partial, nonuniform, etc.) should also be studied:
see \cite{BDV} for a detailed account.
One of these is the the notion of \emph{dominated splittings}, which can be thought as a projective version of hyperbolicity: see Section~\ref{s.DS}.

\medskip

As mentioned above, in this paper we benefit from the viewpoint of differentiable dynamics in the study of linear representations.
For example, it turns out that a linear representation is Anosov if and only if its associated linear flow has a dominated splitting: see Subsection~\ref{ss.equivalences} for the precise statements.

\medskip

Let us summarize the contents of this paper.

In Section~\ref{s.DS} we describe the basic facts about dominated splittings that will be used in the rest of the paper. We present the characterization of dominated splittings given by \cite[Theorem~A]{BG}. We rely on this theorem in different contexts throughout the paper.

In Section~\ref{s.dominationimplieshyp} we introduce \emph{dominated representations} $\rho$ of a given finitely generated group $\Gamma$ into $\GL(d,\R)$. The definition is simple: we require the gap between some consecutive singular values to be exponentially large with respect to the word length of the group element, that is, 
$$
\frac{\sigma_{p+1}(\rho(\gamma))}{\sigma_{p}(\rho(\gamma))} < C e^{-\lambda |\gamma|} 
\quad \text{for all $\gamma \in \Gamma$},
$$
for some constants $p \in \{1,\dots,d-1\}$, $C>0$, and $\lambda>0$ independent of $\gamma\in\Gamma$.

We show that the existence of dominated representations implies that the group $\Gamma$ is word-hyperbolic: see Theorem~\ref{teo:dominationimplieshyp}. Word-hyperbolicity allows us to consider Anosov representations, and in Section~\ref{s.equiv} we show that they are exactly the same as dominated representations. Many of the results in Sections \ref{s.dominationimplieshyp} and \ref{s.equiv} are not new, appearing in the recent works of \cite{GGKW} and \cite{KLP1,KLP2,KLP3} with different terminology. 
Our proofs are different and make use of the formalism of linear cocycles and most importantly of non-trivial properties of dominated splittings.
At the end of Section~\ref{s.equiv} we pose a few questions where the connection with differentiable dynamics is manifest.

In Section~\ref{s.multicones} we give yet another equivalent condition for a representation to be Anosov, which uses the sofic subshift generated by the cone-types instead of the geodesic flow of the group: see Theorem~\ref{t.multicone_representation}.
This condition is very much inspired in \cite{ABY,BG} and provides nice ways to understand the variation of the limit maps, as well as a quite direct method to check if a representation is dominated. This criterion is used in Section~\ref{s.analytic} to reobtain a basic result from \cite{BCLS} on the analyticity of the limit maps for Anosov representations: see Theorem~\ref{teo.analytic}. 

Section \ref{s.morse} shows how \cite[Theorem~A]{BG}  implies a Morse~Lemma-type statement for the symmetric space of $\PSL(d,\R)$. That result is contained in the recent work of \cite{KLP2}, but we provide here a different approach. This section only relies on Section \ref{s.DS}.

In Section~\ref{s.general} we replace $\GL(d,\R)$ with a real-algebraic non-compact semi-simple Lie group. Representation theory of such groups allows one to reduce most of the general statements to the corresponding ones in $\GL(d,\R)$ by a fairly standard procedure.  Nevertheless, more work is needed to obtain a Morse Lemma for symmetric spaces of non-compact type; this occupies most of Section~\ref{s.general}.

The sections are largely independent; dependence is indicated in Fig.~\ref{f.sections}.

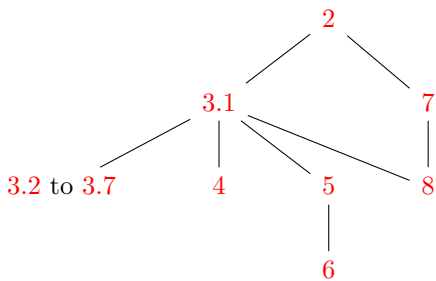
\begin{figure}[hbt]
\begin{center}
\begin{tikzcd}
& & \ref{s.DS} \arrow[dash,dl]\arrow[dash,dr] & \\ 
& \ref{ss.dominatedrep} \arrow[dash,dl]\arrow[dash,d]\arrow[dash,dr]\arrow[dash,drr] & & \ref{s.morse} \arrow[dash,d]\\
\ref{ss.crithyp}\text{ to }\ref{ss.conclusion} & \ref{s.equiv} & \ref{s.multicones} \arrow[dash,d] & \ref{s.general} \\
& & \ref{s.analytic} & 
\end{tikzcd}
\end{center}
\caption{Dependence between sections.}
\label{f.sections}
\end{figure}

\medskip 

The results of this paper were announced at \cite{EHY}.
 
\begin{ack}
We thank Mat\'ias Carrasco, Olivier Guichard, Fanny Kassel, and Jean-Fran\c{c}ois Quint for many interesting discussions. We are greatly indebted to the referee and to O. Guichard for several important corrections, including pointing to a faulty proof. Jairo Bochi acknowledges the kind hospitality of Universidad de la Rep\'ublica. 
\end{ack}

%\marg{J: Cambie el referee a naranja, pues rojo ya es usado.}

\section{Dominated splittings}\label{s.DS}

In the 1970's, Ma\~n\'e introduced the notion of \emph{dominated splittings}, which played an important role in his solution of Smale's Stability Conjecture: see \cite{Samba_dom} and references therein. Independently, dominated splittings had been studied in the theory of Ordinary Differential Equations by the Russian school at least since the 1960s, where it is called \emph{exponential separation}: see \cite{Palmer} and references therein. Dominated splittings continue to be an important subject in Dynamical Systems \cite{CroPo,Samba_dom} and Control Theory \cite{ColKli}.

\subsection{Definition and basic properties of dominated splittings}\label{ss.def_ds}

Let $X$ be a compact metric space.
Let $\mathbb{T}$ be either $\Z$ or $\R$.
Consider a continuous action of $\mathbb{T}$ on $X$, that is, a continuous family of homeomorphisms 
$\{\phi^t \colon X \to X \}_{t \in \mathbb{T}}$ such that $\phi^{t+s}  =  \phi^t \circ \phi^s$.
We call $\{\phi^t\}$ a \emph{continuous flow}.

Let $E$ be a real vector bundle with projection map $\pi \colon E \to X$ and fibers $E_x \coloneqq \pi^{-1}(x)$  of constant dimension $d$.
We endow $E$ with a Riemannian metric (that is, a continuous choice of an inner product on each fiber).
Suppose $\{\psi^t \colon E \to E\}_{t \in \mathbb{T}}$ is a continuous action of $\mathbb{T}$ on $E$ by automorphisms of the vector bundle that covers $\{\phi^t\}$, that is, $\pi \circ \psi^t  =  \phi^t \circ \pi$.
So the restriction of $\psi^t$ to each fiber $E_x$ is a linear automorphism $\psi^t_x$ onto $E_{\phi^t(x)}$.
We say that $\{\psi^t\}$ is a \emph{linear flow} which \emph{fibers over} the continuous flow $\{\phi^t\}$.

The simplest situation is when $\mathbb{T}  =  \Z$ and the vector bundle is trivial, i.e., $E  =  X \times \R^d$ and $\pi(x,v)  =  x$;
in that case the linear flow $\{\psi^t\}$ is called a \emph{linear cocycle},
and in order to describe it is sufficient to specify the maps $\phi  =  \phi^1 \colon X \to X$
and $A \colon X \to \GL(d,\R)$ such that $\psi^1(x,v)  =  (\phi(x), A(x)v)$.
With some abuse of terminology, we sometimes call the pair $(\phi,A)$ a linear cocycle.

\medskip

Suppose that the vector bundle $E$ splits as a direct sum $E^\cu \oplus E^\cs$ of continuous\footnote{In fact, continuity of the bundles follows from condition \eqref{eq:dom1}: \see e.g.\ \cite{CroPo}.} subbundles of constant dimensions.\footnote{$\cu$ and $\cs$ stand for \emph{center-unstable} and \emph{center-stable}, respectively. This terminology is usual in differentiable dynamics.} 
This splitting is called \emph{invariant} under the linear flow $\{\psi^t\}$ if for all $x \in X$ and $t \in \mathbb{T}$,
$$
\psi^t(E^\cu_x)  =  E^\cu_{\phi^t(x)} \, , \quad 
\psi^t(E^\cs_x)  =  E^\cs_{\phi^t(x)} \, .
$$
Such a splitting is called \emph{dominated} (with $E^\cu$ \emph{dominating} $E^\cs$) 
if there are constants $C>0$, $\lambda>0$ such that for all $x \in X$, $t>0$,
and unit vectors $v \in E^\cs_x$, $w\in E^\cu_x$ we have:
\begin{equation}\label{eq:dom1}
\frac{\| \psi^t (v) \| } {\| \psi^t (w) \|} < C e^{-\lambda t} \, .
\end{equation}
Note that this condition is independent of the choice of the Riemannian metric for the bundle.
It is actually equivalent to the following condition (see \cite[p.~156]{ColKli}): 
for all $x \in X$ and all unit vectors $v \in E^\cs_x$, $w\in E^\cu_x$ we have: 
\begin{equation}\label{eq:dom2}
\lim_{t \to +\infty} \frac{\| \psi^t (v) \| } {\| \psi^t (w) \|}  =  0  .
\end{equation}

The bundles of a dominated splitting are unique given their dimensions; more generally:

\begin{prop}\label{p.uniqueness}
Suppose a linear flow $\{\psi^t\}$ has dominated splittings $E^\cu \oplus E^\cs$ and $F^\cu \oplus F^\cs$,
with $E^\cu$ (resp.\ $F^\cu$) dominating $E^\cs$ (resp.\ $F^\cs$).
If $\dim E^\cu \le \dim F^\cs$ then $E^\cu \subset F^\cs$ and $E^\cs \supset E^\cu$.
\end{prop}
See e.g.\ \cite{CroPo} for a proof of this and other properties of dominated splittings.

One can also define domination for invariant splittings into more than two bundles.
This leads to the concept of \emph{finest dominated splitting}, whose uniqueness is basically a consequence of Proposition~\ref{p.uniqueness}; see \cite{BDP,CroPo}.

The existence of a dominated splitting can be characterized in terms of 
cone fields: see e.g.\ \cite{CroPo};
we will use these ideas later in Section~\ref{s.multicones}.

\subsection{Domination in terms of singular values}\label{ss.dom_sing}

Another indirect way to detect the existence of a dominated splitting can be formulated in terms of the ``non-conformality'' of the linear maps. Results of this kind were obtained in \cite{Yoccoz,Lenz} for dimension $2$, and later in \cite{BG} in more generality.\footnote{For a recent generalization to Banach spaces, see \cite{BluMo}.}
Let us explain this characterization.

If $A$ is a linear map between two inner product vector spaces of dimension $d$,
then its \emph{singular values} 
$$
\sigma_1(A) \ge \sigma_2(A) \ge \cdots \ge \sigma_d(A) 
$$
are the eigenvalues of the positive semidefinite operator $\sqrt{A^* A}$,
repeated according to multiplicity.
They equal the semiaxes of the ellipsoid obtained as the $A$-image of the unit ball;
this is easily seen using the singular value decomposition \cite[\S~7.3]{HornJ}.

If $p \in \{1,\dots,d-1\}$ and $\sigma_p(A) > \sigma_{p+1}(A)$, then we 
say that \emph{$A$ has a gap of index $p$}.
In that case, we denote by $U_p(A)$ the $p$-dimensional subspace containing the $p$ biggest axes of the ellipsoid $\{Av \st \|v\| = 1\}$.
Equivalently, $U_p(A)$ is the eigenspace of $\sqrt{A A^*}$ corresponding to the $p$ larger  
eigenvalues.
We also define $S_{d-p}(A) \coloneqq U_{d-p}(A^{-1})$.
Note that $S_{d-p}(A)^\perp  =  A^{-1}(U_p(A))$ and $U_p(A)^\perp  =  A(S_{d-p}(A))$.
See Fig.~\ref{f.boluda}.
\begin{figure}[!hbt]
	\centering
	\begin{tikzpicture}[scale = .85]
		\def\facta{1.35}
		\def\factb{1.2}
		\begin{scope}[rotate = -20]
			\draw[thick] (0,0) circle [radius = 1];
			\draw[thin]  (0,0) ellipse [x radius = {1/\facta}, y radius = \facta];
			\draw[thick] (-\factb,0) -- (\factb,0);
			\draw[thin] (0,-\factb*\facta) -- (0,\factb*\facta) node[right]{\small $S_{d-p}(A)$};
		\end{scope}
		\draw (2,0) edge[out = 20,in = 160,->] node[midway,above]{\small $A$} (3,0);
		\begin{scope}[xshift = 5.25cm,rotate = -70]
			\draw[thin] (0,0) circle [radius = 1];
			\draw[thick] (0,0) ellipse [x radius = {1/\facta}, y radius = \facta];
			\draw[thin] (-\factb,0) -- (\factb,0);
			\draw[thick] (0,-\factb*\facta) -- (0,\factb*\facta) node[right]{\small $U_p(A)$};
		\end{scope}
	\end{tikzpicture}
	\caption{Spaces associated to a linear map $A$.}\label{f.boluda}
\end{figure}
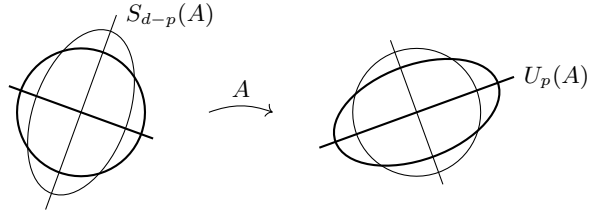

\medskip

The following theorem asserts that the existence of a dominated splitting can be detected in terms of singular values, and also describes the invariant subbundles in these terms:

\begin{teo}[Bochi--Gourmelon {\cite{BG}}]\label{t.BG}
Let $\mathbb{T}  =  \Z$ or $\R$.
Let $\{\psi^t\}_{t \in \mathbb{T}}$ be a linear flow on a vector bundle $E$, fibering over 
a continuous flow $\{\phi^t\}_{t \in \mathbb{T}}$ on a compact metric space $X$.

Then the linear flow  $\{\psi^t\}$ has a dominated splitting $E^\cu \oplus E^\cs$ 
where the dominating bundle $E^\cu$ has dimension $p$ 
if and only if there exist $c>0$, $\lambda>0$ such that for every $x \in X$ and $t\ge 0$ we have
$$
\frac{\sigma_{p+1}(\psi^t_x)}{\sigma_p(\psi^t_x)} < c e^{-\lambda t} \, .
$$
Moreover, the bundles are given by:
\begin{align}
\label{e.BG_cu}
E^\cu_x & =  \lim_{t \to +\infty} U_p \big( \psi^t_{\phi^{-t}(x)} \big) \, , \\
\label{e.BG_cs}
E^\cs_x & =  \lim_{t \to +\infty} S_{d-p} \big( \psi^t_x \big) \, ,
\end{align}
and these limits are uniform.
\end{teo}

(To make sense of the limits above it is necessary to metrize the Grassmann bundle associated to $E$;
the particular way of doing so is irrelevant for the statement.)

In the paper \cite{BG}, the power of the Multiplicative Ergodic Theorem of Oseledets is used to  basically reduce the proof of Theorem~\ref{t.BG} to some angle estimates. Since we will explicitly need such estimates in other parts of this paper, we will also present a sketch of the proof of Theorem~\ref{t.BG} in \S~\ref{ss.sketch_BG}.

\begin{obs}\label{r.dom_complex}
Consider the case of \emph{complex} vector spaces, bundles, etc.
Dominated splittings can be defined analogously;
so can singular values and the subspaces $U_p$, $S_{d-p}$.
Theorem~\ref{t.BG} also extends to the complex case;
indeed it can be deduced from the real case. %\marg{R: comente el footnote como sugiere el referee.}%\footnote{Use that every $d$-dimensional complex vector space can be considered as a $2d$-dimensional real vector space, and that a hermitian inner product on the former induces a inner product on the latter in such a way that the induced norms coincide.}
\end{obs}

\subsection{Domination for sequences of matrices}\label{ss.dom_sequences}
Next, we describe some consequences of Theorem~\ref{t.BG} for sequences of $d \times d$ matrices.\footnote{We note that similar sequences have been considered in \cite{GGKW}, namely sequences that have what they call \emph{coarsely linear increments} (CLI) in certain Cartan projections.}

\medskip

Given $K>1$, define the following compact set: 
$$
\cD(K) \coloneqq \big\{ A \in \GL(d,\R), \ \|A\| \le K , \ \|A^{-1}\| \le K \big\}.
$$
If $I$ is a (possibly infinite) interval in $\Z$, we endow $\cD(K)^I$ with the product topology,
which is compact and is induced e.g.\ by the following metric:
$$
d((A_n), (B_n)) \coloneqq 
\sum_{n\in I} 2^{-|n|} \left(\|A_n - B_n\| + \|A_n^{-1} - B_n^{-1}\|\right)  \, .
$$

Let $p \in \{1,\dots,d-1\}$, $\mu>0$, $c>0$.
For each interval $I\subset \Z$,
let $\cD(K,p,\mu,c,I)$ denote the set of sequences of matrices $(A_n) \in \cD(K)^I$
such that for all $m$, $n \in I$ with $m \ge n$ we have
$$ 
\frac{\sigma_{p+1}}{\sigma_p} (A_{m} \cdots A_{n+1} A_n)  \leq c e^{-\mu (m-n+1)}  \, .
$$

Let us consider the case $I  =  \Z$.
Let $\vartheta$ denote the shift map on the space $\cD(K,p,\mu,c,\Z)$,
and let $A \colon \cD(K,p,\mu,c,\Z) \to \GL(d,\R)$ denote the projection on the zeroth coordinate.
The pair $(\vartheta, A)$ determines a linear cocycle (in the sense explained in \S~\ref{ss.def_ds}). Note that the hypothesis of Theorem~\ref{t.BG} is automatically satisfied.
So we obtain:\footnote{For a similar statement with different notation and not relying on Theorem~\ref{t.BG}, see \cite[Theorem 5.3]{GGKW}.}

\begin{prop}\label{p.BG_sequences}
Fix constants $K>1$, $p \in \{1,\dots,d-1\}$, $\mu>0$, $c>0$.
Then, for each sequence $x  =  (A_n) \in \cD(K,p,\mu,c,\Z)$, the limits:
\begin{align*}
E^\cu(x) &\coloneqq \lim_{n \to +\infty} U_p \big( A_{-1} A_{-2} \cdots A_{-n} \big) \, , \\
E^\cs(x) &\coloneqq \lim_{n \to +\infty} S_{d-p} \big(A_{n-1} A_{n-2} \cdots A_0 \big) \, ,
\end{align*}
exist and are uniform over $\cD(K,p,\mu,c,\Z)$. 
Moreover, $E^\cu \oplus E^\cs$ is a dominated splitting for the linear cocycle $(\vartheta, A)$ defined above.
\end{prop} 

By a compactness argument, the theorem above yields information for finite sequences of matrices (throughout, $\N = \{0,1,2,\dots\}$):

\begin{lem}\label{l.seq_splitting}
Given $K>1$, $\mu>0$, and $c>0$, there exist $\ell_1 \in \N$ and $\delta>0$
with the following properties.
Suppose that $I \subset \Z$ is an interval and $(A_i)_{i\in I}$ is an element of $\cD(K,p,\mu,c,I)$.
If $n < k < m$ all belong to $I$ and $\min \{ k-n, m-k \} > \ell_1$ then:
$$
\angle \big( U_p(A_{k-1} \cdots A_{n+1} A_n) , \, S_{d-p}(A_{m-1} \cdots A_{k+1} A_k) \big) > \delta \, .
$$
\end{lem}
\begin{proof}
The proof is by contradiction.
Assume that there exist numbers $K>1$, $\mu>0$, $c>0$, and sequences $\ell_j \to \infty$, $\delta_j \to 0$ such that for each $j$ there exist an interval $I_j \subset \Z$, an element 
$\big(A_i^{(j)}\big)_{i \in I_j}$ of $\cD(K,p,\mu,c,I_j)$, and integers $n_j < k_j < m_j$ in $I_j$ such that $\min \{ k_j - n_j, m_j - k_j \} > \ell_j$ and 
$$
\angle \big( U_p(A^{(j)}_{k_j-1} \cdots A^{(j)}_{n_j}) , \, S_{d-p}(A^{(j)}_{m_j-1} \cdots A^{(j)}_{k_j}) \big) \le \delta_j \, .
$$ 
Shifting indices, we can assume that $k_j  =  0$ for every $j$.
By a diagonal argument, passing to subsequences we can assume that for each $i \in \Z$, the  matrices $A_i^{(j)}$ (which are defined for sufficiently large $j$) converge to some matrix $A_i$ as $j \to \infty$.
The resulting sequence $x  =  (A_i)_{i\in\Z}$ belongs to $\cD(K,p,\mu,c,\Z)$.
If $E^\cu(x)$ and $E^\cs(x)$ are the limit spaces as in Proposition~\ref{p.BG_sequences},
then their angle is zero, which contradicts domination.
\end{proof}

\section{Domination implies word-hyperbolicity}\label{s.dominationimplieshyp} 

In this section, we define dominated linear representations of a finitely generated group, and prove that groups that admit such representations are word-hyperbolic.

\subsection{Dominated representations and the word-hyperbolicity theorem}\label{ss.dominatedrep}

Let $\Gamma$ be a finitely generated group.
Let us fix a symmetric generating set $S$ of $\Gamma$.
We denote by  $|\gamma|$ the \emph{word-length} of $\gamma \in \Gamma$, i.e.\ the minimum number 
of elements of $S$ required to obtain $\gamma$ as a product of elements of $S$. 
The \emph{word-metric} is defined as: 
\begin{equation}\label{e.word_metric}
d(\gamma,\eta) \coloneqq |\eta^{-1} \gamma| \, .
\end{equation}
Then the action of $\Gamma$ into itself by left multiplication is isometric.

Recall that the group $\Gamma$ is called \emph{word-hyperbolic} if it is a Gromov-hyperbolic metric space when endowed with the word-metric \eqref{e.word_metric}; this does not depend on the choice of the generating set $S$; see \cite{gromov,CDP,GhysdelaHarpe,BH}.

\medskip

A representation $\rho\colon \Gamma \to \GL(d,\R)$ is $p$-\emph{dominated} if there exist constants $C$, $\lambda >0$ such that:
\begin{equation}\label{eq:domrep}
\frac{\sigma_{p+1}(\rho(\gamma))} {\sigma_{p}(\rho(\gamma))} \leq C e^{-\lambda |\gamma|} \  \quad \text{for all } \gamma \in \Gamma.
\end{equation}
It is easy to see that the definition does not depend on the choice of $S$ though the constants $C$ and $\lambda$ may change.

\begin{obs}\label{rem-pdominatedimpliesd-p}
Since $\sigma_i(A) =  \sigma_{d+1-i}(A^{-1})^{-1}$ for each $i$, the property $|\gamma^{-1}| = |\gamma|$ implies that if a representation $\rho$ is $p$-dominated, then it is also $(d-p)$-dominated. 
\end{obs}

The purpose of this section is to show: 

\begin{teo}\label{teo:dominationimplieshyp} If a group $\Gamma$ admits a $p$-dominated representation into $\GL(d,\R)$ then $\Gamma$ is word-hyperbolic.
\end{teo}

\medskip

Theorem~\ref{teo:dominationimplieshyp} follows from a more general result recently obtained by Kapovich, Leeb, and Porti \cite[Theorem~1.4]{KLP2}. 
Their result concerns not only dominated representations, but quasi-isometric embeddings of metric spaces satisfying a condition related to domination.
Here we use the results from Section~\ref{s.DS} to give a direct and more elementary proof.

Let us mention that a related but different notion of domination was recently studied by other authors \cite{DT,GKW}.

\begin{obs}[Representations to $\SL(d,\R)$]\label{rem-det1} Given a representation with target group $\GL(d,\R)$, we can always assume that it has its image contained in matrices with determinant $\pm 1$ by composing with the homomorphism $A \mapsto |\det A|^{-1/d} A$, which does not affect $p$-domination.  
\end{obs}

\begin{obs}[Representations to $\PGL(d,\R)$]\label{rem-pgl}
We can define $p$-dominated representations on $\PGL(d,\R)$ in exactly the same way, since the quotients $\sigma_{p+1}/\sigma_p$ are well-defined in the latter group. Obviously, composing any $p$-dominated representation on $\GL(d,\R)$ with the quotient map $\pi \colon \GL(d,\R) \to \PGL(d,\R)$ we obtain a $p$-dominated representation.
Conversely, given any $p$-dominated representation $\rho \colon \Gamma \to \PGL(d,\R)$, we can find a group $\hat{\Gamma}$, a $2$-to-$1$ homomorphism $f \colon \hat{\Gamma} \to \Gamma$, and a $p$-dominated representation $\hat{\rho} \colon \hat{\Gamma} \to \GL(d,\R)$ (with determinants $\pm 1$) such that $\pi \circ \hat{\rho}  =  f \circ \rho$. Theorem~\ref{teo:dominationimplieshyp} yields that $\hat\Gamma$ is word-hyperbolic, and it follows (see \cite[p.~63]{GhysdelaHarpe}) that $\Gamma$ is word-hyperbolic as well.
\end{obs}

\begin{obs}[General semi-simple Lie groups]\label{rem-lie}
Using exterior powers, any $p$-domi\-na\-ted representation on $\GL(d,\R)$ induces a $1$-dominated representation in $\GL(k,\R)$ for $k  =  \binom{d}{p}$. In Subsection~\ref{generalcase?} we shall see that every representation $\rho\colon \Gamma \to G$, where $G$ is an arbitrary (real-algebraic non-compact) semi-simple Lie group, can be reduced to the case of $\PGL(d,\R)$ for some $d$ via a similar construction. 
\end{obs}

\begin{obs}
Given Theorem~\ref{teo:dominationimplieshyp}, one may wonder whether every hyperbolic group admits a dominated representation. This is far from true, since there exist hyperbolic groups that are nonlinear, i.e.\ every linear representation of these groups factors through a finite group: see \cite[Section~8]{nonlinear}. On the other hand, one can ask if every linear hyperbolic group admits a dominated representation.
\end{obs}

%%%%%%%%%%%%%%%%%%%%%%%%%%%%%%%%%%%%%%

\subsection{Criterion for hyperbolicity}\label{ss.crithyp}

In order to show word-hyperbolicity the following sufficient condition will be used: 

\begin{teo}[Bowditch \cite{Bowditch}]\label{teo-bow} Let $\Gamma$ be a group which acts by homeomorphisms in a perfect compact metrizable topological space $M$ such that the diagonal action of $\Gamma$ on the (nonempty) space  
$$
M^{(3)} \coloneqq \big\{(x_1,x_2,x_3) \in M^3 \st x_i \neq x_j \text{ if } \ i\neq j \big \}
$$ 
is properly discontinuous and cocompact. 
Then $\Gamma$ is word-hyperbolic.
\end{teo}

We recall that a continuous action of $\Gamma$ in a topological space $X$ is:
\begin{itemize}
\item \emph{properly discontinuous} if given any compact subset $K \en X$ there exists $n$ such that if $|\gamma|>n$ then $\gamma K \cap K  =  \emptyset$;
\item \emph{cocompact} if there exists a compact subset $K \en X$ such that $\Gamma x \cap K \neq \emptyset$ for every $x \in X$.
\end{itemize}

\begin{obs}\label{rem-boundary}
Theorem \ref{teo-bow} also gives that the set $M$ is equivariantly homeomorphic to $\partial \Gamma$. Here $\partial \Gamma$ denotes the \emph{visual boundary} of the group $\Gamma$, defined as the set of equivalence classes of quasi-geodesic rays (i.e.\ quasi-geodesic maps from $\N$ to $\Gamma$) by the equivalence of being at finite Hausdorff distance from each other; see for example \cite[Chapitre 7]{GhysdelaHarpe} or \cite[Chapitre 2]{CDP}). The topology in $\partial \Gamma$ is given by pointwise convergence of the quasi-geodesic rays with same constants and starting at the same point. In Section \ref{s.analytic} we will also comment on the metric structure on the boundary. \end{obs} 

A group is called \emph{elementary} if it is finite or virtually cyclic; elementary groups are trivially word-hyperbolic. 

The converse of Theorem \ref{teo-bow} applies to non-elementary word-hyperbolic groups. In the proof of Theorem~\ref{teo:dominationimplieshyp} we must separate the case where the group is elementary, since elementary groups may admit dominated representations while Theorem~\ref{teo-bow} does not apply to them.

%%%%%%%%%%%%%%%%%%%%%%%%%%%%%%%

\subsection{Some preliminary lemmas for $p$-dominated representations} 

The \emph{Grassmannian} $\Gr_p(\R^d)$ is the set of all $p$-dimensional subspaces of $\R^d$.
As we explain in detail in the Appendix~\ref{s.appendix}, the following formula defines a metric on the Grassmannian:
$$
d(P,Q) \coloneqq \cos \angle (P^\perp,Q) \, ;
$$
here $\perp$ denotes orthogonal complement, and $\angle$ denotes the smallest angle 
between pairs of vectors in the respective spaces.
Appendix~\ref{s.appendix} also contains a number of quantitative linear-algebraic estimates that we use in this section.  

%%%%%%%%%%%%%%%%%%%%%%%%%%%%%%%%%%%%
In particular, in Appendix \ref{s.appendix} there are precise statements and proofs of the following results we will use: 

\begin{itemize}
\item Lemmas~\ref{l.nochangeright} and \ref{l.domination_implies_slow_change} estimate the distance of $U_p(A)$ with $U_p(AB)$ and $B U_p(A)$ with $U_p(BA)$ with respect to the norms of $B^{\pm 1}$ and the gap on the singular values of $A$. Lemma~\ref{l.seq_convergence} is a reinterpretation of these facts in terms of dominated sequences of matrices. 

\item Lemma~\ref{l.dominationattractor} shows that if there is a large gap between $p$-th and $p+1$-the singular values of $A$ then $AP$ will be close to $U_p(A)$ for every $P$ which makes a given angle with $S_{d-p}(A)$. 

\item Lemma~\ref{l.expand} estimates the dilatation of the action of $A$ in $\Gr_p(\RR^d)$ for subspaces far from the subspace which is mapped to $U_p(A)$.  

\item Corollary~\ref{c.seq_complement} is a consequence of classical properties of dominated splittings in the context of sequences of matrices in $\cD(K,p,\mu,c,\N)$.

\end{itemize}
%%%%%%%%%%%%%%%%%%%%%%%%%%%%%%%%%%%

Throughout the rest of this section, let $\rho \colon \Gamma \to \GL(d,\R)$ a $p$-dominated representation with constants $C \ge 1$, $\lambda>0$ (c.f.\ relation~\eqref{eq:domrep}).    
Let:
\begin{equation}\label{e.K}
K \coloneqq \max_{g \in S} \|\rho(g)\| \geq 1 \, ,
\end{equation}
where $S$ is the finite symmetric generating set of $\Gamma$ fixed before.
Also fix $\ell_0 \in \N$ such that: 
\begin{equation}\label{e.ell0}
C e^{-\lambda \ell_0} < 1 \, , 
\end{equation}
and, in particular (recalling the notation introduced in \S~\ref{ss.dom_sing}),
the spaces $U_p(\rho(\gamma)) \in \Gr_p(\R^d)$ and $S_{d-p}(\rho(\gamma)) \in \Gr_{d-p}(\R^d)$ 
are well-defined whenever $|\gamma|\ge \ell_0$. 

\medskip

Suppose that $\gamma$, $\eta$ are large elements in $\Gamma$ such that the spaces $U_p(\rho(\gamma))$ and $U_p(\rho(\eta))$ are not too close; then the next two  lemmas respectively assert that $d(\gamma,\eta)$ is comparable to $|\gamma|+|\eta|$, and that  $U_p(\rho(\gamma))$ and $S_{d-p}(\rho(\eta^{-1}))$ are transverse.

\begin{lem}\label{l.desechable}
There exist constants $\nu \in (0,1)$, $c_0>0$, $c_1>0$ with the following properties.
If $\gamma$, $\eta \in \Gamma$ are such that
$|\gamma|$, $|\eta| \ge \ell_0$ (where $\ell_0$ is as in \eqref{e.ell0})
then:
$$
d(\gamma,\eta) \ge \nu (|\gamma|+|\eta|) - c_0 - c_1 \left| \log d\big(U_p(\rho(\gamma)), U_p(\rho(\eta))\big) \right| \, .
$$
\end{lem}

\begin{proof}
Consider two elements $\gamma$, $\eta \in \Gamma$ with word-length at least $\ell_0$.
%and write $\delta \coloneqq d \big( U_p(\rho(\eta)), U_p(\rho(\gamma)) \big)$.
Assume $|\gamma| \le |\eta|$, the other case being analogous.
Applying Lemma~\ref{l.nochangeright} to $A = \rho(\eta)$ and $B =  \rho(\eta^{-1}\gamma)$, 
and using \eqref{eq:domrep} and \eqref{e.K}, we obtain: 
\begin{equation}\label{e.useful}
d \big( U_p(\rho(\eta)), U_p(\rho(\gamma)) \big)
\leq K^{2 |\eta^{-1} \gamma|}  \, C e^{-\lambda |\eta|} \, ,
\end{equation}
or equivalently,
$$
d(\gamma,\eta)  =  |\eta^{-1} \gamma| \ge \frac{\lambda |\eta| - \log C - \left| \log d \big( U_p(\rho(\eta)), U_p(\rho(\gamma)) \big)\right|}{2 \log K} \, .
$$
Using that $|\eta | \ge (|\gamma|+|\eta|)/2$, we obtain the lemma. 
\end{proof}

\begin{lem}\label{l.transv} 
For every $\eps>0$ there exist $\ell_1 \ge \ell_0$  
and $\delta \in (0, \frac{\pi}{2})$ with the following properties:
If $\gamma$, $\eta \in \Gamma$ are such that:
\begin{enumerate}
\item 
$|\gamma|$, $|\eta| > \ell_1$, and
\item\label{i.eps_dist} 
the distance $d \big( U_p(\rho(\gamma)), U_p(\rho(\eta)) \big) >\eps$,   
\end{enumerate}
then it follows that $\angle \big( U_p(\rho(\gamma)), S_{d-p}(\rho(\eta^{-1})) \big) > \delta$. 
\end{lem}

\begin{proof} 
Consider two elements $\gamma$, $\eta \in \Gamma$ with word-length at least $\ell_0$.
Let $\epsilon>0$, and suppose the hypothesis (\ref{i.eps_dist}) is satisfied.
Write $\gamma  =  g_1 \cdots g_n$ with each $g_i$ in the symmetric generating set of $\Gamma$ in such a way that $n$ is minimal, that is, $n = |\gamma|$. Similarly, write $\eta  =  h_1 \cdots h_m$ with each $h_i$ in the symmetric generating set of $\Gamma$ in such a way that  $|\eta| = m$.

Let $\gamma_i \coloneqq g_1 \cdots g_i$ and $\eta_i \coloneqq h_1 \cdots h_i$. Note that for $j>i$ we have $d(\gamma_i,\gamma_j) = |\gamma_i^{-1}\gamma_j|  =  |g_{i+1} \cdots g_j|  =  j-i$ so that the sequence $\{\gamma_i\}$ is a geodesic from $\id$ to $\gamma$. The same holds for $\{\eta_i\}$ (but note that $\{\eta_i^{-1}\}$ is not necessarily a geodesic). 

By the domination condition \eqref{eq:domrep} and Lemma~\ref{l.seq_convergence},
we can find $\ell_*  =  \ell_*(\eps) > \ell_0$
such that if $n \ge i>\ell_*$ then 
$$
d(U_p(\rho(\gamma_i)) , U_p(\rho(\gamma))) < \eps/3 \, ,
$$
and analogously,  if $m \ge j>\ell_*$ then 
$$
d(U_p(\rho(\eta_j)) , U_p(\rho(\eta))) < \eps/3.
$$
In particular, if both conditions hold then hypothesis \eqref{i.eps_dist} implies that
\begin{equation}\label{e.eps3}
d(U_p(\rho(\gamma_i)) , U_p(\rho(\eta_j)) > \eps/3.
\end{equation}

Let $\nu \in (0,1)$, $c_0$, $c_1$ be given by Lemma~\ref{l.desechable}.
Let $c \coloneqq \max \{2\ell_*, c_0 + c_1 \log(3/\epsilon) \}$.
We claim that 
\begin{equation}\label{e.QG}
d(\gamma_i, \eta_j)  =  |\eta_j^{-1} \gamma_i| \ge \nu (i+j) - c
\quad \text{for all $i\in \{1,\dots,n\}$, $j\in \{1,\dots,m\}$.} 
\end{equation}
To prove this, consider first the case when $j \le \ell_*$.
Then
$$
d(\gamma_i,\eta_j) \ge d(\gamma_i,\id) - d(\eta_j,\id)  =  i-j \ge i+j - 2\ell_* \ge \nu (i+j) - c \, ,
$$
as claimed.
The case $i \le \ell_*$ is dealt with analogously.
The remaining case where $i$ and $j$ are both bigger than $\ell_*$
follows from Lemma~\ref{l.desechable} and property \eqref{e.eps3}.
This proves \eqref{e.QG}.

As a consequence of \eqref{eq:domrep} and \eqref{e.QG}, we obtain:
\begin{equation}\label{e.almostthere}
\frac{\sigma_{p+1}}{\sigma_p}(\eta_j^{-1} \gamma_i) \le \hat{C} e^{-\mu (i+j)} 
\quad \text{for all $i\in \{1,\dots,n\}$, $j\in \{1,\dots,m\}$,} 
\end{equation}
where $\hat{C} \coloneqq C e^{\lambda c}$ and $\mu \coloneqq \lambda \nu$.
Now consider the following sequence of matrices:
\begin{equation*} %\label{e.concat}
\big( A_{-n}, \dots, A_{-1}, A_0, A_1, \dots, A_{m-1} \big) \coloneqq  
\big( \rho(g_n), \dots, \rho(g_1), \rho(h_1^{-1}), \rho(h_2^{-1}), \dots, \rho(h_m^{-1}) \big) \, .
\end{equation*}
So
$A_{-1} A_{-2} \cdots A_{-i}  =  \rho(\gamma_i)$ and $A_{j-1} A_{j-2} \cdots A_0  =  \rho(\eta_j^{-1})$.
It follows from \eqref{eq:domrep} and \eqref{e.almostthere}, together with the facts that $\hat{C}>C$ and $\mu<\lambda$,
that the sequence of matrices defined above belongs to the set 
$\cD(K,p,\mu,\hat{C},I)$, where $I \coloneqq \{-n,\dots,m-1\}$
and $K$ is defined by equation~\eqref{e.K}.
Now let $\ell_1$ and $\delta$ be given by Lemma~\ref{l.seq_splitting}.
Then
$$
\angle \big(U_p(A_{-1} \cdots A_{-n}), S_{d-p}(A_{m-1}\cdots A_0) \big) > \delta \, ,
$$
provided that $|\gamma| = n$ and $|\eta| = m$ are both bigger than $\ell_1$.
This concludes the proof.
\end{proof}

%%%%%%%%%%%%%%%%%%%%%%%%%%%%%%%%%%%%%%%%%%%%
\subsection{Candidate for boundary of \texorpdfstring{$\Gamma$}{Gamma}} 

In this subsection we define a candidate for the role of $M$ in Theorem~\ref{teo-bow} using the domination of the representation and show some of the topological properties of $M$ required for applying Theorem~\ref{teo-bow}. 
The set we shall consider is the following: 
\begin{equation}\label{eq-conjuntoM}
M \coloneqq \bigcap_{n \ge \ell_0} \overline{ \{ U_{p}(\rho(\gamma)) \st \ |\gamma|\geq n \} } \en \Gr_p(\RR^d)  \, ,
\end{equation}
where $\ell_0$ is as in \eqref{e.ell0}. This set has been considered before and named \emph{limit set} by Benoist (see \cite[Section 6]{BenoistNotes}) in the Zariski dense context and extended in \cite[Definition 5.1]{GGKW} to a more general setting.  

\medskip

The first properties to be established about $M$ are the following ones:

\begin{prop}\label{p.PropsSetM} The set $M$ is compact, non-empty, and $\rho(\Gamma)$-invariant. 
\end{prop}

\begin{proof} The fact that $M$ is compact and non-empty follows at once since it is a decreasing intersection of non-empty closed subsets of the compact space $\Gr_p(\RR^d)$.

Let us show that $M$ is $\rho(\Gamma)$-invariant. Fix $\eta \in \Gamma$ and  $P \in M$.
Choose a sequence $(\gamma_n)$ in $\Gamma$ such that $|\gamma_n| \to \infty$ and  $U_p(\rho(\gamma_n)) \to P$.
Note that 
the spaces $U_p(\rho(\eta\gamma_n))$ are defined
for large enough $n$ (namely, whenever $|\gamma_n| \ge \ell_0 + |\eta|$); 
moreover, by Lemma~\ref{l.domination_implies_slow_change} and the domination condition~\eqref{eq:domrep}, we have:
$$ 
d\big(\rho(\eta) U_p(\rho(\gamma_n)), U_p(\rho(\eta\gamma_n))\big) \le
\|\rho(\eta)\| \, \|\rho(\eta)^{-1}\| \, C \, e^{-\lambda|\gamma_n|} \to 0 \quad \text{as }
n\to \infty.
$$
This shows that $U_p(\rho(\eta\gamma_n)) \to \rho(\eta) P$ as $n\to \infty$,
and in particular $\rho(\eta) P \in M$,
as we wished to prove.
\end{proof}

If one manages to show that $M^{(3)}$ is non-empty then
one of the assumptions of Theorem~\ref{teo-bow} is satisfied:

\begin{lem}\label{l.3pointsperfect}
If the set $M$ has at least $3$~points then it is perfect. 
\end{lem}

We first show the following lemma:

\begin{lem}\label{l.diciembre}
Given $\epsilon>0$, $\epsilon'>0$, there exists $\ell > \ell_0$ (where $\ell_0$ is as in \eqref{e.ell0}) with the following properties: 
If $\eta \in \Gamma$ is such that $|\eta|>\ell$ 
and $P \in M$ is such that 
$$
d \big( P, \, U_p(\rho(\eta^{-1})) \big) > \epsilon \, ,
$$
then 
$$
d \big( \rho(\eta) P, \, U_p(\rho(\eta)) \big) < \epsilon' \, .
$$
\end{lem}

\begin{proof}
Let $\ell_1 \ge \ell_0$ and $\delta>0$ be given by Lemma~\ref{l.transv}, depending on $\epsilon$.
Let $\ell>\ell_1$ be such that
$C e^{-\lambda \ell} < (\epsilon' \sin \delta)/2$,
where $C$ and $\lambda$ are the domination constants, as in \eqref{eq:domrep}.
Now fix $\eta \in \Gamma$  and $P \in M$
such that $|\eta|>\ell$ and $d \big( P, U_p(\rho(\eta^{-1})) \big) > \epsilon$.
Choose a sequence $(\gamma_n)$ in $\Gamma$ such that $|\gamma_n| \to \infty$ and  $U_p(\rho(\gamma_n)) \to P$.
Without loss of generality, we can assume that for each $n$ we have $|\gamma_n|> \ell_1$ and 
$$
d \big( U_p(\rho(\gamma_n)) , U_p(\rho(\eta^{-1})) \big) > \epsilon.
$$
It follows from Lemma~\ref{l.transv} that
$$
\angle \big( U_p(\rho(\gamma_n)) , S_{d-p}(\rho(\eta)) \big) > \delta.
$$
Using Lemma~\ref{l.dominationattractor} and the domination condition \eqref{eq:domrep}, we obtain:
$$
d \Big( \rho(\eta) \big( U_p(\rho(\gamma_n)) \big) \, , U_p(\rho(\eta)) \Big)
< \tfrac{\sigma_{p+1}}{\sigma_p} \big( \rho(\eta) \big) \  \frac{1}{\sin \delta} < \frac{\epsilon'}{2} \, .
$$
Letting $n \to \infty$ yields
$d \big( \rho(\eta) P \, , U_p(\rho(\eta)) \big) \le \epsilon'/2$.
\end{proof}

\begin{proof}[Proof of Lemma~\ref{l.3pointsperfect}]
Let $P_1$, $P_2$, $P_3$ be three distinct points in $M$, and let $\epsilon' > 0$.
We will show that the $2\epsilon'$-neighborhood of $P_1$ contains another element of $M$.

Let $\eps \coloneqq \frac 1 2 \min_{i\neq j} d(P_i, P_j)$.
Let $\ell > \ell_0$ be given by Lemma~\ref{l.diciembre}, depending on $\eps$ and $\eps'$.
Choose $\eta \in \Gamma$ such that $|\eta|>\ell$ and $d\big( U_p(\rho(\eta)) , P_1 \big) < \epsilon'$.
Consider the space $U_p(\rho(\eta^{-1}))$; it can be $\epsilon$-close to at most one of the spaces $P_1$, $P_2$, $P_3$. 
In other words, there are different indices $i_1$, $i_2 \in \{1,2,3\}$ such that 
for each $j \in \{1,2\}$ we have:
$$
d \big( P_{i_j} \, , U_p(\rho(\eta^{-1})) \big) > \epsilon 
$$
In particular, by Lemma~\ref{l.diciembre},
$$
d \big( \rho(\eta) P_{i_j}, P_1 \big) 
\le d \big( \rho(\eta) P_{i_j}, U_p(\rho(\eta)) \big) + \epsilon'
< 2\epsilon' \, .
$$
By Proposition~\ref{p.PropsSetM}, the spaces $\rho(\eta) P_{i_1}$ and $\rho(\eta) P_{i_2}$ belong to $M$.
Since at most one of these can be equal to $P_1$, we conclude that the $2\epsilon'$-neighborhood of $P_1$ contains another element of $M$, as we wished to prove.
\end{proof}

So we would like to ensure that $M$ has at least $3$ points (provided $\Gamma$ is non-elementary). 
This requires some preliminaries. 
First, we show the following estimate:
\begin{lem}\label{l.2}
$M$ has at least $2$ elements.
\end{lem}

\begin{proof}
Using Lemma~\ref{l.seq_splitting} 
we can find $\ell_1 \in \N$ and $\delta>0$ such that if $\gamma$, $\eta \in \Gamma$ have lengths both bigger than $\ell_1$, and satisfy $|\eta \gamma|  =  |\eta| + |\gamma|$,
then 
\begin{equation}\label{e.Us_apart}
\angle \big( U_p(\rho(\gamma)) , S_{d-p}(\rho(\eta))  \big) > \delta.
\end{equation}

We already know that $M$ is nonempty, so assume by contradiction that it is a singleton $\{P\}$.
Take $n>\ell_1$ such that if an element $\gamma \in \Gamma$ has length $|\gamma|>n$ then $d \big( U_p(\rho(\gamma)) , P \big) < \frac{1}{2} \sin \delta$.
Now fix arbitrary elements $\gamma$, $\eta \in \Gamma$ with lengths both bigger than $\ell_1$ such that $|\eta \gamma|  =  |\eta| + |\gamma|$.
Then $d \big( U_p(\rho(\gamma)) , U_p(\rho(\eta^{-1})) \big) < \sin \delta$. The space $U_p(\rho(\eta^{-1}))$ either contains or is contained in $U_{d-p}(\rho(\eta^{-1}))  =  S_{d-p}(\rho(\eta))$.
Using the trivial bound \eqref{e.stupid_bound2}, we contradict \eqref{e.Us_apart}.
\end{proof}

Recalling Remark~\ref{rem-det1}, let us assume for convenience for the remainder of this subsection that 
$$
|\det\rho(\gamma)| = 1 \quad \text{for every } \gamma \in \Gamma \, .
$$
Let $\jac( \mathord{\cdot})$ denote the \emph{jacobian} of a linear map, i.e., the product of its singular values. 
The next step is to prove the expansion property expressed by the following lemma:

\begin{lem}\label{l.somebody_expands} 
There exists a constant $\ell \in \N$ such that for every $\gamma \in \Gamma$ with length $|\gamma| \ge \ell$ there exists $P \in M$ such that $\jac(\rho(\gamma)|_P) \ge 2$.
\end{lem}

\begin{proof}
As a rephrasing of Lemma~\ref{l.2}, there exists $\epsilon > 0$ such that for each $Q \in M$ there exists $P \in M$ such that $d(P,Q) > 2 \epsilon$.
Let $\ell_1 \ge \ell_0$ and $\delta \in (0, \frac{\pi}{2})$ be given by Lemma~\ref{l.transv}, depending on $\epsilon$.
Let 
$$
\delta' \coloneqq \min \tfrac{1}{2} \big\{ \epsilon, 1 - \cos \delta \big\} \, .
$$
Let $\ell_2 \ge \ell_1$ be such that the sets $M$ and 
$\big \{ U_p(\rho(\gamma)) \st \gamma \in \Gamma, \ |\gamma|\ge \ell \big\}$
are each contained in the $\delta'$-neighborhood of the other.
Fix a large $\ell \ge \ell_2$; how large it needs to be will be clear at the end.

Now fix $\gamma \in \Gamma$ with length $|\gamma| \ge \ell$.
Take $Q \in M$ such that $d\big( Q, U_p(\rho(\gamma^{-1})) \big) < \delta'$.
Let $P \in M$ be such that $d(P,Q) > 2 \epsilon$.
Let $\eta \in \Gamma$ be such that $|\eta| \ge \ell$ and $d\big(P, U_p(\rho(\eta)) \big) < \delta'$.
Then 
\begin{align*}
d\big( U_p(\rho(\eta)), U_p(\rho(\gamma^{-1})) \big) 
&\ge d(P,Q) - d\big( U_p(\rho(\eta)), P \big) - d\big( Q, U_p(\rho(\gamma^{-1})) \big) \\
&> 2 \epsilon - 2 \delta'  \\
&> \epsilon \, . 
\end{align*}
Therefore Lemma~\ref{l.transv} guarantees that
$$
\angle \big( U_p(\rho(\eta)) , S_{d-p}(\rho(\gamma)) \big) > \delta.
$$
So, using identity \eqref{e.distGr4},
\begin{align*}
\cos \angle \big( P, S_{d-p}(\rho(\gamma)) \big)
& =    d\big( P, S_{d-p}(\rho(\gamma))^\perp \big) \\
&\le d\big( P, U_p(\rho(\eta)) \big) + d\big( U_p(\rho(\eta)) , S_{d-p}(\rho(\gamma))^\perp \big) \\
& =  d\big( P, U_p(\rho(\eta)) \big) + \cos \angle \big( U_p(\rho(\eta)) , S_{d-p}(\rho(\gamma)) \big) \\
&\le \delta' + \cos \delta  \\
&< 1 \, .
\end{align*}
Write $A \coloneq \rho(\gamma)$ for simplicity.
We must estimate the jacobian of $A|_P$.
We will use some facts about exterior powers: see Subsection~\ref{ss.ext}.
Let $\iota(P) \in \grass_1(\Wedge^p \R^d)$ be the image of $P$ under the Pl\"ucker embedding, and take a nonzero $w \in \iota(P)$.
Then:
\begin{alignat*}{2}
\jac(A|_P) 
&=   \frac{\|(\Wedge^p A) w\|}{\|w\|}  &\quad &\text{(by \eqref{e.jac_norm})}    \\
&\ge \sigma_1(\Wedge^p A) \sin \angle \Big( w , S_{\binom{d}{p}-1}(\Wedge^p A) \Big) &\quad &\text{(by \eqref{e.pit1})} \\
&=   \sigma_1(A) \cdots \sigma_p(A) \sin \angle \Big( \iota(P) , \big( \iota(S_{d-p}(A)^\perp) \big)^\perp  \Big) &\quad &\text{(by \eqref{e.ext_sigma1} and \eqref{e.ext_S})} \\ 
&\ge  \sigma_1(A) \cdots \sigma_p(A) \big[ \sin \angle ( P , S_{d-p}(A) \big]^{\min\{p,d-p\}} &\quad &\text{(by \eqref{e.plucker_ang})}
\end{alignat*} 
The sine can be bounded from below by a positive constant.
On the other hand, since the product of the singular values of $A$ is $|\det A|  =  1$, we have:
$$
\sigma_1(A) \cdots \sigma_p(A)
 =  \frac{\big[\sigma_1(A) \cdots \sigma_p(A) \big]^{\frac{d-p}{d}}}{\big[\sigma_{p+1}(A) \cdots \sigma_d(A) \big]^{\frac{p}{d}}}
\ge \left[ \frac{\sigma_p(A)}{\sigma_{p+1}(A)} \right]^{\frac{p(d-p)}{d}}
\, ,
$$
which is exponentially large with respect to $|\gamma| \ge \ell$.
We conclude that $\jac(A|_P) \ge 2$ if $\ell$ is large enough. 
\end{proof}

As a last digression, we show that virtually abelian groups cannot have a dominated representation unless they are virtually cyclic. 

\begin{lem}\label{l.Z2notdominated}
Let $\rho \colon \Gamma \to \GL(d,\R)$ be a $p$-dominated representation.
Let $\Gamma'$ be a finite-index subgroup of $\Gamma$, and let $m\ge 2$.
Then there exists no surjective homomorphism $\varphi \colon \Gamma' \to \Z^m$ with finite kernel.
\end{lem}

\begin{proof}  
Assume for a contradiction that $\Gamma$ contains a subgroup $\Gamma'$ which admits a homomorphism $\varphi$ onto $\Z^m$ with finite kernel, where $m \ge 2$. 
Let $Z$ be the standard symmetric generating set for $\Z^m$, with cardinality $2m$; then $\varphi^{-1}(Z)$ is a finite symmetric generating set for $\Gamma'$. Since $\Gamma'$ is a finite-index subgroup of $\Gamma$, the inclusion $\Gamma' \hookrightarrow \Gamma$ is a quasi-isometry (using e.g.\ the Svarc--Milnor lemma \cite[Proposition I.8.19]{BH}), and therefore the restriction of $\rho$ to $\Gamma'$ is $p$-dominated.
For simplicity of notation, we assume that $\Gamma  =  \Gamma'$, with generating set $S  =  \varphi^{-1}(Z)$. Fix the constants $C$, $\lambda$, $K$, and $\ell_0$ satisfying \eqref{eq:domrep}, \eqref{e.K}, and \eqref{e.ell0}.

Fix $g \in \Gamma$ such that $\varphi(g)$ is an element of $Z$, say $(1,0,\dots,0)$.
Since the representation is $p$-dominated and $(g^n)_{n \in \Z}$ is a geodesic in $\Gamma$,
it follows from Lemma~\ref{l.seq_splitting} that there exist $\ell_1 > \ell_0$ and $\delta>0$ such that 
$$
\angle \big( U_p(\rho(g^n)) , S_{d-p}(\rho(g^n)) \big) > \delta \quad \text{for all $n > \ell_1$.}
$$
The space $S_{d-p}(\rho(g^n))  =  U_{d-p}(\rho(g^{-n}))$ either contains or is contained in $U_p(\rho(g^{-n}))$; so the trivial bound \eqref{e.stupid_bound2} yields: \begin{equation}\label{e.domination_will_tear_us_apart}
d \big( U_p(\rho(g^n)) , U_p(\rho(g^{-n})) \big) \ge \sin \delta 
\quad \text{for all $n > \ell_1$.}
\end{equation}

Fix $n > \ell_1$.
Since $m \ge 2$, we can find $\gamma_0$, $\gamma_1$, \dots, $\gamma_{4n} \in \Gamma$ with the following properties:
$$
\gamma_0  =  g^n, \quad \gamma_{4n}  =  g^{-n}, \quad
|\gamma_{i+1}^{-1}\gamma_i|  =  1,\quad
|\gamma_i| \ge n.
$$
Indeed, we can take a preimage under $\varphi$ of an appropriate path in $\Z^m$, sketched in Fig.~\ref{f.path} for the case $m = 2$.

\begin{figure}[htb]
	\begin{tikzpicture}[scale = .3,font = \small]
		\draw[->] (-6,0)--(6,0);
		\draw[->] (0,-2)--(0,6);
		\draw[very thick,-stealth] (4,0) node[below]{$(n,0)$} --(4,4)--(-4,4)--(-4,0) node[below]{$(-n,0)$};
		%\draw[very thick,-stealth] (4,0) node[below]{$(n,0)$} -- (4,2);
		%\draw[very thick,-stealth] (4,1) -- (4,4) -- (2,4);
		%\draw[very thick,-stealth] (3,4) -- (-2,4);
		%\draw[very thick,-stealth] (-1,4) -- (-4,4) -- (-4,2);
		%\draw[very thick] (-4,3) -- (-4,0) node[below]{$(-n,0)$};
	\end{tikzpicture}
	\caption{A path in $\Z^2$.}\label{f.path}
\end{figure}

Now we estimate:
\begin{alignat*}{2}
d \big( U_p(\rho(g^n)) , U_p(\rho(g^{-n})) \big) 
&\le \sum_{i = 0}^{4n-1} d \big( U_p(\rho(\gamma_i)) , U_p(\rho(\gamma_{i+1})) \big)				\\
&\le \sum_{i = 0}^{4n-1} K^{2 |\gamma_{i+1}^{-1}\gamma_i|} \ C \, e^{-\lambda|\gamma_i|}	&\quad\text{(by estimate \eqref{e.useful})}\\
&\le 4n K^2 C e^{-\lambda n} \, .
\end{alignat*}
Taking $n$ large enough, we contradict \eqref{e.domination_will_tear_us_apart}.
This proves the lemma.
\end{proof}

Now we are ready to obtain the topological property of $M$ required by Theorem~\ref{teo-bow}.
(Later we will have to check the hypotheses about the action on $M^{(3)}$.)
Recall that $\Gamma$ is elementary if it is finite or virtually cyclic (i.e.\ virtually $\Z$).

\begin{prop}\label{p.perfect_again} 
If $\Gamma$ is non-elementary then the set $M$ is perfect. 
\end{prop}

\begin{proof} By Lemma \ref{l.3pointsperfect}, it is enough to show that $M$ is infinite. Recall that $M$ is non-empty. We assume by contradiction that $M$ is finite, say $M = \{P_1, \ldots, P_k \}$. 

By Proposition~\ref{p.PropsSetM}, the set $M$ is $\rho(\Gamma)$-invariant. Consider the set $\Gamma' \en \Gamma$ of those elements such that $\rho(\gamma)P_i  =  P_i$ for all $i$. Then $\Gamma'$ is a finite-index subgroup of $\Gamma$. 
Consider the map $\varphi \colon \Gamma' \to \R^k$ defined by:
$$ \gamma \in \Gamma' \mapsto ( \log \jac( \rho(\gamma)|_{P_i})  )_{i = 1, \ldots, k} \ \ . $$
Invariance of the subspaces $P_i$ implies that the jacobian is multiplicative and so $\varphi$ is a homomorphism.
By Lemma~\ref{l.somebody_expands}, the kernel of $\varphi$ is finite. On the other hand,
the image of $\varphi$ is a subgroup of $\RR^k$ and therefore is abelian, without torsion, and finitely generated (since so is $\Gamma'$); therefore this image is isomorphic to some $\ZZ^m$.
Lemma~\ref{l.Z2notdominated} now yields that $m = 1$, and it follows that $\Gamma$ is elementary.
This contradiction concludes the proof.
\end{proof}

%%%%%%%%%%%%%%%%%%%%%%%%%%%%%%%%%%%%%%%%%%%%%%%%%%%%%%%%%%%%%%%%%%%%%%%%%%%%%%%%%%%%%%%%%%%%%%%%%%%%%%%%%%%

\subsection{Proper discontinuity}

Given a triple $T  =  (P_1,P_2,P_3) \in M^{(3)}$, let us denote
\begin{equation}\label{e.triple}
|T| \coloneqq \min_{i\neq j} d( P_i, P_j) \, .
\end{equation}
Note that for any $\delta>0$, the set 
$\big\{T \in M^{(3)} \st |T| \ge \delta \big\}$
is a compact subset of $M^{(3)}$;
conversely, every compact subset of $M^{(3)}$ is contained in a subset of that form.

\begin{prop}\label{p.propdisc}
For every $\delta>0$ there exists $\ell \in \N$ such that if $T \in M^{(3)}$ satisfies $|T|>\delta$ and $\eta\in \Gamma$ satisfies $|\eta|> \ell$, then we have $|\rho(\eta)T| < \delta$.
\end{prop}

\begin{proof}
Given $\delta>0$, let $\ell$ be given by Lemma~\ref{l.diciembre} with 
$\epsilon  =  \epsilon'  =  \delta/2$.
Now consider $(P_1,P_2,P_3) \in M^{(3)}$ such that $|T| > \delta$ and $\eta \in \Gamma$ such that $|\eta| > \ell$.
Note that $d(U_p(\eta^{-1}),P_i) > \delta/2$ for at least two of the spaces $P_1$, $P_2$, $P_3$ -- say, $P_1$ and $P_2$.
Lemma~\ref{l.diciembre} yields 
$d \big( \rho(\eta) P_i, U_p(\rho(\eta)) \big) < \delta/2$ for each $i = 1,2$.
In particular, $d \big( \rho(\eta) P_1, \rho(\eta) P_2 \big) < \delta$ and 
so $|\rho(\eta)T|<\delta$, as we wanted to show.
\end{proof}

%%%%%%%%%%%%%%%%%%%%%%%%%%%%%%%%%%%%%%%%%%%%%%%%%%%%%%%%%%%%%%%%%%%%%%%%%%%
\subsection{Cocompactness}\label{ss.coco}

The purpose of this subsection is to prove the following proposition which will complete the proof of Theorem~\ref{teo:dominationimplieshyp}. Recall notation \eqref{e.triple}.

\begin{prop}\label{p.cocompact}
There exists $\eps>0$ such that for every $T \in M^{(3)}$ there exists $\gamma \in \Gamma$ such that 
$|\rho(\gamma) T| \ge \epsilon$.
\end{prop}

We need some preliminaries.
Recall that a \emph{geodesic ray} from the identity is a sequence $(\eta_j)_{j\ge 0}$ such that $\eta_0  =  \id$, $\eta_j  =  g_1 \cdots g_j$ where each $g_j$ belongs to $S$ (the fixed symmetric generating set of $\Gamma$), and $|\eta_j|  =  j$. 
Then we have the following characterization of $M$:

\begin{lem}\label{l.diagonal}
For every $P \in M$ there exists a geodesic ray $(\eta_j)_{j\ge 0}$ from the identity such that $U_p(\rho(\eta_j)) \to P$ as $j \to \infty$.
\end{lem}

\begin{proof}
Fix $P \in M$.
By definition, there exists a sequence $(\gamma_i)$ in $\Gamma$ with $n_i \coloneqq |\gamma_i| \to \infty$ such that $U_p(\rho(\gamma_i)) \to P$.
We can assume that $n_1 < n_2 < \cdots$.
Write each $\gamma_i$ as a product of elements of $S$, say 
$\gamma_i  =  g_1^{(i)} g_2^{(i)} \cdots g_{n_i}^{(i)}$.
By a diagonal argument, we can assume that each of the sequences $(g_j^{(i)})_i$ stabilizes; more precisely, for every $j \ge 1$ there exists $g_j \in S$ and $k_j$ such that $n_{k_j} \ge j$ and $g_j^{(i)}  =  g_j$ for every $i \ge k_j$. 

Consider the geodesic ray $(\eta_j)$ defined by $\eta_j \coloneqq g_1 \cdots g_j$.
Let $i_j \coloneqq \max \{k_1, \dots, k_j\}$.
We will prove that:
\begin{equation}\label{e.same_lim}
d \big( U_p(\rho(\eta_j)) , U_p(\rho(\gamma_{i_j})) \big) \to 0 \, ,
\end{equation}
which has the desired property $U_p(\rho(\eta_j)) \to P$ as a consequence.

Consider the truncated products 
$\gamma_i^{[j]} \coloneqq g_1^{(i)} g_2^{(i)} \cdots g_{j}^{(i)}$.
Using Lemma~\ref{l.nochangeright} and the fact that the representation $\rho$ is $p$-dominated (or, equivalently, using estimate~\eqref{e.useful}), we see that 
$d \big( U_p(\rho(\gamma_i^{[j]})) , U_p(\rho(\gamma_i^{[j+1]})) \big)$ can be bounded by a quantity exponentially small with respect to $j$ and independent of $i$. 
Therefore  $d \big( U_p(\rho(\gamma_i^{[j]})) , U_p(\rho(\gamma_i)) \big)$ is also exponentially small with respect to $j$. 
Applying this to $i  =  i_j$, we obtain \eqref{e.same_lim}.
\end{proof}

\begin{lem}[Expansivity]\label{l.expansivity}
There exist constants $\delta>0$ and $\ell \in \N$ with the following properties.
For every $P \in M$ there exists $\gamma \in \Gamma$ with $|\gamma|\le \ell$
such that if $P'$, $P''$ belong to the $\delta$-neighborhood of $P$ in $\Gr_p(\R^d)$
then 
$$
d(\rho(\gamma) P' , \rho(\gamma) P'') \ge 2 d(P', P'') \, .
$$
\end{lem}

\begin{proof}
By compactness of $M$, it is sufficient to prove that for every $P \in M$ there exists $\gamma \in \Gamma$ such that
if $P'$, $P''$ belong to a sufficiently small neighborhood of $P$ in $\Gr_p(\R^d)$
then $d(\rho(\gamma) P' , \rho(\gamma) P'') \ge 2 d(P', P'')$.

Given $P \in M$, by Lemma~\ref{l.diagonal} there exists a geodesic ray $(\eta_n)$ such that  $U_p(\rho(\eta_n)) \to P$. Write $\eta_n  =  g_1 \cdots g_n$, where each $g_n \in S$.
Consider the sequence of matrices $(A_0,A_1,\dots)$ given by $A_n \coloneqq \rho(g_{n-1}^{-1})$
By the domination condition~\eqref{eq:domrep}, the sequence belongs to some $\cD(\mu,c,K,d-p,\N)$.
Note that $S_p(A_{n-1} \dots A_0)$ equals $U_p(\rho(\eta_n))$ and thererfore converges to $P$.
Applying Corollary~\ref{c.seq_complement} to the sequence of matrices, we find $\tilde P \in \Gr_{d-p}(\R^d)$ such that, for all $n \ge 0$, \begin{gather*}
\angle \big(\rho(\eta_n^{-1}) \tilde P , \rho(\eta_n^{-1}) P \big) \ge \alpha \, , \\
\frac{\| \rho(\eta_n^{-1})|_{P} \| }{\mm( \rho(\eta_n^{-1})|_{\tilde P})} < \tilde{c} e^{-\tilde{\mu} n} \, ,
\end{gather*} 
where $\alpha$, $\tilde c$, $\tilde \mu$ are positive constants that do not depend on $P$.
Let $b>0$ be given by Lemma~\ref{l.expand}, depending on $\alpha$.
Fix $k$ such that $b \tilde{c}^{-1} e^{\tilde{\mu} k} > 2$, and let $\gamma \coloneqq \eta_k^{-1}$.
Applying Lemma~\ref{l.expand} to $A \coloneqq \rho(\gamma)$, we conclude that 
for all $P'$, $P''$ in a sufficiently small neighborhood of $P$ in $\Gr_p(\R^d)$ we have
$$
d(\rho(\gamma) P' , \rho(\gamma) P'') \ge 2 d(P', P'') \, ,
$$
as we wanted to show.
\end{proof}

\begin{proof}[Proof of Proposition~\ref{p.cocompact}]
Let $\delta$ and $\ell$ be given by Lemma~\ref{l.expansivity}.
Let 
$$
\epsilon \coloneqq \inf \Big\{ d\big( \rho(\gamma)P, \rho(\gamma)P' \big) \st 
\gamma \in \Gamma, \ |\gamma| \le \ell, \ P,P'\in \Gr_p(\R^d), \ d(P,P') \ge \frac{\delta}{2}
\Big\} \, .
$$
So $0 < \epsilon \le \delta/2$.
We claim that
\begin{equation}\label{e.claim}
\forall T \in M^{(3)} \  \exists \gamma \in \Gamma \text{ such that }
|\rho(\gamma)T| \ge \min \{ 2|T|, \epsilon\} \, .
\end{equation}
Indeed, given $T  =  (P_1,P_2,P_3) \in M^{(3)}$, 
we can suppose that $|T| < \epsilon$, otherwise we simply take $\gamma  =  \id$.
Permuting indices if necessary we can assume that $d(P_1,P_2)  =  |T|$.
We apply Lemma~\ref{l.expansivity} and find $\gamma \in \Gamma$ such that the action of $\rho(\gamma)$ on
$N_\delta(P_1)$ (the $\delta$-neighborhood of $P_1$) expands distances by a factor of at least $2$.
Since $\epsilon \le \delta/2$, for each pair $\{i \neq j\} \subset \{1,2,3\}$ we have
$$
\{P_i, P_j \} \subset N_\delta(P_1) \quad \text{or} \quad d(P_i,P_j) \ge \frac{\delta}{2} \, .
$$
So $d(\rho(\gamma) P_i, \rho(\gamma) P_j) \ge \min \{ 2|T|, \epsilon\}$, thus proving the claim \eqref{e.claim}.
Now the proposition follows by an obvious recursive argument.
\end{proof}

\subsection{Conclusion}\label{ss.conclusion}
Now we join the pieces and obtain the main result of this section:

\begin{proof}[Proof of  Theorem \ref{teo:dominationimplieshyp}]
Consider a $p$-dominated representation $\rho \colon \Gamma \to \GL(d,\R)$.
If $\Gamma$ is an elementary group then it is word-hyperbolic and there is nothing to prove.
So assume that $\Gamma$ is non-elementary.
Using the representation $\rho$, we define an action of $\Gamma$ on $\Gr_p(\R^d)$.
Consider the set $M \subset \Gr_p(\R^d)$ defined by \eqref{eq-conjuntoM}, 
which is perfect (by Proposition~\ref{p.perfect_again})
and invariant under the action of $\Gamma$ (by Proposition~\ref{p.PropsSetM}).
The diagonal action of $\Gamma$ on $M^{(3)}$ is 
properly discontinuous  (by Proposition~\ref{p.propdisc})
and cocompact (by Proposition~\ref{p.cocompact}).
Therefore Theorem~\ref{teo-bow} assures that $\Gamma$ is word-hyperbolic.
\end{proof}

%%%%%%%%%%%%%%%%%%%%%%%%%%%%%%%%%%%%%%%%%%%%%%%%%%%

\section{Anosov representations and dominated representations}\label{s.equiv}

The main goal of this section is to show that being $p$-dominated (c.f.\ condition \eqref{eq:domrep}) and satisfying the Anosov condition as defined by Labourie \cite{Labourie-AnosovFlows} (and extended by Guichard-Wienhard \cite{GuichardWienhard} to arbitrary hyperbolic groups) are equivalent. 
That equivalence (among others) is contained in the results of \cite{KLP1,KLP2,KLP3}.
Our approach also yields a slightly different characterization directly related to dominated splittings (see Proposition~\ref{p.DS_implies_dom}).

In the final subsection we discuss relations with characterizations of \cite{GGKW},
and pose some questions.

\medskip

We first introduce the notion of Anosov representations into $\GL(d,\RR)$ which requires introducing the geodesic flow of a hyperbolic group.

\subsection{The geodesic flow}\label{ss-geodesicflow}

In order to define the Anosov property for a representation of a hyperbolic group, we need to recall the \emph{Gromov geodesic flow} of $\Gamma$.

Given a word-hyperbolic group $\Gamma$ we can define its visual boundary $\partial \Gamma$ (c.f.\ Remark \ref{rem-boundary}).
Denote 
$\partial^{(2)} \Gamma \coloneqq \{ (x,y) \in \partial \Gamma \times \partial \Gamma \st  x\neq y \}$.
We define a flow on the space $\widetilde{U\Gamma} \coloneqq \partial^{(2)}\Gamma \times \RR$, called the \emph{lifted geodesic flow} by the formula $\tilde \phi^t(x,y,s) \coloneqq (x,y,s+t)$. 

A function 
$c: \Gamma \times \partial^{(2)} \Gamma \to \RR$ such that 
%$c: \partial^{(2)} \Gamma \times \Gamma \to \RR$ such that 
$$
c(\gamma_0\gamma_1, x, y)  =  c(\gamma_0, \gamma_1(x, y)) + c(\gamma_1, x, y)
\quad \text{for any $\gamma_0,\gamma_1 \in \Gamma$ and $(x,y)\in \partial^{(2)}\Gamma$}
$$
is called a \emph{cocycle}.
Recall that every infinite order element $\gamma \in \Gamma$  
acts on $\partial \Gamma$ leaving only two fixed points, an attractor $\gamma^+$ and a repeller $\gamma^-$.
Let us say that a cocycle $c$ is \emph{positive} if $c(\gamma,\gamma^-,\gamma^+)>0$ for every such $\gamma$.

Given a cocycle, we can define an action of $\Gamma$ on $\widetilde {U\Gamma}$ by $\gamma \cdot (x,y,s) =  (\gamma \cdot x, \gamma \cdot y, s - c(x,y,\gamma))$, which obviously commutes with the lifted geodesic flow. 
Gromov \cite{gromov} (see also \cite{matheus, champetier,mineyev}) proved that there exists a positive cocycle such that the latter action is properly discontinuous and cocompact. This allows to define the \emph{geodesic flow} $\phi^t$ of $\Gamma$ on $U\Gamma \coloneqq \widetilde{U\Gamma}/\Gamma$, the \emph{unit tangent bundle} of $\Gamma$.

There is a metric on $\widetilde{U\Gamma}$, well-defined up to H\"older equivalence, so that $\G$ acts by isometries,
the lifted geodesic flow acts by bi-Lipschitz homeomorphisms, and its flow lines are quasi-geodesics.

\begin{obs}
If $\Gamma  =  \pi_1(M)$ is the fundamental group of a negatively curved closed manifold $M$ then the geodesic flow on the unit tangent bundle $UM$ is hyperbolic and equivalent to the abstract geodesic flow defined above. In that case, the unit tangent bundle of the universal cover $\tilde M$ is homeomorphic to $\widetilde{U\Gamma}$ by means of the Hopf parametrization. For details, see \cite{Ledrappier-Bord}.
\end{obs}

\begin{lem}\label{l.comparison}
For any compact set $K \subset \widetilde{U\Gamma}$, there exist $a>0$ and $\kappa>1$ such that 
if $t \in \R$ and $\gamma \in \Gamma$ satisfy 
$$
\tilde\phi^t(K) \cap \gamma(K) \neq \emptyset
$$
then
$$
\kappa^{-1} |t| - a \le |\gamma| \le \kappa |t| + a \, .
$$
\end{lem}

\begin{proof}
Take a ball $B(u_0, r)$ containing $K$. From the construction (see e.g. \cite[Theorem IV.1 (ii)]{matheus}  or \cite{mineyev}) of the metric in $\widetilde{U \Gamma}$ one has that  the map $\gamma \in \Gamma \mapsto \gamma u_0 \in \widetilde{U\Gamma}$ is a quasi-isometry, so there exist $\kappa>1$, $b>0$ such that for all $\gamma_1$, $\gamma_2 \in \Gamma$ we have 
$$
\kappa^{-1} d(\gamma_1,\gamma_2) - b \le d(\gamma_1 u_0, \gamma_2 u_0) \le \kappa d(\gamma_1,\gamma_2) + b
$$
In particular,
$$
\kappa^{-1} |\gamma| - b \le d(\gamma u_0, u_0) \le \kappa |\gamma| + b \, .
$$
Now assume that $t$ and $\gamma$ satisfy 
$\tilde\phi^t(K) \cap \gamma(K) \neq \emptyset$,
that is, there exist $u_1$, $u_2 \in K$ such that $\tilde\phi^t u_1  =  \gamma u_2$.
So\footnote{To get exact equality in $(*)$ we need the construction in \cite{mineyev}, for which orbits of the flow are geodesics. If instead we use \cite{champetier} or \cite{matheus}, the orbits of the flow are quasi-geodesics, so the equality $(*)$ is only approximate, which is sufficient for our purpose.}
$$
|t| \stackrel{*}{ = } d(\tilde\phi^t u_1, u_1)  =  d(\gamma u_2, u_1)
\begin{cases}
	\le d(\gamma u_0, u_0) + 2r \\ 
	\ge d(\gamma u_0, u_0) - 2r 
\end{cases}
$$
The desired inequalities follow.
\end{proof}

\subsection{Equivariant maps and the definition of Anosov representations} 

Let $\rho: \Gamma \to \GL(d,\RR)$ be a representation of a word-hyperbolic group $\Gamma$. The definitions here can be adapted for representations into general semisimple Lie groups and the results are equivalent. In order to be able to present our results in a more elementary manner, we have deferred the introduction of the general context to section \ref{s.general}.

We say that the representation $\rho$ is \emph{$p$-convex} if there exist continuous maps $\xi: \partial \Gamma \to \Gr_{p}(\RR^d)$ and $\theta: \partial \Gamma \to \Gr_{d-p}(\RR^d)$ such that:
\begin{itemize}
\item ({\bf transversality:}) for every $x \neq y \in \partial \Gamma$ we have $\xi(x)\oplus \theta(y) =  \RR^d$, 
\item ({\bf equivariance:}) for every $\gamma \in \Gamma$ we have $\xi (\gamma \cdot x) =  \rho(\gamma) \xi(x)$ and $\theta( \gamma \cdot x)  =  \rho(\gamma) \theta(x)$. 
\end{itemize}

Using the representation $\rho$, it is possible to construct a linear flow $\psi^t$  over the geodesic flow $\phi^t$ of $\Gamma$ as follows. 
Consider the lifted geodesic flow $\tilde \phi^t$ on $\widetilde{U\Gamma}$, and define  a linear flow on $\tilde E \coloneqq \widetilde{U\Gamma} \times \RR^d$ by: 
$$ 
\tilde \psi^t ((x,y,s),v) \coloneqq (\tilde\phi^t(x,y,s), v) , 
$$ 
Now consider the action of $\Gamma$ on $\tilde E$ given by: 
$$ 
\gamma \cdot ((x,y,s),v) \coloneqq (\gamma \cdot (x,y,s), \rho(\gamma)v) 
$$
where the action of $\Gamma$ in $\widetilde{U\Gamma}$ is the one explained in Subsection \ref{ss-geodesicflow}. It follows that $\tilde \psi^t$ induces in $E_\rho \coloneqq \tilde E /_{\Gamma}$ (which is a vector bundle over $U\Gamma$) a linear flow $\psi^t$ which covers $\phi^t$. See Fig.~\ref{f.cube}.

\begin{figure}[htb] 
\begin{center}
\includegraphics{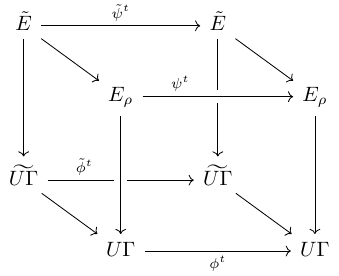}
\end{center}
\caption{A commutative diagram. 
The $\downarrow$ arrows are vector bundle projections. The $\searrow$ arrows are quotient maps w.r.t.\ the corresponding actions of $\Gamma$. The  $\rightarrow$ arrows are the flow actions, as indicated.}\label{f.cube}
\end{figure}

When the representation $\rho$ is $p$-convex, by equivariance there exists a $\psi^t$-invariant splitting of the form $E_\rho =  \Xi \oplus \Theta$; it is obtained taking the quotient of the bundles $\tilde \Xi (x,y,s) \coloneqq \xi(x)$ and $\tilde \Theta(x,y,s) \coloneqq \theta(y)$ with respect with the $\Gamma$-action.

We say that a $p$-convex representation $\rho$ is \emph{$p$-Anosov} if the splitting $E_\rho =  \Xi \oplus \Theta$ is a dominated splitting for the linear bundle automorphism $\psi^t$, with $\Xi$ dominating $\Theta$. This is equivalent to the fact that the bundle $\mathrm{Hom}(\Theta,\Xi)$ is uniformly contracted by the flow induced by $\psi^t$ (see \cite{BCLS}). 

Conversely, dominated splittings for the linear flow $\psi^t$ must be of the form $\Xi \oplus \Theta$ as above: see Propositions~\ref{p.DS_implies_dom} and \ref{p.DominatedImpliesAnosov} below.

Let us mention that by \cite[Theorem 1.5]{GuichardWienhard}, if the image of the representation $\rho$ is Zariski dense, being $p$-Anosov is a direct consequence of being $p$-convex. 

\begin{obs} As before, it is possible to use exterior powers to transform a $p$-Anosov representation into a $1$-Anosov one. The latter are called \emph{projective Anosov} and are discussed in Section~\ref{s.general} as it is shown that any Anosov representation to an arbitrary semisimple Lie group can be transformed into a projective Anosov one. See also \cite[Section 2.3]{BCLS}.
\end{obs}

\subsection{Equivalence between the definitions}\label{ss.equivalences}

We will show that a representation $\rho: \Gamma \to \GL(d,\R)$ is $p$-dominated if and only if it is $p$-Anosov.
Note that the definition of $p$-Anosov representation requires the group to be word-hyperbolic.
On the other hand, we have shown in Section \ref{s.dominationimplieshyp} that if $\rho$ is $p$-dominated then the group $\Gamma$ is automatically word-hyperbolic. 
So we can assume in what follows that $\Gamma$ is word hyperbolic.

Let us first show the following:

\begin{lem}\label{l.comparison2}
Endow $E_\rho$ with a Riemannian metric\footnote{The Riemannian metric allows us to consider singular values for the linear maps $\psi^t_x$.}.
Then there exist constants $\kappa>1$, $a>0$, and $C>1$ with the following properties:
\begin{enumerate}
\item\label{i.U_to_Gamma}
For every $z \in U\Gamma$ and $t \in \R$, there exists $\gamma\in \Gamma$ such that:
\begin{alignat}{3}
\kappa^{-1} |t| - a &\le \hspace{1.5em} |\gamma| &&\le \kappa |t| + a &\quad &\text{and} \label{e.comp_t_gamma}\\
C^{-1}\sigma_p(\psi_z^t) &\le \sigma_p(\rho(\gamma)) && \le C \sigma_p(\psi_z^t)
&\quad &\text{for every $p = 1,\dots,d-1$.}  \label{e.comp_psi_rho}
\end{alignat}

\item\label{i.Gamma_to_U}
Conversely, for every $\gamma\in \Gamma$ there exist $z \in U\Gamma$ and $t \in \R$ such that \eqref{e.comp_t_gamma} and \eqref{e.comp_psi_rho} hold.

\end{enumerate}
\end{lem}

\begin{proof}
Using the covering map $\tilde E \to E_\rho$, we lift the fixed Riemannian metric on $E_\rho$, obtaining a Riemannian metric $\| \mathord{\cdot}\|_*$ on $\tilde E$ preserved by the action of $\Gamma$.

On the other hand, since the vector bundle $\tilde E$ is trivial, we can also endow it with the euclidian metric $\| \mathord{\cdot}\|$ on the fibers. Let $K \subset \widetilde{U\Gamma}$ be a compact set intersecting every $\Gamma$-orbit. Then there exists $C_K>1$ such that, for every $v \in \tilde{E}$ that projects on $K$,
$$
C_K^{-1} \|v\| \le \|v\|_* \le C_K \|v\| \, .
$$
By Lemma~\ref{l.sing_value_change}, a bounded change of inner product has a bounded effect on the singular values. It follows that for all $\tilde{z} \in K$, $t \in \R$, and $\gamma \in \Gamma$ such that $\tilde\phi^t(\tilde z) \in \gamma(K)$, if $z$ is the projection of $\tilde z$ in $U\Gamma$, inequality \eqref{e.comp_psi_rho} holds for $C  =  C_K^2$.

The rest of the proof of part~(\ref{i.U_to_Gamma}) follows directly from Lemma~\ref{l.comparison}.

Now let us prove part~(\ref{i.Gamma_to_U}).
Consider the action of $\Gamma$ on the compact metric space $\partial \Gamma$. 
Fix a positive $\delta < \frac{1}{4} \mathrm{diam}\, \partial \Gamma$.

\begin{claim} For every $\gamma \in \Gamma$ there exists $x,y \in \partial \Gamma$ such that $d(x,y) \geq \delta$ and $d(\gamma^{-1}x, \gamma^{-1} y) \ge \delta$.
\end{claim}
\begin{proof}
Fix $x_1$, $x_2 \in \partial\Gamma$ such that $d(x_1, x_2) > 3\delta$.
We can assume that $d(\gamma^{-1}x_1, \gamma^{-1} x_2) < \delta$, because otherwise we would take $(x,y) \coloneqq (x_1,x_2)$.
It follows that there exists $i \in \{1,2\}$ such that $d(x_i, \gamma^{-1} x_1)$ and $d(x_i, \gamma^{-1} x_2)$ are both bigger than $\delta$.
Analogously, we can assume that $d(\gamma x_1, \gamma x_2) < \delta$, because otherwise we would take $(x,y) \coloneqq (\gamma x_1,\gamma x_2)$.
It follows that there exists $j \in \{1,2\}$ such that $d(x_j, \gamma x_1)$ and $d(x_j, \gamma x_2)$ are both bigger than $\delta$.
Therefore the pair $(x,y) \coloneqq (x_j, \gamma x_i)$ has the desired properties.
\end{proof}

Let $K \coloneqq \{(x,y,0) \st x,y \in \partial\Gamma, \ d(x,y)\ge \delta\}$; this is a compact subset of $\widetilde{U\Gamma}$. Therefore, letting $\tilde{z} \coloneqq (x,y,0)$ and $t \coloneqq c(x,y,\gamma)$,
we conclude with the same arguments as above.
\end{proof} 

Now let us prove the equivalence between $p$-dominated and $p$-Anosov representations.
We first show: 

\begin{prop}\label{p.AnosovImpliesDominated} Let $\rho: \Gamma \to \GL(d,\R)$ be a $p$-Anosov representation. Then $\rho$ is $p$-dominated.
\end{prop}

\begin{proof}
Since $\rho$ is a $p$-Anosov representation, Theorem \ref{t.BG} implies that there exists $C,\lambda>0$ such that for every $z\in U\Gamma$ and $t>0$ we have:
$$ 
\frac{\sigma_{p+1}(\psi^t_z)}{\sigma_{p}(\psi^t_z)} < C e^{-\lambda t} .
$$
Using part~(\ref{i.U_to_Gamma}) of Lemma~\ref{l.comparison2}, we can find constants $C', \lambda'>0$  so that for every $\gamma \in \Gamma$ we have:
$$ 
\frac{\sigma_{p+1}(\rho(\gamma))}{\sigma_{p}(\rho(\gamma))} < C' e^{-\lambda' |\gamma|}.
$$
This means that $\rho$ is $p$-dominated. 
\end{proof}

Note that the proof above only uses the fact that the linear flow $\psi^t$ has a dominated splitting, so we obtain:

\begin{prop}\label{p.DS_implies_dom}
If the linear flow $\psi^t$ on $E_\rho$ has a dominated splitting with dominating bundle of dimension $p$ then $\rho$ is $p$-dominated.
\end{prop}

To prove that $p$-domination implies the $p$-Anosov property, we shall first show the existence of the equivariant maps $\xi$, $\theta$. This is a relatively easy consequence of what is done in Section \ref{s.dominationimplieshyp} (see Remark \ref{rem-boundary}).  The equivariant maps exist under an even weaker hypothesis, as shown in \cite[Theorem~5.2]{GGKW}. We provide here a proof for completeness. 

Recall that a $(a,b)$-quasi-geodesic in $\Gamma$ is a sequence $\{\gamma_n\}$ so that 
$$
a^{-1} |n-m| - b < d(\gamma_n, \gamma_m) <  a |n-m| +b \, .
$$ 
We denote by $\mathcal{Q}_{(a,b)}^{\mathrm{id}}$ the set of $(a,b)$-quasi-geodesics such that $\gamma_0 = \mathrm{id}$.  

\begin{lem}\label{lem-equivariant} 
Let $\rho: \Gamma \to \GL(d,\RR)$ be a representation such that for some $a,b>0$ we have:
\begin{equation}\label{eq:uniformsum}  \sup_{{\gamma_n} \in \mathcal{Q}_{(a,b)}^{\mathrm{id}}} \sum_{n\geq n_0} \frac{\sigma_{p+1}(\rho(\gamma_n))}{\sigma_p(\rho(\gamma_n))} \xrightarrow[n_0 \to \infty]{} 0  \ . 
\end{equation}
Then there exists an equivariant continuous map $\xi: \partial \Gamma \to \Gr_p(\RR^d)$ defined by
$$ \xi (x) \coloneqq \lim_n U_p(\rho(\gamma_n)), $$
where $\{\gamma_n\}$ is any $(a,b)$-quasi-geodesic ray representing $x \in \partial \Gamma$. 
\end{lem}

\begin{proof} 
For each $x \in \partial \Gamma$, choose $\{\gamma_n^x\} \in \mathcal{Q}_{(a,b)}^{\mathrm{id}}$ representing $x$. We define 
$$ \xi (x) \coloneqq \lim_{n} U_p(\rho(\gamma_n^x)) \ .$$
To see that this limit exists, let $C_0$ be an upper bound of $\|\rho(g)\| \|\rho(g^{-1})\|$ for $g\in S$ a finite generating set of $\Gamma$ and use Lemma \ref{l.nochangeright} to see that
$$ d(U_p(\rho(\gamma_n^x)), U_p(\rho(\gamma_{n-1}^x))) \leq C_0^{d(\gamma_n,\gamma_{n-1})} \frac{\sigma_{p+1}(\rho(\gamma^x_{n-1}))}{\sigma_p(\rho(\gamma^x_{n-1}))} $$
This implies that $U_p(\rho(\gamma_n^x))$ is a Cauchy sequence and therefore has a limit. The fact that the limit does not depend on the chosen $(a,b)$-quasi-geodesic follows directly from a similar estimate using Lemma \ref{l.nochangeright}. 

Since the estimates are uniform, this becomes a uniform limit as one changes $x\in \partial \Gamma$, providing continuity of the maps (recall the topology in $\partial \Gamma$ introduced in Remark \ref{rem-boundary}). Equivariance follows from Lemma \ref{l.domination_implies_slow_change}. 
\end{proof}

\begin{obs}
If $\rho$ is $p$-dominated then it satisfies the hypothesis of Lemma \ref{lem-equivariant}, since the terms in the sum of \eqref{eq:uniformsum} are uniformly exponentially small. 
\end{obs} 

\begin{prop}\label{p.DominatedImpliesAnosov} Let $\rho: \Gamma \to \GL(d,\R)$ be a $p$-dominated representation. Then $\rho$ is $p$-Anosov.
\end{prop}

\begin{proof} If a representation is $p$-dominated, then it is also $(d-p)$-dominated (see Remark \ref{rem-pdominatedimpliesd-p}). Lemma \ref{lem-equivariant} then provides two equivariant continuous maps $\xi: \partial \Gamma \to \Gr_p(\RR^d)$ and $\theta: \partial \Gamma \to \Gr_{d-p}(\RR^d)$. 

The fact that $\xi (x) \oplus \theta(y)  =  \RR^d$ for $x \neq y \in \partial \Gamma$ is a direct consequence of Lemma~\ref{l.transv} and the definition of the maps $\xi$ and $\theta$ given by Lemma~\ref{lem-equivariant}. 

Using  part~(\ref{i.Gamma_to_U}) of Lemma~\ref{l.comparison2}, we obtain an exponential gap in the singular values of $\psi^t$, and by Theorem \ref{t.BG}, the splitting $\xi \oplus \theta$ is dominated. Therefore the representation $\rho$ is $p$-Anosov.
\end{proof}

\subsection{Some questions}

Given a matrix $A \in \GL(d,\R)$, let
$$
\chi_1(A) \ge \chi_2(A) \ge \cdots \ge \chi_d(A)
$$
denote the absolute values of its eigenvalues, repeated according to multiplicity.

Given a finitely generated group, let 
$$
\ell(\eta) \coloneqq \inf_\eta |\eta^{-1} \gamma \eta|  =  \inf_\eta d(\gamma\eta, \eta)
$$
(i.e., the \emph{translation length}).
If $\Gamma$ is word-hyperbolic then there exists a constant $a > 0$ such that
for every $\gamma \in \Gamma$ we have:
\begin{equation}\label{e.CDP}
\ell(\gamma) - a \le \lim_{n \to \infty} \frac{|\gamma^n|}{n} \le \ell(\gamma) \, ;
\end{equation}
see \cite[p.~119]{CDP}.

Note that if $\rho \colon \Gamma \to \GL(d,\R)$ is a $p$-dominated representation then
there exists constants $C'>0$, $\lambda>0$ such that for all $\gamma \in \Gamma$ we have:
\begin{equation}\label{e.eigen_gap}
\frac{\chi_{p+1}(\rho(\gamma))}{\chi_p(\rho(\gamma))} < C' e^{-\lambda' \ell(\gamma)} \, .
\end{equation}
Indeed, if the domination condition \eqref{eq:domrep} holds then the group $\Gamma$ is word-hyperbolic by Theorem~\ref{teo:dominationimplieshyp} and, using \eqref{e.CDP}, we obtain:
$$
\frac{\chi_{p+1}(\rho(\gamma))}{\chi_p(\rho(\gamma))}  = 
\lim_{n \to \infty} \left( \frac{\sigma_{p+1}(\rho(\gamma^n))}{\sigma_p(\rho(\gamma^n))} \right)^{1/n} 
\le \lim_{n \to \infty} \left( C e^{-\lambda|\gamma^n|} \right)^{1/n} 
\le C' e^{-\lambda \ell(\gamma)} \, ,
$$
for $C' \coloneqq e^{a \lambda}$.

Condition \eqref{e.eigen_gap} is invariant under conjugacies, while condition \eqref{eq:domrep} does not enjoy this property.

It is natural to pose the following question:

\begin{question}\label{q.eigen_domination}
Let $\rho \colon \Gamma \to \GL(d,\R)$ be a representation of a finitely generated group $\Gamma$.
Suppose that there exists constants $p \in \{1,\dots,d-1\}$, $C'>0$, $\lambda>0$ such that relation \eqref{e.eigen_gap} holds.
Does it follow that $\rho$ is $p$-dominated?
\end{question}

Gu\'eritaud, Guichard, Kassel, and Wienhard have shown that for $p$-convex representations, the question above has a positive answer, even relaxing condition \eqref{e.eigen_gap}: see \cite[Theorem~1.6]{GGKW}.\footnote{Kassel has informed us that techniques similar to those in \cite{GGKW} allow one to give a positive answer to Question \ref{q.eigen_domination} for certain word hyperbolic groups, including free groups and surface groups.}

In terms of the linear flow $\{\psi^t\}$,
condition \eqref{e.eigen_gap} means that for every periodic orbit $\cO$ of $\{\phi^t\}$, say of period $\ell(\cO)$, there exists a gap between the $p$-th and $p+1$-th moduli of the eigenvalues of $\psi^{\ell(\cO)}$ which is exponentially large with respect to $\ell(\cO)$.
Question~\ref{q.eigen_domination} can be reformulated in the general context of linear flows over hyperbolic dynamics; however, that question has a negative answer: see for example \cite{Gogolev}.
Therefore a positive answer to Question~\ref{q.eigen_domination} would require a finer use of the fact that the linear flow comes from a representation.

\medskip

The following important result was obtained by Bonatti, D\'iaz, and Pujals \cite{BDP}: if a diffeomorphism $f$ of a compact manifold has that property that all sufficiently small $C^1$-perturbations have dense orbits, then the derivative cocycle $Df$ admits a dominated splitting.
This is an example of a general principle in Differentiable Dynamics, tracing back to Palis--Smale Stability Conjecture: robust dynamical properties often imply some uniform property for the derivative.
Coming back to the context of linear representations, one can try to apply the same principle.
For example, if a representation $\rho$ of a hyperbolic group is robustly faithful and discrete (or robustly quasi-isometric), does it follow that $\rho$ is $p$-dominated for some $p$?

%%%%%%%%%%%%%%%%%%%%%%%%%%%%%%%%%%%%%%%%%%%%%%%%%%%%%

\section{Characterizing dominated representations in terms of multicones}\label{s.multicones}

The main result of this section is Theorem~\ref{t.multicone_representation}, which gives another characterization of dominated representations. Related results have been obtained in \cite{ABY,BG}. In dimension two, such results have been used to study how domination can break along a deformation: see \cite[\S~4]{ABY}.

As a consequence of Theorem~\ref{t.multicone_representation}, domination obeys a ``local-to-global'' principle, a fact that was first shown in \cite{KLP1} by different methods.

\subsection{Sofic linear cocycles and a general multicone theorem} \label{ss.sofic}

In this subsection we introduce a special class of linear cocycles called \emph{sofic}.
Then we state a necessary and sufficient condition for the existence of a dominated splitting for these cocycles,
generalizing the ``multicone theorems'' of \cite{ABY,BG}.

\medskip

Let $\mathcal{G}$ be a \emph{graph}, or more precisely, a finite directed multigraph. 
This means that we are given finite sets $\mathcal{V}$ and $\mathcal{E}$ whose elements are called respectively \emph{vertices} and \emph{edges}, and that each edge has two (not necessarily different) associated vertices, called its \emph{tail} and its \emph{head}.
A \emph{bi-infinite walk} on $\mathcal{G}$ is a two-sided sequence of edges $(e_n)_{n \in \Z}$ such that 
for each $n \in \Z$, the head of $e_n$ equals the tail of $e_{n+1}$.

The graph $\mathcal{G}$ is called \emph{labeled} if in addition each edge has an associated \emph{label}, taking values in some finite set $\mathcal{L}$.
Let $(e_n)_{n \in \Z}$ be a bi-infinite walk on $\mathcal{G}$; then its \emph{label sequence} is defined as $(\ell_n)_{n\in \Z}$ where each $\ell_n$ is the label of the edge $e_n$.
Let $\Lambda$ be the set of all label sequences; this is a closed, shift-invariant subset of $\mathcal{L}^\Z$.
Let $T \colon \Lambda \to \Lambda$ denote the restriction of the shift map.
Then $T$ is called a \emph{sofic shift},
and the labeled graph from which it originates is called a \emph{presentation} of $T$.
We refer the reader to \cite{LindMarcus} for examples, properties, and alternative characterizations of sofic shifts. 
Let us just remark that every subshift of finite type is a sofic shift, and every sofic shift is a factor of a subshift of finite type.

Let us say that a graph is \emph{recurrent} if it is a union of directed cycles. 
Given a (labeled) graph $\cG$, let $\cG^*$ denote the maximal recurrent (labeled) subgraph, or equivalently, the subgraph containing all the bi-infinite walks on $\cG$.
Note that the sofic shifts presented by $\cG$ and $\cG^*$ are exactly the same.
Therefore we may always assume that $\cG$ is recurrent, if necessary.

\medskip

Fix a sofic shift $T$ and a presentation $\mathcal{G}$ as above.
Let $d \ge 2$ be an integer.
Given a family of matrices $(A_\ell)_{\ell \in \mathcal{L}}$ in $\GL(d,\R)$,
consider the map $A \colon \Lambda \to \GL(d,\R)$ defined by $A((\ell_n)_{n\in \Z}) \coloneqq A_{\ell_0}$.
We call the pair $(T,A)$ a \emph{sofic linear cocycle}. 
We are interested in the existence of dominated splitting for such cocycles.

\medskip

A \emph{multicone of index $p$} is an open subset of the projective space $\P(\R^d)  =  \Gr_1(\R^d)$ that contains the projectivization of some $p$-plane and does not intersect the projectivization of some $(d-p)$-plane. 
Such a multicone is called \emph{tame} if it has finitely many connected components, and these components have disjoint closures.

Suppose that for each vertex $v$ of the graph $\mathcal{G}$ it is given a multicone $M_v \subset \P(\R^d)$ of index $p$; then we say that $(M_v)_{v\in \mathcal{V}}$ is \emph{a family of multicones of index $p$}. We say that this family is \emph{strictly invariant} (with respect to the sofic linear cocycle) if for each edge $e \in \mathcal{E}$ we have\footnote{$X \Subset Y$ denotes that the closure of $X$ is contained in the interior of $Y$.}  
$$
A_\ell(M_{v_0}) \Subset M_{v_1} \, ,
$$
where $\ell$ is the label of $e$, $v_0$ is the tail of $e$, and $v_1$ is the head of $e$.

\begin{teo}\label{t.multicone_sofic} 
Let $T$ be a sofic shift with a fixed presentation.
Consider a sofic linear cocycle $(T,A)$.
The following statements are equivalent:
\begin{enumerate}[label = {\textnormal{(\alph*)}},ref = {\textnormal{(\alph*)}}]
\item\label{i.ds} 
the cocycle $(T,A)$ has a dominated splitting $E^{\cu} \oplus E^{\cs}$ 
where the dominating bundle $E^\cu$ has dimension $p$;
\item\label{i.multicone}
there exists a strictly invariant family of multicones of index $p$. %$(M_v)_{v\in \mathcal{V}}$.
\end{enumerate}
Moreover, in \ref{i.multicone} we can always choose a family composed of tame multicones.
\end{teo}

\begin{remark} 
It is not always possible to choose \emph{connected} multicones %(i.e., cones) 
in \ref{i.multicone}.
Let us sketch the simplest example, referring the reader to \cite{ABY} for more information.
Let $T$ be the full shift on two symbols $1$, $2$, presented by the graph $\begin{tikzcd} \arrow[loop left, "1"] \bullet \arrow[loop right, "2"] \end{tikzcd}$. 
Let $A_1$, $A_2$ be matrices in $\SL(2,\R)$ whose unstable and stable directions in $\P(\R^2)$ can be cyclically ordered as: 
$$
E^\mathrm{u}(A_1) < E^\mathrm{s}(A_1) < E^\mathrm{u}(A_2) < E^\mathrm{s}(A_2)  \, .
$$
If the eigenvalues are sufficiently away from $1$ then the sofic cocycle $(T,A)$ has a dominated splitting. 
A strictly invariant family of multicones consists on a single element, since the graph has a single vertex. 
It is possible to chose a multicone $M \subset \P(\R^2)$ with two connected components.
But $M$ cannot be chosen connected, since it must contain $\{E^\mathrm{u}(A_1) , E^\mathrm{u}(A_2)\}$ and it cannot intersect $\{E^\mathrm{s}(A_1) , E^\mathrm{s}(A_2)\}$.

\end{remark}

Theorem~\ref{t.multicone_sofic} implies, for example, \cite[Theorem~2.2]{ABY}; to see this, note that the full shift on $N$ symbols can be presented by the graph with a single vertex and self-loops labeled $1$, \dots, $N$.
Theorem~\ref{t.multicone_sofic} also extends \cite[Theorem~2.3]{ABY} and \cite[Theorem~B]{BG} (in the case of finite sets of matrices).

One can prove Theorem~\ref{t.multicone_sofic} %and Proposition~\ref{p.multicone_sofic_Ecu} 
by minor adaptation of the proof of \cite[Theorem~B]{BG}; for the convenience of the reader we provide details in Subsection~\ref{ss.multicone_sketch}.

As a complement to Theorem~\ref{t.multicone_sofic}, let us explain how to obtain the dominating bundle $E^\cu$ in terms of the multicones:

\begin{prop}\label{p.multicone_sofic_Ecu}
Consider a sofic linear cocycle $(T,A)$ with a dominated splitting $E^\cu \oplus E^\cs$ where the dominating bundle $E^\cu$ has dimension $p$.
Let $\{M_v\}_v$ be a strictly invariant family of multicones. 
Consider an element $x  =  (\ell_n)_{n \in \Z} \in \Lambda$, that is, the label sequence of a bi-infinite walk $(e_n)_{n \in \Z}$.
Let $v_n$ be the tail of the edge $e_n$.
Let $(P_n)_{n\in \Z}$ be a sequence in $\Gr_p(\R^d)$ such that for each $n$, the projectivization of $P_n$ is contained in the multicone $M_{v_n}$.
Then
$$
E^\cu(x)  =  \lim_{n \to +\infty} A_{\ell_{-1}} A_{\ell_{-2}} \cdots A_{\ell_{-n}} (P_{-n}) \, .
$$
Moreover, the speed of convergence is exponential and can be estimated independently of $x$, $(e_n)$, and~$(P_n)$,
and is the same for all nearby sofic linear cocycles $(T,\tilde{A})$. 
\end{prop}

\subsection{How to prove Theorem~\ref{t.multicone_sofic} and Proposition~\ref{p.multicone_sofic_Ecu}}\label{ss.multicone_sketch}

The implication \ref{i.multicone} $\Rightarrow$ \ref{i.ds} in Theorem~\ref{t.multicone_sofic} and  Proposition~\ref{p.multicone_sofic_Ecu} actually hold in much greater generality:

\begin{prop}\label{p.conefield}
If a linear cocycle $(T,A)$ \footnote{or, more generally, a linear flow in the context of Subsection~\ref{ss.def_ds}} admits a strictly invariant continuous field $\cC$ of multicones of index $p,$ then the cocycle admits a dominated splitting $E^\cu \oplus E^\cs$ where the dominating bundle $E^\cu$ has dimension $p$.
Moreover, 
\begin{equation}\label{e.conefield_and_bundles}
E^{\cu}(x) \in \cC(x) \quad \text{and} \quad
\inf_{P \in \cC(x)} \angle\big(P, E^{\cs}(x) \big) > 0 
\quad \text{for every $x$.}
\end{equation}
\end{prop}

The proposition can be shown by straightforward adaptation of the proof of \cite[Theorem~2.6]{CroPo}, and we skip the details. 

\medskip

Proposition~\ref{p.multicone_sofic_Ecu} follows from the angle bound in \eqref{e.conefield_and_bundles} combined with Lemma~\ref{l.expand}.

\medskip

We now prove the implication \ref{i.ds} $\Rightarrow$ \ref{i.multicone} in Theorem~\ref{t.multicone_sofic}. Suppose that the sofic cocycle $(T,A)$ admits a dominated splitting $E^{\cu} \oplus E^{\cs}$. For each vertex $v$ of the graph $\mathcal{G}$, let $\Lambda_v$ be the (compact) set of label sequences $x  =  (\ell_n)$ in $\Lambda$ for which the tail of $\ell_0$ is $v$. 
The space $E^{\cu}(x)$ depends only on the past of the sequence (i.e.\ on $(\ell_n)_{n<0}$) while the space $E^{\cs}(x)$ depends only on its future (i.e.\ on $(\ell_n)_{n\geq 0}$); this follows from the characterizations \eqref{e.BG_cu} and \eqref{e.BG_cs}.

For each vertex $v$, the set $\mathcal{K}_v^{\cu}$ consisting of the spaces $E^{\cu}(x)$ with $x \in \Lambda_v$ is a compact subset of the Grassmannian $\grass_p(\R^d)$.
Analogously we define a compact set $\mathcal{K}_v^{\cs} \subset \grass_{d-p}(\R^d)$.
We claim that these two sets are \emph{transverse} in the sense that any $E^\cu(x)$ in $\mathcal{K}_v^{\cu}$ is  transverse to any $E^\cu(y)$ in $\mathcal{K}_v^{\cs}$.
Indeed, consider the sequence $z$ formed by concatenation of the past of $x$ with the future of $y$. Note that this is indeed the label sequence of a bi-infinite path (i.e.\ admissible sequence of edges) in the graph, also obtained by concatenation.
By a previous remark, $E^\cu(z)  =  E^\cu(x)$ and $E^\cu(z)  =  E^\cu(y)$, and so these two spaces are transverse.

Let $\mathcal{K}_v^{\pitchfork \cs}$ be the open subset of $\grass_p(\R^d)$ formed by the $p$-planes that are transverse to $\mathcal{K}_v^{\cs}$. 
Define a metric $d_v$ on $\mathcal{K}_v^{\pitchfork \cs}$ as follows:
$$
d_v(P,Q) \coloneqq \sum_{N = 0}^\infty \sup_{(\ell_n) \in \Lambda_v}  
d \big( A_{\ell_{N-1}} \cdots A_{\ell_{0}} (P) ,  A_{\ell_{N-1}} \cdots A_{\ell_{0}} (Q) \big) \, ;
$$
the convergence is exponential as a consequence of Lemma~\ref{l.expand}.
This family of metrics is \emph{adapted} in the sense that they are contracted by a single iteration of the cocycle. 
More precisely, if an edge $e \in \mathcal{E}$ has tail $v$, head $w$, and label $\ell$ then: 
$$
d_w (A_\ell(P), A_\ell(Q)) < d_v (P,Q) \quad \text{for all } P, Q \in \mathcal{K}_v^{\pitchfork \cs}.
$$
So, fixing $\varepsilon>0$ and letting $M_v^*$ denote the $\varepsilon$-neighborhood of $\mathcal{K}_v^{\cu}$ with respect to the metric $d_v$, we have the invariance property $A_\ell(M_v^*) \subset M_w^*$.
Let $M_v \subset \P(\R^d)$ be the union of lines contained in elements of $M_v^*$; then $\{M_v\}_v$ is an invariant family of multicones.
Moreover, by the same argument as in \cite[p.~288]{BG}, for an appropriate choice of $\varepsilon$ the multicones are tame and strictly invariant.

\subsection{Cone types for word-hyperbolic groups}\label{ss.cone_types}

Cone types were originally introduced by Cannon \cite{Cannon} for groups of hyperbolic isometries.

Let $\Gamma$ be a finitely generated group with a fixed finite symmetric generating set $S$.
Let $| \mathord{\cdot} |$ and $d(\mathord{\cdot} , \mathord{\cdot})$ denote word-length and the word-metric, respectively.

The \emph{cone type} of an element $\gamma \in \Gamma$ is defined as:
$$
C^+(\gamma) \coloneqq \big\{ \eta \in \Gamma \st |\eta\gamma|  =  |\eta| + |\gamma| \big\} \, .
$$
For example, $C^+(\gamma)  =  \Gamma$ if and only if $\gamma  =  \id$.

% Geometrical interpretation: \marg{This terminology was used in a very old version of the paper, but then it disappeared. Shall we keep this remark?}
% Suppose $I  =  \{ \eta_i \}_{i = n_0}^{n_1}$ and $J  =  \{ \gamma_i \}_{i = n_1}^{n_2}$ are geodesic segments whose concatenation $IJ$ is well-defined (i.e., $\eta_{n_1}  =  \gamma_{n_1}$); then $IJ$ is a geodesic segment iff $\eta_{n_0}^{-1}\eta_{n_1} \in C^+(\gamma_{n_1}^{-1}\gamma_{n_2})$.

\begin{remark}\label{r.cone_types1}
Actually the usual definition is different:
$$
C^-(\gamma) \coloneqq \big\{ \eta \in \Gamma \st |\gamma\eta|  =  |\eta| + |\gamma| \big\} \, .
$$
But working with one definition is essentially equivalent to working with the other, because
$C^-(\gamma)  =  [C^+(\gamma^{-1})]^{-1}$.
If we need to distinguish between the two, we shall call them \emph{positive} and \emph{negative} cone types.
\end{remark}

A fundamental fact is that word-hyperbolic groups have only finitely many cone types: see \cite[p.~455]{BH} or \cite[p.~145]{CDP}.
\sout{In fact, there is a constant $k$ (depending only on the hyperbolicity constant of the group) such that for any $\gamma\in \Gamma$, the cone type $C^+(\gamma)$ is uniquely determined by the $k$-prefix}\footnote{\sout{Or $k$-suffix in the case of negative cone types.}} \sout{of a shortest word representation of $\gamma$.}

Given a cone type $C$ and $a \in S \cap C$, we can define a cone type 
$$
aC \coloneqq C^+(a\gamma),
$$ 
where $\gamma \in \Gamma$ is such that $C^+(\gamma)  =  C$.

\begin{lem}\label{l.image}
$aC$ is well-defined.
\end{lem}

Though the lemma is contained in \cite[p.~147, Lemme~4.3]{CDP}, let us provide a proof for the reader's convenience:

\begin{proof}
Suppose that $C^+(\gamma)  =  C^+(\gamma')  =  C$ and $a \in S \cap C$; we need to prove that $C^+(a\gamma) \subset C^+(a\gamma')$.
Take $\eta \in C^+(a\gamma)$, so $|\eta a \gamma|  =  |\eta| + |a\gamma|  =  |\eta| + 1 + |\gamma|$.
In follows that $|\eta a|  =  |\eta| + 1$ and so $|\eta a \gamma|  =  |\eta a| + |\gamma|$, that is $\eta a \in C$.
In particular $|\eta a \gamma'|  =  |\eta a | + |\gamma'|  =  |\eta| + |a\gamma|$, proving that $\eta \in C^+(a\gamma')$.
\end{proof}

We associate to $(\Gamma,S)$ a labeled graph $\cG$ (in the sense of \S~\ref{ss.sofic})
called the \emph{geodesic automaton} and 
defined as follows:
\begin{itemize}
\item the vertices are the cone types of $\Gamma$; 
\item there is an edge $C_1 \xrightarrow{a} C_2$ from vertex $C_1$ to vertex $C_2$, labeled by a generator $a \in S$, 
iff  $a\in C_1$ and $C_2  =  a C_1$.
\end{itemize}

\begin{remark}\label{r.cone_types2}
Replacing each vertex $C$ by $C^{-1}$ (a negative cone type)
and each edge $C_1 \xrightarrow{a} C_2$ by $C_1^{-1} \xrightarrow{a^{-1}} C_2^{-1}$,
we obtain the graph described in \cite[p.~456]{BH}. 
\end{remark}

Let us explain the relation with geodesics.
Consider a \emph{geodesic segment} $(\gamma_0, \gamma_1, \dots, \gamma_\ell)$
that is, a sequence of elements of $\Gamma$ such that $d(\gamma_n,\gamma_m)  =  |n-m|$,
and assume that $\gamma_0  =  \id$.
Then, there are $a_0$, \dots $a_{\ell-1}$ in a generating set $S$ such that 
$\gamma_n  =  a_0 a_1 \cdots a_{n-1}$.
Note that for each $n$, the following is an edge of the geodesic automaton graph~$\cG$:
$$
C^+(\gamma_n^{-1}) \xrightarrow{a_n^{-1}} C^+(\gamma_{n+1}^{-1}) 
$$
Thus we obtain a (finite) walk on $\cG$ starting from the vertex $C^+(\id)$.
Conversely, for each such walk we may associate a geodesic segment starting at the identity.

Let us define also the \emph{recurrent geodesic automaton}
as the maximal recurrent subgraph $\cG^*$ of $\cG$;
its vertices are called \emph{recurrent cone types}.
Similarly to what was explained above, for each two-sided geodesic on $\Gamma$ passing through the identity element we can associate a bi-infinite walk on $\cG^*$, and vice-versa.  

Using the fact that the geodesic automation is a finite graph, it is simple to obtain the following property:\footnote{Alternatively, one can deduce the lemma from the ``bounded dead-end depth'' property \cite{Bogo}.} %\marg{R: Guichard pregunta si de esto se recupera el flujo geod\'esico. Yo tengo la impresi\'on que quizas se puede, pero no se hacerlo. No se si corresponde poner algo aqu\'i. Supongo que no.\\J: Tampoco lo s\'e. Dejemos as\'i.} 

\begin{lem}\label{l.dense_geodesics}
Let $\Gamma$ be an infinite word-hyperbolic group, with a fixed set of generators.
Let $\Gamma^*$ be the union of all two-sided geodesics passing through the identity element.
Then for some finite $c$, the set $\Gamma^*$ is $c$-dense on $\Gamma$, that is,
for every $\gamma \in \Gamma$ there exists $\eta \in \Gamma^*$ such that $d(\eta,\gamma) \le c$.
\end{lem}

% \begin{proof}
% It was proved by Bogopolski~\cite{Bogo} that infinite word-hyperbolic groups have the following ``bounded dead-end depth'' property:
% $$
% \exists \ell > 0 \ \forall \gamma \in \Gamma \ \exists \eta \in \Gamma \text{ such that }
% |\eta| \le \ell \text{ and } |\gamma \eta|  =  |\gamma| + 1 \, .
% $$
% It follows immediately that for appropriate constants $\lambda$, $c$,
% the union of all $(\lambda,c)$-quasi-geodesic rays\footnote{See e.g.\ \cite[Chapitre~5]{GhysdelaHarpe} for the definition of quasi-geodesics and the Morse Lemma.}
% from the identity element is the whole $\Gamma$.
% By the Morse Lemma, we conclude that for appropriate $k$, the union of all geodesic rays from the identity element is $c$-dense in $G$.
% It follows that the same is true for the union of the two-sided geodesics passing through the identity element.
% \end{proof}

\subsection{Multicones for dominated representations}\label{ss.multicone_representation}

Let $\Gamma$ be a word-hyperbolic group (with a fixed finite symmetric generating set),
and let $\rho \colon \Gamma \to \GL(d,\R)$ be a representation.

Recall from \S~\ref{ss.sofic} that a multicone of index $p$ is an open subset of $\P(\R^d)$ that contains the projectivization of some $p$-plane and does not intersect the projectivization of some $(d-p)$-plane.

If for each recurrent cone type $C$ it is given a multicone $M(C)$ of index $p$,
then we say that $\{M(C)\}_C$ is a \emph{family of multicones} for $\Gamma$.
We say that this family is \emph{strictly invariant} with respect to $\rho$ if for each edge $C_1 \xrightarrow{g} C_2$ of the geodesic automaton graph, we have:
$$
\rho(g) (M(C_1)) \Subset M(C_2) \, .
$$

\begin{teo}\label{t.multicone_representation}
A representation of a word-hyperbolic group is $p$-dominated if and only if it has a strictly invariant family of multicones of index $p$. 
Moreover, we can always choose a family composed of tame multicones.
\end{teo}

\begin{proof} %[Proof of Theorem~\ref{t.multicone_representation}]
Fix a word-hyperbolic group $\Gamma$.
We can assume that $\Gamma$ is infinite, otherwise the theorem is vacuously true.
Fix a finite symmetric generating set $S$.
Consider the associated recurrent geodesic automaton $\cG^*$,
and the sofic shift $T \colon \Lambda \to \Lambda$ presented by this labeled graph.
Then a sequence $(a_n)$ in $S^\Z$ belongs to $\Lambda$ if and only if 
% for all $m>n$ in $\Z$, the product $\gamma \coloneqq a_{m-1} a_{m-2} \cdots a_n$ has length $|\gamma|  =  m-n$.
the sequence $(\gamma_n)_{n\in \Z}$ defined by:
$$
\gamma_n \coloneqq 
\begin{cases}
	a_0^{-1} a_1^{-1} \cdots a_{n-1}^{-1} &\quad\text{if $n>0$,} \\
	\id &\quad\text{if $n = 0$,}\\
	a_{-1} a_{-2} \cdots a_n &\quad\text{if $n<0$,}
\end{cases}
$$
is a geodesic on $\Gamma$.
The union of (the traces of) all such geodesics is a set $\Gamma^* \subset \Gamma$ which by Lemma~\ref{l.dense_geodesics} is $c$-dense in $\Gamma$ for some finite $c$.

Let $\rho \colon \Gamma \to \GL(d,\R)$ be a representation. 
Consider the family of matrices $(\rho(a))_{a\in S}$, and let $(T,A)$ be the induced sofic linear cocyle. Then, for each $x  =  (a_n) \in X$, if $(\gamma_n)$ is the geodesic defined above then 
$A(T^{n-1}x) \cdots A(x)  =  \rho(\gamma_n^{-1})$ for every positive $n$.

By Remark~\ref{rem-pdominatedimpliesd-p}, $\rho$ is $p$-dominated iff it is $(d-p)$-dominated.
Since the set $\Gamma^*$ is $c$-dense, $\rho$ is $(d-p)$-dominated iff:
$$
\exists C>0 \ \exists \lambda>0 \ \forall \gamma \in \Gamma^* \text{ we have }
\frac{\sigma_{p+1}}{\sigma_p} (\rho(\gamma^{-1})) \le C e^{-\lambda|\gamma|} \, .
$$
Note that this condition holds iff the sofic linear cocycle $(T,A)$ has a dominated splitting with a dominating bundle of dimension $p$; this follows from Theorem~\ref{t.BG} and the previous observations.
Also note that a strictly invariant family of multicones for the sofic linear cocycle $(T,A)$
is exactly the same as a strictly invariant family of multicones for the representation $\rho$. Therefore Theorem~\ref{t.multicone_sofic} allows us to conclude.
\end{proof}

Theorem~\ref{t.multicone_representation} also has the following well-known consequence (see \cite{GuichardWienhard,Labourie-AnosovFlows}). 

\begin{cor}\label{cor-openproperty}
Among representations $\Gamma \to \GL(d,\R)$, the $p$-dominated ones form an open set. 
\end{cor}

\begin{proof}
Given a $p$-dominated representation, take a strictly invariant family of multicones.
Then the same family is also strictly invariant under all nearby representations.
\end{proof}

As a complement to Theorem~\ref{t.multicone_representation}, let us explain how to determine the equivariant map $\xi \colon \partial \Gamma \to \Gr_p(\R^d)$ (defined in Section~\ref{s.equiv})
in terms of multicones.

\begin{prop}\label{p.multicone_representation_xi}
Consider a $p$-dominated representation $\rho \colon \Gamma \to \GL(d,\R)$,
and let $\{M(C)\}_C$ be a strictly invariant family of multicones.
Consider any geodesic ray $(\gamma_n)_{n \in \N}$ with $\gamma_0  =  \id$,
and let $x \in \partial \Gamma$ be the associated boundary point.
Let $(P_n)_{n\ge 1}$ be a sequence in $\Gr_p(\R^d)$ such that for each $n$, the projectivization of $P_n$ is contained in the multicone $M(C^+(\gamma_n))$.
Then
\begin{equation}\label{e.xi_multicone}
\xi(x)  =  \lim_{n \to \infty} \rho(\gamma_n^{-1}) P_n \, .
\end{equation} 
Moreover, the speed of convergence can be estimated independently of $x$, $(\gamma_n)$, and $(P_n)$, and is the same for all nearby representations.
\end{prop}

\begin{proof}
It suffices to translate Proposition~\ref{p.multicone_sofic_Ecu} to the context of representations.
\end{proof}

%%%%%%%%%%%%%%%%%%% 

\begin{obs}
Recall from Remark~\ref{rem-pgl} that it also makes sense to define $p$-dominated representations into $\PGL(d,\R)$.
All that was said in this subsection applies verbatim to that case,
since $\PGL(d,\R)$ also acts on $\P(\R^d)$.
\end{obs}

%%%%%%%%%%%%%%%%%%%%%%%%%

\subsection{An example}

Consider the free product $\Gamma \coloneqq (\nicefrac{\Z}{3\Z}) * (\nicefrac{\Z}{2\Z})$ (which is isomorphic to $\PGL(2,\Z)$) with a presentation: 
$$
\Gamma  =  \big\langle a,b \mid a^3  =  b^2  =  \id \big\rangle.
$$
Then there are only $2$ recurrent cone types, namely $C^+(a)$ and $C^+(b)$,
and the recurrent geodesic automaton is:
\begin{center}
	\includegraphics{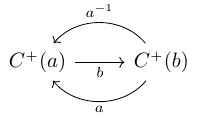}
\end{center}
Fix $\lambda>1$. 
Consider the following pair of matrices in $\SL(2,\R)$:
$$
A \coloneqq D^{-1} R_{\pi/3} D \quad \text{and} \quad
B \coloneqq R_{\pi/2} \, , 
$$
where:
$$
D \coloneqq \begin{pmatrix} \lambda & 0 \\ 0 & \lambda^{-1} \end{pmatrix} 
\quad \text{and} \quad
R_\theta \coloneqq \begin{pmatrix} \cos \theta & -\sin \theta \\ \sin \theta & \cos \theta \end{pmatrix}  
\, .
$$
Since $A^3  =  B^2  =  -\Id$, we can define a representation $\rho \colon \Gamma \to \PSL(2,\R)$ by setting
$\rho(a) \coloneqq A$, $\rho(b) \coloneqq B$.
We claim that if $\lambda$ is sufficiently large (namely, $\lambda>\sqrt[4]{2}$) then this representation is dominated.
Indeed, it is possible to find a strictly invariant family of multicones as in Fig.~\ref{f.modular_multicones}. %\marg{R: Comentario que no entiendo.\\J: Creo que \'el quiere decir que una representacion $\rho \colon \Gamma \to \PSL(2,\R)$ cualquiera es conjugada a una que tiene nuestra formula. Quizas incluso sepamos clasificar todas las representaciones; puede que la condicion $\lambda>\sqrt[4]{2}$ sea ademas necesaria para dominacion. Pero el paper ya esta demasiado largo, y ya hicimos demasiados cambios.} 

\begin{figure}[!hbt]
\begin{center}  
\begin{tikzpicture}[scale = 1.8]
	\def\angA{22}
	\def\angB{29.74}  % see spreadsheet modular.numbers for the calculations
	\def\angC{75.48}
	\def\angD{83.59}

	% cone I
	\fill[fill = gray!30] (0,0)--({cos(\angB)},{sin(\angB)}) node[above]{$I$} arc [start angle = {\angB}, end angle = {-\angB}, radius = 1]--cycle;
	\fill[fill = gray!30] (0,0)--({cos(180+\angB)},{sin(180+\angB)}) arc [start angle = {180+\angB}, end angle = {180-\angB}, radius = 1]--cycle;
	
	% cone B(J)
	\fill[fill = gray!80] (0,0)--({1.2*cos(\angA)},{1.2*sin(\angA)}) arc [start angle = {\angA}, end angle = {-\angA}, radius = 1.2]--cycle;
	\fill[fill = gray!80] (0,0)--({1.2*cos(180+\angA)},{1.2*sin(180+\angA)}) arc [start angle = {180+\angA}, end angle = {180-\angA}, radius = 1.2]--cycle;
	\node[above right] at (1.2,0) {$B(J)$};
	
	% cone J
	\fill[fill = gray!30] (0,0)--({cos(90-\angA)},{sin(90-\angA)}) node[right]{$J$} arc [start angle = {90-\angA}, end angle = {90+\angA}, radius = 1]--cycle;
	\fill[fill = gray!30] (0,0)--({cos(270-\angA)},{sin(270-\angA)}) arc [start angle = {270-\angA}, end angle = {270+\angA}, radius = 1]--cycle;
	
	% cone A(I)
	\fill[fill = gray!80] (0,0)--({1.2*cos(\angC)},{1.2*sin(\angC)}) node[above]{$A(I)$} arc [start angle = {\angC}, end angle = {\angD}, radius = 1.2]--cycle;
	\fill[fill = gray!80] (0,0)--({1.2*cos(180+\angC)},{1.2*sin(180+\angC)}) arc [start angle = {180+\angC}, end angle = {180+\angD}, radius = 1.2]--cycle;

	% cone A^{-1}(I)	
	\fill[fill = gray!80] (0,0)--({1.2*cos(180-\angC)},{1.2*sin(180-\angC)}) node[above]{$A^{-1}(I)$} arc [start angle = {180-\angC}, end angle = {180-\angD}, radius = 1.2]--cycle;
	\fill[fill = gray!80] (0,0)--({1.2*cos(-\angC)},{1.2*sin(-\angC)}) arc [start angle = {-\angC}, end angle = {-\angD}, radius = 1.2]--cycle;

	% coordinate axes
	\draw(-1.3,0)--(1.3,0);
	\draw(0,-1.15)--(0,1.15);
	
\end{tikzpicture}
\end{center}
\caption{A strictly invariant multicone for the representation $\rho \colon (\nicefrac{\Z}{3\Z}) * (\nicefrac{\Z}{2\Z}) \to \PSL(2,\R)$. We have $A(I) \Subset J$, $A^{-1}(I) \Subset J$, $B(J) \Subset I$.} % In this example, $\lambda = 1.8$.
\label{f.modular_multicones}
\end{figure}

%%%%%%%%%%%%%%%%%%%%%%%%%%%%%%%%%%%%%%%%%%%%%%%%%%%%%%%%%%%%

\section{Analytic variation of limit maps}\label{s.analytic}

The purpose of this section is to give another proof of a theorem from \cite{BCLS}, which establishes that the equivariant limit maps $\xi$, $\eta$ (defined in Section~\ref{s.equiv}) depend analytically on the representation.
This fact is useful to show that some quantities such as entropy vary analytically with respect to the representation, which in turn is important to obtain certain rigidity results (see \cite{BCLS,PotSam}). 

While the original approach of \cite{BCLS} used the formalism of \cite{HPS}, we present here a more direct proof, based on the tools discussed in the previous section.
We remark that an alternative approach using the implicit function theorem and avoiding complexification can be found in \cite{Ruelle}, though the context is different and the results are not exactly the same. 

\medskip

A family of representations $\{\rho_u\}_{u \in D}$ of a word-hyperbolic group $\Gamma$ into $\GL(d,\RR)$ is called \emph{real analytic} if the parameter space is a real analytic manifold,
and for each $\gamma \in \Gamma$, the map $u \in D \mapsto \rho_u(\gamma) \in \GL(d,\R)$ is real analytic. (Of course, it suffices to check the later condition for a set of generators.)

The boundary $\partial \Gamma$ of the group $\Gamma$ admits a distance function within a ``canonical''  H\"older class (\cite[Chapitre~11]{CDP}). For a fixed metric in this class, the geodesic flow defined previously is Lipschitz. The limit maps $\xi_{\rho} \colon \partial \Gamma \to \Gr_p(\RR^d)$ and $\theta_{\rho}\colon \partial \Gamma \to \Gr_{d-p}(\RR^d)$ of a $p$-dominated representations are well known to be H\"older continuous (see Theorem~\ref{teo-holder}). %\marg{R: Saque una frase a sugerencia de Guichard.\\J: Creo que \'el pidio sacar solo la primera afirmacion y no las siguientes; las resucit\'e.}
%Note that, since $\Gr_p(\RR^d)$ is an analytic manifold, 
One can endow the space $C^\alpha(\partial \Gamma, \Gr_p(\RR^d))$ of $\alpha$-H\"older maps with a Banach manifold structure; so analyticity of a map from an analytic manifold to $C^{\alpha}(\partial \Gamma, \Gr_p(\RR^d))$ makes sense.

\begin{teo}[Theorem 6.1 of \cite{BCLS}]\label{teo.analytic} 
Let $\big\{\rho_u \colon \Gamma \to \GL(d,\RR) \big\}_{u \in D}$ be a real analytic family of representations.
Suppose that $0 \in D$ and $\rho_0$ is $p$-dominated.
Then there exists a neighborhood $D' \en D$ of $0$ such that for every $u \in D'$, the representation $\rho_u$ is $p$-dominated, and moreover $u \mapsto \xi_{\rho_u}$ defines a real-analytic map from $D'$ to $C^\alpha(\partial \Gamma, \Gr_p(\RR^d))$, for some $\alpha>0$.
\end{teo}

Let us provide a proof.
Corollary~\ref{cor-openproperty} ensures that for every $u$ sufficiently close to $0$, the representation $\rho_u$ is also $p$-dominated. 

As in \cite{BCLS} (see also \cite[Proposition~A.5.9]{Hubbard}), it is enough to show \emph{transverse real analyticity}.
More precisely, we need to show that the map $F\colon D' \times \partial \Gamma \to \Gr_p(\RR^d)$ given by $F(u,x) \coloneqq \xi_{\rho_u}(x)$ has the following properties: 
\begin{enumerate}
\item\label{i.rem_1} it is $\alpha$-H\"older continuous;
\item\label{i.rem_2} for every $x\in \partial \Gamma$, the map $F(\cdot, x) \colon D' \to \Gr_p(\RR^d)$ is real analytic.
\end{enumerate}

H\"older continuity is a standard property of dominated splittings (see for example \cite[Section 4.4]{CroPo}), and is independent of the analyticity of the family. For completeness, and since we could not find the specific statement in the literature, we included a sketch of the proof in the appendix: see Corollary \ref{cor-holder}. 
In the case at hand, it follows that property~(\ref{i.rem_1}) above is satisfied,
for some neighborhood $D' \ni 0$ and some  uniform H\"older exponent $\alpha>0$.

So we are left to prove the analyticity property~(\ref{i.rem_2}). 
To proceed further, we consider the complexification of the representations. 
We can assume without loss of generality that $D'$ is a neighborhood of $0$ in some $\RR^k$.
Let $\{g_1, \ldots, g_m\}$ be a finite generating set of $\Gamma$.
For each $i$, the Taylor expansion of $u \mapsto \rho_u(g_i)$ around $0$ converges on a polydisk $\hat{D}_i$ in $\C^k$ centered at $0$; we keep the same symbol for the extended map.
Take a smaller polydisk $D \subset \bigcap_i \hat{D}_i$ also centered at $0$ such that for $u\in D$, the complex matrix $\rho_u(g_i)$ is invertible.
Since these maps are analytic, and every relation of $\Gamma$ is obeyed when $u$ is real, the same happens for all $u \in \hat D$. So we have constructed a complex analytic family of representations $\rho_u \colon \Gamma \to  \GL(d,\mathbb{C})$, where $u$ takes values in a neighborhood $\hat{D}$ of $0$ in $\C^k$.

Recall from Remark~\ref{r.dom_complex} that dominated splittings make sense in the complex case; so do dominated representations, with the exact same definition \eqref{eq:domrep}.
Actually, if $\iota$ denotes the usual homomorphism that embeds $\GL(d,\C)$ into $\GL(2d,\R)$,
then a representation $\rho \colon \Gamma \to \GL(d,\C)$ is $p$-dominated in the complex sense iff $\iota \circ \rho$ is $2p$-dominated in the real sense.
In particular, the openness property of Corollary~\ref{cor-openproperty} also holds in the complex case.
So, reducing $\hat{D}$ if necessary, we assume that each $\rho_u$ is $p$-dominated.

\begin{prop}\label{prop.pointwiseanalytic}
For each $x \in \partial \Gamma$, the map $u \mapsto \xi_{u}(x)$
from $\hat{D}$ to $\Gr_p(\C^d)$ is complex analytic.

\end{prop}

\begin{proof}
It suffices to check analyticity around $u = 0$.
For each $x \in \partial\Gamma$, choose a geodesic ray $(\gamma^x_n)_{n\in \N}$ in $\Gamma$ 
starting at the identity element and converging to $x$.
We consider the sequence of maps from $\hat D$ to $\Gr_p(\mathbb{C}^d)$ defined as follows:
$$
\varphi^x_n(u) \coloneqq \rho_u(\gamma^x_n)^{-1}  (\xi_0 (\gamma^x_n (x))) \, .
$$
Each of these maps is complex analytic, since so is $u \mapsto \rho_u(\gamma^x_n)$.

We claim that for each $x \in \partial \Gamma$, convergence $\varphi^x_n(u) \to \xi_u(x)$ holds uniformly in a neighborhood of $0$.
Since $\Gr_p(\mathbb{C}^d)$ may be considered as a closed subset of 
$\Gr_{2p}(\R^{2d})$, it is sufficient to check convergence in the latter set.
This convergence follows from Proposition~\ref{p.multicone_representation_xi}  and the fact that for $u$ in a small neighborhood $V$ of $0$, the representations are still dominated, and with the same multicones (see Corollary~\ref{cor-openproperty}). Note that $\xi_0$ always belongs to the corresponding multicone because the multicones remain unchanged.

Being a uniform limit of complex-analytic maps, the map $u \mapsto \xi_u(x)$ is complex analytic on $V$ (see for example \cite[Corollary 2.2.4]{Hormander}).
\end{proof}

Restricting to the real parameters, the proposition yields the property~(\ref{i.rem_2}) that we were left to check.
So the proof of Theorem~\ref{teo.analytic} is complete.

%%%%%%%%%%%%%%%%%%%%%%%%%%%%%%%%%%%%%%%%%%%%%%%%%%%%%

\section{Geometric consequences of Theorem \ref{t.BG}: A Morse Lemma for \texorpdfstring{$\PSL(d,\R)$}{PSL(d,R)}'s symmetric space} \label{s.morse}

In this section we explain why Theorem~\ref{t.BG} (and more precisely Proposition \ref{p.BG_sequences}) has a deep geometric meaning for the symmetric space of $\PSL(d,\R)$. This is a version of the Morse Lemma recently proved by \cite{KLP2}. Because of this application, one is tempted to call Theorem~\ref{t.BG} a \emph{twisted Morse Lemma}\footnote{Lenz \cite{Lenz}, who previously obtained a weaker version of Theorem \ref{t.BG} for $\SL(2,\R)$, had already noted that his result was related to the classical Morse Lemma for the hyperbolic plane.}.

The exposition is purposely pedestrian for two reasons: (1) it is intended for the reader unfamiliar with symmetric spaces; (2) it mimics, in the case of $\PSL(d,\R)$, the general structure theory of semi-simple Lie groups, in order to ease the way to Section \ref{s.general}.

The reader familiar with semi-simple Lie groups should jump to subsection \ref{angyGromov}, or even Section \ref{s.general}, for a proof of the Morse Lemma due to \cite{KLP2}, for symmetric spaces of non-compact type, using dominated splittings. Specific references for subsections \ref{Cartan}--\ref{a+} are, for example, \cite{Eberlein}, \cite{bordeF}, \cite{Helgason}, \cite {lang}. 

This section is independent of sections \ref{s.dominationimplieshyp}--\ref{s.analytic}.

\subsection{A Cartan subalgebra}\label{Cartan}

Fix an inner product $\<\mathord{\cdot},\mathord{\cdot}\>$ on $\R^d$ and denote by $o$ its homothety class, i.e.\ $\<\mathord{\cdot},\mathord{\cdot}\>$ up to positive scalars. One has then the adjoint involution $T\mapsto T^\texttt{t}$ defined by $\<Tv,w\> = \<v,T^\texttt{t}w\>$ (note that ${}^\texttt{t}$ only depends on $o$). This involution splits the vector space $\frak{sl}(d,\R)$ of traceless $d\times d$ matrices as 
$$
\frak{sl}(d,\R) = \frak p^o\oplus\frak k^o
$$
where 
$$
\frak p^o = \{T\in\frak{sl}(d,\R):T^\texttt{t} = T\}\textrm{ and }\frak k^o = \{T\in\frak{sl}(d,\R):T^\texttt{t} = -T\}.
$$

The subspace $\frak k^o$ is a Lie algebra so we can consider its associated Lie group $K^o = \exp\frak k^o$, consisting on (the projectivisation of) determinant 1 matrices preserving the class $o$. The subspace $\frak p^o$, of traceless matrices diagonalizable on a $o$-orthogonal basis, is not a Lie algebra. 

Fix then a $o$-orthogonal set of $d$ lines $\cal E$ and denote by $\frak a\subset\frak p^o$ those matrices diagonalizable in the chosen set $\cal E$. This is an abelian algebra, called a \emph{Cartan subalgebra} of $\frak{sl}(d,\R);$ its associated Lie group $\exp \frak a$ consists on (the projectivisation of) determinant $1$ matrices diagonalizable on $\cal E$ with positive eigenvalues. For $a\in\frak a$ and $u\in\cE$, we will denote by $\lambda_u(a)$ the eigenvalue for $a$ associated to the eigenline $u$. Note that $\lambda_u$ is linear on $a$ and hence an element of the dual space $\frak a^*$ of $\frak a$.

\subsection{The action of \texorpdfstring{$\frak a$}{a} on \texorpdfstring{$\frak{sl}(d,\R)$}{sl(d,R)}}

The action of $\frak a$ on $\frak{sl}(d,\R)$ given by $(a,T)\mapsto [a,T] = aT-Ta$ is also diagonalizable. Indeed, the set of (projective) traceless transformations $\{\epsilon_{uv},\phi_{uv} \st u\neq v\in\cE\}$ defined by 
$$
\epsilon_{uv}(v) = u\quad\textrm{and}\quad\epsilon_{uv}|\cE-\{v\} = 0,
$$
and 
$$
\phi_{uv}|u = t\id, \quad \phi_{uv}|v = -t\id, \quad\textrm{and}\quad \phi_{uv}|\cE-\{u,v\} = 0,
$$
contains a linearly independent set\footnote{Redundancy only appears in the set $\{\phi_{uv} \st u\neq v\in\cE\}$.} of eigenlines of $[a,\mathord{\cdot}]$, specifically $[a,\phi_{uv}] = 0$ and 
$$
[a,\epsilon_{uv}] = \alpha_{uv}(a)\epsilon_{uv} = (\lambda_u(a)-\lambda_v(a))\epsilon_{uv}.
$$
The set of functionals 
$$
\E = \{\alpha_{uv}\in\frak a^*:u\neq v\in\cE\}
$$
is called a \emph{root system} (or simply the \emph{roots}) of $\frak a$. For $\alpha\in\E\cup\{0\}$, one usually denotes by $\frak {sl}(d,\R)_\alpha$ the eigenspace associated to $\alpha$, that is,
$$
\frak {sl}(d,\R)_\alpha = \{T\in\frak{sl}(d,\R) \st [a,T] = \alpha(a)T,\,\forall a\in\frak a\},
$$
and one has\footnote{The fact that $\frak {sl}(d,\R)_0 = \frak a$ is particular of split real algebras, such as $\frak{sl}(d,\R)$.}  
\begin{equation}\label{descomposicion}
\frak{sl}(d,\R) = \bigoplus_{\alpha\in\E\cup\{0\}}\frak {sl}(d,\R)_\alpha = \frak a\oplus\bigoplus_{\alpha\in\E}\frak {sl}(d,\R)_\alpha.
\end{equation}

\subsection{Expansion/contraction}\label{ordenE}

The closure of a connected component of 
$$
\frak a -\bigcup_{\alpha\in\E} \ker\alpha
$$
is called \emph{a closed Weyl chamber}. Fix a closed Weyl chamber and denote it by $\frak a^+;$ this is not a canonical choice: it is equivalent to choosing an order on the set $\cal E$. Indeed, consider the subset of \emph{positive roots} defined by $\frak a^+$: 
$$
\E^+ = \{\alpha\in\E:\alpha|\frak a^+\geq0\},
$$
then one can set $u> v$ if $\alpha_{uv}\in\E^+$.\footnote{In the case of $\PSL(d,\R)$ one usually applies the inverse procedure: `let $\{e_1,\ldots, e_d\}$ be the canonical basis of $\R^d$ and let $\frak a^+$ be the set of (determinant one) diagonal matrices with decreasing eigenvalues'.}

Let $\frak n^+$ be the Lie algebra defined by 
$$
\frak n^+ = \bigoplus_{\alpha\in\E^+}\frak {sl}(d,\R)_\alpha = \bigoplus_{u> v}\epsilon_{uv}.
$$
By definition, $\frak a\oplus\frak n^+$ is the subspace of $\frak{sl}(d,\R)$ which is non-expanded by $\Ad\exp a$, for $a\in \frak a^+$.

\subsection{Simple roots} Observe that the order on $\cE$ defined in subsection \ref{ordenE} is a total order; indeed, the kernel of every $\alpha\in\E$ has empty intersection with the interior of $\frak a^+$, so given distinct $u,v\in\cE$, either $\alpha_{uv}\in\E^+$ or $\alpha_{vu} = -\alpha_{uv}\in\E^+$.

Consider pairwise distinct $u,v,w\in\cE$ such that $u>v>w$. Note that if $a\in\frak a^+\cap\ker\alpha_{uw}$ necessarily 
$$
a\in\ker\alpha_{uv}\cap\ker\alpha_{vw}.
$$
A positive root $\alpha$ such that $\ker\alpha \cap\frak a^+$ has maximal co-dimension is called a \emph{simple root} associated to $\frak a^+$. This corresponds to choosing two successive elements of $\cE$. The set of simple roots is denoted by $\Pi$. Note that this is a basis of $\frak a^*$.

From now on we will denote by $\cE = \{u_1,\ldots,u_d\}$ with $u_p> u_{p+1}$. Then $\frak n^+$ can be interpreted as the space of upper triangular matrices on $\cE$ (with 0's in the diagonal), and denoting by $a_i = \lambda_{u_i}(a)$ one has 
$$
\frak a^+ = \{a\in \frak a:a_1\geq\cdots\geq a_d\}.
$$
We will use $\lambda_{u_i}$ to introduce coordinates in $\frak a:$ if $(a_1,\ldots,a_d)\in\R^d$ are such that $a_1+\cdots+ a_d = 0$ then $(a_1,\ldots,a_d)$ will denote the element $a\in\frak a$ such that $\lambda_{u_i}(a) = a_i$. Finally, given $p\in\{1,\ldots,d-1\}$ we will denote by $\alpha_p$ the simple root 
$$
\alpha_p(a) = \alpha_{u_pu_{p+1}}(a) = a_p-a_{p+1}.
$$

\subsection{Flags}
 
Denote by $N = \exp\frak n^+$ and by $M$ the centralizer of $\exp \frak a$ in $K^o$. The group $M$ consists of (the projectivisation of) diagonal matrices w.r.t.\ $\cE$ with eigenvalues $1$'s and $-1$'s.

The group $P = M\exp \frak a\,N$ is called a Borel subgroup of $\PSL(d,\R)$ and $N$ is called its unipotent radical.

Recall that a \emph{complete flag} on $\R^d$ is a collection of subspaces $E = \{E_p\}_1^{d-1}$ such that $E_p\subset E_{p+1}$ and $\dim E_p = p$. The spaces of complete flags is denoted by $\scr F$. Observe that $\PSL(d,\R)$ acts transitively (i.e.\ has only one orbit) on $\scr F$ and that the group $P$ is the stabilizer of 
$$
\{u_1\oplus\cdots\oplus u_p\}_{p = 1}^{d-1};
$$
thus we obtain an equivariant identification $\scr F = \PSL(d,\R)/P$.

Two complete flags $E$ and $F$ are in \emph{general position} if for all $p = 1,\ldots,d-1$ one has 
$$
E_p\cap F_{d-p} = \{0\}.
$$
Denote by $\posgen$ the space of pairs of flags in general position.

The same procedure applied to the Weyl chamber $-\frak a^+$, provides the group $\wk P$ that stabilizes the complete flag 
$$
\{u_d\oplus\cdots\oplus u_{d-p+1}\}_{p = 1}^{d-1}.
$$
Observe that the flags $\{u_1\oplus\cdots\oplus u_p\}_{p = 1}^{d-1}$ and $\{u_d\oplus\cdots\oplus u_{d-p+1}\}_{p = 1}^{d-1}$ are in general position and that the stabilizer in $\PSL(d,\R)$ of the pair is the group $M\exp\frak a = P\cap\wk P$.

\subsection{Flags and singular value decomposition}\label{procedure}

If $\<\mathord{\cdot},\mathord{\cdot}\>\in o$ is an inner product with induced norm $\|\mathord{\cdot}\|$ on $\R^d$, note that the operator norm of $g\in\GL(d,\R)$, 
$$
\sigma_1^o(g) = \|g\|_o = \sup\left\{\frac{\|gv\|}{\|v\|}:v\in\R^d-\{0\}\right\},
$$
only depends on the homothety class $o.$ 
The same holds for the other singular values, defined in subsection \ref{ss.dom_sing}. Throughout this section, the choice of the class $o$ is important, so we will stress the fact that the singular values depend on $o$ by denoting them as $\sigma_i^o(g)$.

The singular value decomposition provides a map $a:\PSL(d,\R)\to \frak a^+$, called the \emph{Cartan projection}, such that for every $g\in \PSL(d,\R)$ there exist $k_g,l_g\in K^o$ such that 
$$
g = k_g\exp(a(g))l_g.
$$
More precisely, one has 
$$
a(g) = (\log\sigma_1^o(g),\ldots,\log\sigma_d^o(g)).
$$

Recall from Section \ref{ss.dom_sing} that $g$ has a gap of index $p$ if $\sigma_p^o(g)>\sigma_{p+1}^o(g)$. If this is the case then 
\begin{equation}\label{UyK}U^o_p(g) = k_g (u_1\oplus\cdots\oplus u_p),
\end{equation} (note again the dependence on $o$). 

Given $\alpha_p\in\Pi$ denote by $K^o(\{\alpha_p\})$ the stabilizer in $K^o$ of the vector space $u_1\oplus\cdots \oplus u_p$. Moreover, given a subset $\t\subset \Pi$, denote by 
$$
K^o(\t) = \bigcap_{\alpha_p\in\t}K^o(\{\alpha_p\}).
$$

If for some $p\in\{1,\ldots,d\}$ and $g\in\PSL(d,\R)$ one has $a_1(g) = a_p(g)>a_{p+1}(g)$, then any element of $k_gK^o(\{\alpha_p\})$ can be chosen in a Cartan decomposition of $g.$ If all the gaps of $g$ are indexed on a subset $\t\en\Pi$, then $k_g$ is only defined modulo $K^o(\t)$ and one has the \emph{partial flag} 
$$
U^o(g) = \{U_p^o(g):\alpha_p\in\t\}.
$$

Note that the Cartan projection of $g^{-1}$ is simply $a(g^{-1}) = (-a_d(g),\ldots,-a_1(g))$. The linear transformation $\ii:\frak a\to \frak a$ defined by 
$$
\ii(a_1,\ldots,a_d) = (-a_d,\ldots,-a_1)
$$
is called \emph{the opposition involution}. If $g$ has gaps indexed by $\t$ then $g^{-1}$ has gaps indexed by $\ii\t = \{\alpha\circ\ii:\alpha\in\t\}$. Denote by $S^o(g) = U^o(g^{-1})$.

\subsection{The symmetric space}

Recall that fixing a class $o$ defines a splitting 
$$
\frak{sl}(d,\R) = \frak p^o\oplus\frak k^o.
$$
This splitting is orthogonal with respect to the \emph{Killing form}, the symmetric bilinear form $\kappa$ on $\frak{sl}(d,\R)$ defined by 
$$
\kappa(A,B) = 2d\traza(AB).
$$
This linear form is related to adjoint involution ${}^\texttt{t}$ in the following sense: the linear form $\kappa(\cdot,\cdot^{\texttt{t}})$ is positive definite.

Since $\frak p^o$ consists on fixed point for ${}^{\texttt{t}}$, the restriction of $\kappa$ to $\frak p^o$, denoted by $(\mathord{\cdot},\mathord{\cdot})_o:\frak p^o\times\frak p^o\to\R$, is positive definite. Explicitly, if $v\in\frak p^o$ then $v$ is diagonalizable and 
$$
|v|_o^2 \coloneqq (v,v)_o
$$
equals the sum of squared eigenvalues of $v$.

The space 
$$
X_d = \{\textrm{inner products on $\R^d$}\}/\R_+
$$
is a contractible $\PSL(d,\R)$-homogenous space, the action being given by 
$$
g\cdot\<\mathord{\cdot},\mathord{\cdot}\> = \langle \mathord{\cdot}, \mathord{\cdot} \rangle' \quad \text{where } \langle v, v \rangle' = \langle g^{-1} v, g^{-1} w \rangle.
$$
The stabilizer of $o$ is the group $K^o$ and thus the orbit map $(g,o)\mapsto g\cdot o$ identifies the tangent space ${\sf{T}}_o X_d$ with the vector space $\frak p^o$. A direct computation shows that the Riemannian metric $o\mapsto (\mathord{\cdot},\mathord{\cdot})_o$ is $\PSL(d,\R)$-invariant. 

The space $(X_d,(\mathord{\cdot},\mathord{\cdot})_o)$ is, by definition, \emph{the symmetric space} of $\PSL(d,\R)$.

\subsection{Maximal flats} A direct computation shows that the orbit $\exp\frak a\cdot o\subset X_d$ is isometric to $(\frak a,(\mathord{\cdot},\mathord{\cdot})_o)$. Moreover, one can show that $\exp \frak a \cdot o$ is a \emph{maximal totally geodesic flat} (flat as in ``isometric to a Euclidean space'', maximal with respect to dimension; see for example \cite[Section XII.3]{lang}). 
For every $g\in\PSL(d,\R)$ one has 
\begin{equation}\label{distancia}d_{X_d}(o,g\cdot o) = d_{X_d}(o,\exp a(g)\cdot o) = |a(g)|_o = \sqrt{\sum_i (\log \sigma_i^o(g))^2},
\end{equation} where $d_{X_d}$ is the distance on $X_d$ induced by the Riemannian metric $(\mathord{\cdot},\mathord{\cdot})_o$.

The translated orbit $g\exp\frak a\cdot o$ is again a maximal totally geodesic flat (through $g\cdot o$) and hence, all geodesics of $X_d$ are of the form 
$$
t\mapsto g\exp (ta)\cdot o
$$
for a given $g\in\PSL(d,\R)$ and $a \in\frak a^+$.

In other words, a \emph{maximal flat} in $X_d$ consists on fixing a set $\cal L$ of $d$ lines that span $\R^d$ and considering the space of inner products, up to homothety, that make $\cal L$ an orthogonal set.

The following lemma is simple but extremely useful for estimations:

\begin{lem}\label{estimaciones} 
Consider $\varphi\in\frak a^*$ such that $\varphi|\frak a^+-\{0\}>0$. Then there exists $c>1$ such that for all $g\in \PSL(d,\R)$ one has 
$$
\frac 1c\varphi(a(g)) \leq d_{X_d}(o,g\cdot o)\leq c\varphi(a(g)).
$$
\end{lem}

\begin{proof} 
Since $\frak a^+$ is closed and $\ker\varphi\cap\frak a^+ = \{0\}$ the function 
$$
a\mapsto\frac{|a|_o}{\varphi(a)}
$$
is invariant under multiplication by scalars and bounded on $\frak a^+-\{0\}$. Equation (\ref{distancia}) completes the proof.
\end{proof}

For example, $\log\sigma_1^o(g)$ and $-\log\sigma_d^o(g)$ are comparable to $d_{X_d}(o,g\cdot o)$.

\subsection{The Furstenberg boundary and parallel sets}\label{F-Parallel}

A \emph{parametrized flat} is a function $\p:\frak a\to X_d$ of the form 
$$
\p(a) = g\exp a\cdot o
$$
for some $g\in\PSL(d,\R)$. A maximal flat is thus a subset of the form $\p(\frak a)\subset X_d$ for some parametrized flat $\p$.

Observe that $\PSL(d,\R)$ acts transitively on the set of parametrized flats and that the stabilizer of $\p_0:a\mapsto \exp a\cdot o$ is the group $M$. We will hence identify the space of parametrized flats with $\PSL(d,\R)/M$.

Two parametrized flats $\p,{\sf g}$ are \emph{equivalent} if the function $\frak a\to\R$ defined by 
$$
a\mapsto d_{X_d}(\p(a),{\sf g}(a))
$$
is bounded on $\frak a^+$.

The \emph{Furstenberg boundary} of $X_d$ is the space of equivalence classes of pa\-ra\-me\-tri\-zed flats. Note that, \emph{by definition} of $N = \exp \frak n^+$ one has that the distance function 
$$
a\mapsto d_{X_d}(n\exp a\cdot o, \exp a\cdot o)
$$
is bounded on $a\in\frak a^+$ only if $n\in M\exp\frak a\,N = P$\footnote{The $ij$ entry on $\cE$ of $\exp(-ta) n \exp( ta)$ is $\exp(t(a_j-a_i))n_{ij}$. In order to have this entry bounded for all $t>0$ one must have $n_{ij} = 0$ for all $j<i$.}. Thus, the equivalence class of the flat $\p_0$ is $P\cdot\p_0$. Hence, the Furstenberg boundary is $\PSL(d,\R)$-equivariantly identified with the space of complete flags $\scr F = \PSL(d,\R)/P$. 

Given a parametrized flat $\p$ denote by $\zz(\p)\in\scr F$ its equivalence class in the Furstenberg boundary. Also, denote by $\wk\zz(\p)\in\scr F$ the class of the parametrized flat\footnote{This is still a parametrized flat.} 
$$
a\mapsto \p(-a).
$$

This last identification can be seen directly: a parametrized flat $\p$ consists in fixing an \emph{ordered}\footnote{Recall $\frak a^+$ is fixed beforehand.} set $\{\ell_1,\ldots,\ell_d\}$ of $d$ lines that span $\R^d$ and considering all inner products (up to homothety) that make this set an orthogonal set, and the choice of one of these inner products. The associated point ``at infinity'' in the Furstenberg boundary of this parametrized flat is the complete flag 
$$
\zz(\p)_p = \ell_1\oplus\cdots\oplus\ell_p.
$$
Moreover, $\wk\zz(\p)_p = \ell_d\oplus\cdots\oplus\ell_{d-p+1}$.

One easily concludes the following properties:
\begin{itemize}
\item[-] Given $x\in X_d$ and a complete flag $F$ there exists a unique maximal flat $\p(\frak a)$ containing $x$ such that $\zz(\p) = F$: apply the Gram-Schmidt process to flag $F$ and any inner product in the class $x$.
\item[-] Given two flags in general position $(E,F)\in\posgen$ there exists a unique maximal flat $\p(\frak a)$ such that $\zz(\p) = E$ and $\wk\zz(\p) = F:$ it suffices to consider the ordered set $\ell_p = E_p\cap F_{d-p+1}$. 
\end{itemize}
Recalling that $M \exp \frak a = P \cap \wk{P}$, observe that the maps $\wk\zz$ and $\zz$ are exactly the canonical quotient projections 
$$
(\wk\zz,\zz):\PSL(d,\R)/M\to \posgen = \PSL(d,\R)/M\exp \frak a.
$$
Given subsets $\t'\en\t\en\Pi$, denote by $\scr F_\t $ the space of \emph{partial flags of type $\t$} and given $E\in\scr F_\t$ denote by ${E}^{\t'}\in\scr F_{\t'}$ the partial flag of type $\t'$ consisting on forgetting the irrelevant subspaces of $E$. Denote by $\zz_\t(\p) = \zz(\p)^\t$ and by $\wk\zz_\t = \zz_{\ii\t}$.

Given a pair of partial flags in general position $E\in\scr F_\t$ and $F\in\scr F_{\ii\t}$ and a point $x\in X_d$, we define: 
\begin{itemize}
\item[-] The \emph{Weyl cone} $V(x,E)$ determined by $x$ and $E$ is 
$$
\bigcup_\p \p(\frak a^+),
$$
where the union is indexed on all parametrized flats $\p$ with $\p(0) = x$ and $\zz_\theta(\p) = E$. 
\item[-] The \emph{parallel set} $P(F,E)$ determined by $F$ and $E$ is 
$$
\bigcup_\p \p(\frak a),
$$
where the union is indexed on all parametrized flats $\p$ with $\wk\zz_\t(\p) = F$ and $\zz_\t(\p) = E$.
\end{itemize}

\subsection{Parametrized flats through $o$ and $g\cdot o$}\label{ogo} Consider $g\in\PSL(d,\R)$ and $g = k_g\exp a(g)l_g$ a Cartan decomposition of $g$. Observe that the set of $d$ lines 
$$
k_g\cal E = \{k_g u_1,\ldots,k_g u_d\}
$$
is simultaneously $o$-orthogonal and $g\cdot o$-orthogonal. The set of classes of inner products that make this set orthogonal is hence a maximal flat through $o$ and $g\cdot o$.

If $g$ has gaps of certain indices, indexed by $\theta\subset\Pi$, then every element of $k_gK_o(\t)$ can be chosen in a Cartan decomposition of $g$. Thus, the set of maximal flats through $o$ and $g\cdot o$ is the translated $K_o(\theta)$-orbit $k_gK_o(\theta)\cal E$. 
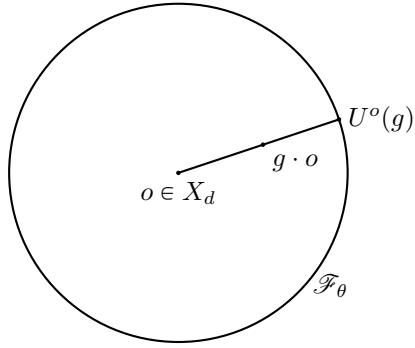
\begin{figure}[hbt]
	\centering
	\begin{tikzpicture}[scale = 0.75]
	\draw [thick] circle [radius = 3];
	\draw [fill] circle [radius = 0.03];	
	\draw [fill] (1.5,0.5) circle [radius = 0.03];
	\node [below] at (0,0) {$o\in X_d$};
	\node [below right] at (1.5,0.5) {$g\cdot o$};
	\draw [thick] (0,0) -- (2.846,0.948);
	\node [right] at (2.846,0.948) {$U^o(g)$};	
	\draw [fill] (2.846,0.948) circle [radius = 0.03];
	\node at (2.7,-2) {$\scr F_\t$};	
	\end{tikzpicture}
	\caption{All flats through $o$ and $g\cdot o$ have $U^o(g)$ as a partial flag at infinity.}\label{figure:U(g)}
\end{figure}

All parametrized flats $\p$ with $\p(0) = o$ such that $g\cdot o\in\p(\frak a^+)$, have the flag $U^o(g)$ as a partial subflag of their corresponding flag at infinity $\zz(\p)$ (recall equation (\ref{UyK})), see Fig.~\ref{figure:U(g)}.

\subsection{$\frak a^+$-valued distance}\label{a+} 

Recall that $X_d$ is $\PSL(d,\R)$-homogeneous and consider the map $\aa:X_d\times X_d\to\frak a^+$ defined by 
$$
\aa(g\cdot o,h\cdot o) = a(g^{-1}h).
$$
Note that $\aa$ is $\PSL(d,\R)$-invariant for the diagonal action of $\PSL(d,\R)$ on $X_d\times X_d$, that 
\begin{equation}\label{distvect}d_{X_d}(x,y) = |\aa(x,y)|_o,
\end{equation} (this is due to equation (\ref{distancia})) and that $\ii(\aa(x,y)) = \aa(y,x)$.

Consider the subset of simple roots defined by 
$$
\theta(x,y) = \{\alpha\in\Pi:\alpha(\aa(x,y))\neq0\}
$$
and consider the partial flag 
$$
U(x,y) = \{gU^o_\alpha(g^{-1}h)\}_{\alpha\in\theta(x,y)},
$$
where $(x,y) = (g\cdot o,h\cdot o)$. Given $\theta\subset\theta(x,y)$, let the \emph{Weyl cone of type $\t$ from $x$ to $y$} be denoted by $V_\t(x,y)$ and defined by 
$$
V_\t(x,y) = \bigcup_\p \p(\frak a^+),
$$
where the union is indexed on all parametrized flats $\p$ with $\p(0) = x$ and $\zz_\t(\p) = U(x,y)^\t$. Finally, the \emph{diamond of type $\t$} between $x$ and $y$ is the subset 
$$
\dd_\t(x,y) = V_\t(x,y)\cap V_{\ii\t}(y,x).
$$
This diamond is contained in the parallel set 
$$
P(U(x,y)^\t,S(x,y)^{\ii\t}),
$$
where $S(x,y) = hS^o(g^{-1}h)$.

For example, consider $x = o$. Any $y \in X_d$ can be written as $y = g \exp a \cdot o$ with $a = \aa(o,y) \in \mathfrak{a}^+$ and $g \cdot o = o$. If $a\in\inte\frak a^+$ then $\t(o,y) = \Pi$ and 
\begin{align*}
V_{\Pi}(o, y) & = \{ g \exp v \cdot o \st v \in \mathfrak{a}^+ \}, \\
\dd_{\Pi}(o, y) & = \{ g \exp v \cdot o \st v \in \mathfrak{a}^+ \cap (a - \mathfrak{a}^+) \}.
\end{align*}

\subsection{Angles and distances to parallel sets}\label{productoGromov}\label{angyGromov} 
The purpose of this subsection and the next one is to relate the distance from a given point $o$ to a parallel set $P(E,F)$, for two partial flags in general position, with the angle between $E$ and $F$ for an inner product in the class $o$ (Proposition \ref{cjtoparalelogeneral}). 

Fix an inner product $\<\mathord{\cdot},\mathord{\cdot}\>\in o$ and denote by $\|\mathord{\cdot} \|$ the induced norm on $\R^d$. The $o$-\emph{angle} between non-zero vectors
$v$, $w\in\R^d$ is defined as the unique number $\angle_o(v,w)$ in $[0,\pi]$
whose cosine is $\langle v,w \rangle / (\|v\| \, \|w\|)$.
If $E$, $F \subset \R^d$ are nonzero subspaces then we define their $o$-\emph{angle} as:

\begin{equation}\label{e.angle}
\angle_o(E,F) \coloneqq \min_{v \in E^\times} \min_{w \in F^\times} \angle_o(v,w) \, ,
\end{equation}
where $E^\times \coloneqq E - \{0\}$.
We also write $\angle_o (v, F)$ instead of $\angle_o(\R v, F)$, if $v$ is a nonzero vector. Observe that $\angle_o(\mathord{\cdot},\mathord{\cdot})$ is independent on $\<\mathord{\cdot},\mathord{\cdot}\>\in o$.

We haven't found a precise reference for the following proposition, we will hence provide a proof. See Fig.~\ref{figure:paralelo}.

\begin{prop}\label{cjtoparalelogeneral} Given $\t\en\Pi$ there exist $c>1$ and $c'>0$, only depending on $\t$ and the group $\PSL(d,\R)$, such that if $(E,F)\in\posgen_\t$, then 
$$
\frac{-1}c\log\sin\min_{\alpha_p\in\t}\angle_o(E_p,F_{d-p}) \leq d_{X_d}(o,P(F,E)) \leq c'-c\log\sin\min_{\alpha_p\in\t}\angle_o(E_p,F_{d-p}).
$$

\end{prop}

\begin{figure}[hbt]
	\centering
	\begin{tikzpicture}[scale = 0.75]
	\draw [thick] circle [radius = 3];
	\draw [fill] circle [radius = 0.03];	
	\node [below] at (0,0) {$o$};
	\node [above] at (0,3) {$E\in\scr F_\t$};
	\draw [fill] (0,3) circle [radius = 0.03];
	\node [right] at (2.846,0.948) {$F\in\scr F_{\ii\t}$};	
	\draw [fill] (2.846,0.948) circle [radius = 0.03];	
	\draw [thick] (2.846,0.948) to [out = 197,in = 270] (0,3);
	\draw [thick] (0,0) -- (0.89,1.37);
	\draw [fill] (0.882,1.35) circle [radius = 0.03];
	\node [right] at (3,-2) {${\displaystyle d_{X_d}(o,P(F,E)) \asymp -\log\min_{\alpha_p\in\t}\{\sin\angle_o(E_p,F_{d-p})\}}$};
	\draw [->] (3,-2) to [out = 100,in = -10] (0.45,0.60);
	\draw [->] (3.5,3) to [out = 180,in = 90] (1.5,1.2);
	\node [right] at (3.5,3) {$P(F,E)$};
	\end{tikzpicture}
	\caption{The statement of Proposition \ref{cjtoparalelogeneral}}\label{figure:paralelo}
\end{figure}
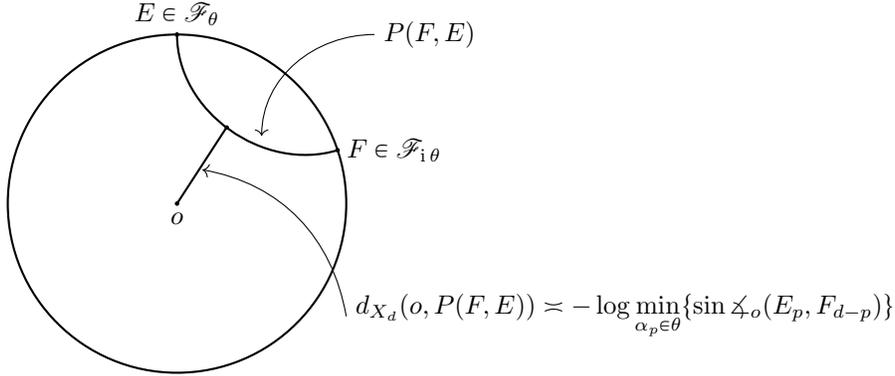

\begin{proof}
Without loss of generality, we may assume that for all $\alpha_p\in\t$ one has 
$$
F_{d-p} = u_{d-p}\oplus\cdots\oplus u_d.
$$ Denote by $G_\cE(E,F)$ the set of elements $g\in \PSL(d,\R)$ such that for all $\alpha_p\in\t$ one has $g(u_1\oplus\cdots\oplus u_p) = E_p$ and $g(F_{d-p}) = F_{d-p}$. Then, the parallel set 
$$
P(F,E) = \{g\cdot o:g\in G_{\cE}(E,F)\}.
$$

Let us prove the first inequality. Consider arbitrary $p$ with $\alpha_p \in \theta$ and arbitrary nonzero vectors $v \in E_p$ and $w \in F_{d-p}$  
Let $\tilde{w}$ be the orthogonal projection of $v$ on the line spanned by $w$. 
Consider arbitrary $g \in G_{\cE}(E,F)$.
Then $g^{-1} v$ and $g^{-1} w$ are orthogonal, which allows us to estimate:
$$
\frac{1}{\sin \angle_o(v,w)}
=   \frac{\|\tilde w\|}{\|v - \tilde{w}\|}
\le \frac{\sigma_1^o(g) \|g^{-1}\tilde w\|}{\sigma_d^o(g) \|g^{-1}v - g^{-1}\tilde{w}\|} 
\le \frac{\sigma_1^o(g)}{\sigma_d^o(g)} \, .
$$
By Lemma~\ref{estimaciones}, the logarithm of the right-hand side is comparable to $d_{X_d}(o,g \cdot o)$. 
So we obtain the first inequality in the proposition.

Next, let us prove first the second inequality.
Write the set $\{p \st \alpha_p \in \t\} \cup \{0,d\}$ as $\{0 = p_0 < p_1<p_2<\dots<p_k = d\}$. Denote by 
$$
H_i^0 = u_{p_{i-1}}\oplus\cdots\oplus u_{p_i}.
$$
Denote by $H_i = E_{p_i}\cap F_{d-p_{i+1}}$. For all $i\in\{1,\ldots,k\}$ one has 
$$
\dim\,  H_1^0\oplus\cdots\oplus H_i^0 = \dim\, H_1\oplus\cdots\oplus H_i.
$$

If for $g\in\PSL(d,\R)$ one has $g(H^0_i) = H_i$ then $g\in G_{\cE}(E,F)$ and hence $g\cdot o\in P(F,E)$. We will define a suited such $g$ and estimate its operator norm for the class~$o$.

Consider $g$ as the unique element of $\PSL(d,\R)$ such that, for each $i = 1,\dots,k$,
the restriction of $g^{-1}$ to $H_i$ coincides with the orthogonal projection on $H^0_i$. 

We proceed to estimate $\|g\|$. Given $v \in \R^d$, write it as $v = v_1 + \cdots + v_k$ with $v_i \in H_i^0$. Then 
$$
\|g v\| \le \sum_{i = 1}^k \|g v_i\| = \sum_{i = 1}^k \frac{\|v_i\|}{\sin \angle_o (g(v_i), H_1^0 \oplus \cdots \oplus H_{i-1}^0)} \, .
$$
For each $i$, orthogonality yields $\|v_i\| \le \|v\|$;
moreover, \begin{multline*} \angle_o (g(v_i), H_1^0 \oplus \cdots \oplus H_{i-1}^0) \\ \ge \angle_o (H_i \oplus \cdots \oplus H_d, H_1^0 \oplus \cdots \oplus H_{i-1}^0) = \angle_o(E_{p_i},F_{d-p_i}) \\ \ge 
\min_{\alpha_p\in\t}\angle_o (E_p,F_{d-p}) \, .\end{multline*} Therefore we obtain 
$$
\|g \| \le \frac{d}{\displaystyle \sin\min_{\alpha_p\in\t} \angle_o(E_p,F_{d-p})} \, .
$$
Taking $\log$, Lemma \ref{estimaciones} yields the second inequality. 
\end{proof}

\subsection{Regular quasi-geodesics and the Morse Lemma of Kapovich--Leeb--Porti} Let $I\en\Z$ be an interval and let $\mu,c$ be positive numbers. A \emph{$(\mu,c)$-quasi-geodesic} is a map $x:I\to X_d$ (also denoted by $\{x_n\}_{n\in I}$) such that for all $n,m\in I$ one has 
$$
\frac 1\mu|n-m|-c\leq d_{X_d}(x_n,x_m)\leq \mu|n-m|+c.
$$

Let $\scr C\en\frak a^+$ be a closed cone. Following \cite{KLP2} we will say that a quasi-geodesic segment $\{x_n\}$ is $\scr C$-\emph{regular} if for all $n<m\in I$ one has $\aa(x_n,x_m)\in \scr C$. %\marg{Note that $\{x_{-n}\}_{n\in -I}$ is $\ii\scr C$-regular. }
Denote by 
$$
\t_\scr C = \{\alpha\in\Pi:\ker\alpha\cap\scr C = \{0\}\}.
$$
We state the following version of the Morse Lemma \cite[Theorem 1.3]{KLP2}, specialized to the symmetric space of $\PSL(d,\R)$:

%\marg{A: saqu\'e la existencia de $\ell$ a sugerencia de guichard}

\begin{teo}[Kapovich--Leeb--Porti \cite{KLP2}]\label{teo:MorseSLd} Let $\mu$, $c$ be positive numbers and $\scr C\en\frak a^+$ a closed cone. Then there exists $C>0$ such that if $\{x_n\}_{n\in I}$ is a $\scr C$-regular $(\mu,c)$-quasi-geodesic segment, then \begin{itemize}\item[-] If $I$ is finite then $\{x_n\}$ is at distance at most $C$ from the diamond 
$$
\dd_{\t_\scr C}(x_{\min I},x_{\max I}).
$$
\item[-] If $I = \N$ then there exists $F\in\scr F_{\t_{\scr C}}$ such that $\{x_n\}$ is contained in a $C$-neighborhood from the Weyl cone $V(x_{\min I},F)$. \item[-] If $I = \Z$ then there exists $(E,F)\in\posgen_{\t_{\scr C}}$ such that $\{x_n\}$ is contained in a $C$-neighborhood from the union $V(z,E)\cup V(z,F)$ for some $z\in P(E,F)$ at uniform distance from $\{x_n\}$. \end{itemize}
\end{teo}

\begin{proof} We can assume that $0\in I$ and that $x_0 = o$. Consider then a sequence $\{h_n\}_{n\in I}\en\PSL(d,\R)$ such that $h_n\cdot o = x_n$. Since $\{x_n\}$ is a quasi-geodesic, equation \ref{distvect} implies 
$$
|a(h_{n+1}^{-1}h_n)|_o = |\aa(h_n\cdot o,h_{n+1}\cdot o)|_o = d_{X_d}(x_n,x_{n+1})\leq \mu+c.
$$
If we denote by $g_n = h_{n+1}^{-1}h_n$, then the last equation implies that $\{g_n\}$ lies in a compact subset of $\PGL(d,\R)$. Moreover, if $m\geq n$ then 
$$
a(g_m\cdots g_n) = a(h_{m+1}^{-1}h_n) = \aa(x_{m+1},x_n)\,\in \ii \scr C.
$$

One has the following:\begin{itemize}\item[1.] The sequence $\{g_n\}_{n\in I}$ is $\ii\alpha$-dominated for all $\alpha\in\t_{\scr C}$: indeed, since $\scr C$ is closed and does not intersect $\ker\alpha-\{0\}$, one has $\ii\scr C\cap\ker\ii\alpha = \{0\}$ and hence there exists $\delta>0$ such that for all $a\in\,\ii \scr C-\{0\}$ and $\alpha\in\t_{\scr C}$ one has 
$$
\ii \alpha(a)> \delta | a|_o.
$$
Thus, one concludes that 
$$
\alpha(a(g_m\cdots g_n)) = \alpha(\aa(x_{m+1},x_n))> \delta |a(x_{m+1},x_n)|_o>(\delta/\mu)|n-m|-\delta c.
$$
This is to say, the sequence $\{g_n\}_{n\in I}$ belongs to the space 
$$
\cD(\mu+c,d-p,\,\delta/\mu,e^{-c\delta},I)
$$
for all $p$ such that $\alpha_p\in\t_{\scr C}$, using the operator norm $\|\,\|_o$ associated to $o$.
\item[2.] Subsection \ref{ogo} implies that the (partial) flag at infinity associated to the Weyl cone $V_{\t_\scr C}(o,x_{m+1})$ is: \begin{itemize}\item[-] $U(o,x_{m+1})^{\t_\scr C} = $  
$$
\{U^o_\alpha(h_{m+1})\}_{\alpha\in\t_{\scr C}} = \{U^o_\alpha(g_0^{-1}\cdots g_m^{-1})\}_{\alpha\in\t_{\scr C}} = \{S^o_{\ii\alpha}(g_m\dots g_0)\}_{\alpha\in\t_\scr C}
$$
for $m\geq0$, \item[-] $U(o,x_{-m})^{\ii\t_\scr C} = \{U^o_{\ii\alpha}(h_{-m})\}_{\alpha\in\t_{\scr C}} = \{U^o_{\ii\alpha}(g_0\cdots g_{-m})\}_{\alpha\in\t_{\scr C}}$ for $m\geq0$. \end{itemize}\end{itemize}

Consider $\ell_1$ given simultaneously by Lemmas \ref{l.seq_splitting} and \ref{l.seq_convergence} for all $p$ such that $\alpha_p\in\t_{\scr C}$ and constants $\mu+c,\, \delta/\mu$, and $e^{-c\delta}$. Assume $[-\ell_1,\ell_1]\en I$ (i.e.\ $I$ is long enough). Item~2 and Lemma \ref{l.seq_splitting} imply the existence of $\delta_0$ such that for all $m$ and $-n$ in $I$ with $m,n\geq \ell_1$ one has 
$$
\angle_o(U(o,x_m)^{\t_\scr C},U(o,x_{-n})^{\ii\t_\scr C})>\delta_0.
$$
Moreover, by Lemma \ref{l.seq_convergence} one has $\angle_o(U(o,x_{\ell_1})^{\t_\scr C},U(o,x_m)^{\t_\scr C})<\eps$ and the same occurs with $U(o,x_{-\ell_1})^{\ii\t_\scr C}$ and $U(o,x_{-n})^{\ii\t_\scr C}.$

Proposition \ref{cjtoparalelogeneral} implies thus that for all $m,-n\in I$ with $n,m\geq\ell_1$ the distance between $o$ and the parallel set $P(U(o,x_{m})^{\t_\scr C},U(o,x_{-n})^{\ii \t_\scr C})$ is bounded above by a number $C$, depending on $\mu,c$, the cone $\scr C$ and \emph{a priori} the point $o$, but independent of the quasi-geodesic through $o.$

Since $X_d$ is $\PSL(d,\R)$-homogeneous and this action is by isometries, one concludes that for any $k$ such that $[k-\ell_1,k+\ell_1]\en I$, one has \begin{equation}\label{menorC}d_{X_d}(x_k,P(U(x_k,x_{m})^{\t_\scr C},U(x_k,x_{-n})^{\ii \t_\scr C}))<C,\end{equation}  provided $m\geq k+\ell_1$ and $-n\leq k-\ell_1$.

\begin{figure}[hbt]
	\centering
	\begin{tikzpicture}[scale = .75]
	\draw [thick] circle [radius = 3];
	\draw [fill] circle [radius = 0.03];	
	\node [left] at (0,0) {$x_0 = o$};
	\node [above right] at (0.5,2.958) {$U(o,x_m) = S^o(g_{m-1}\cdots g_0)$}; 
	\draw [fill] (0.5,2.958) circle [radius = 0.03]; %U(xm)
	\draw [thick] (2,-2.236) to [out = 130,in = 260] (0.5,2.958); %arco de U a S
	\draw [thick] (0,0) -- (0.5,2.958);
	\draw [fill] (0.35,2.1) circle [radius = 0.03]; %xm
	\node[left] at (0.35,2.1) {$x_m$};
	\draw [fill] (0.55,2.75) circle [radius = 0.02]; %xm+1
	\draw [fill] (0,1.7) circle [radius = 0.03]; %xm-1
	\node[left] at (0,1.7) {$x_{m-1}$}; %xm-1
	\draw [fill] (-0.2,1.3) circle [radius = 0.02];%..
	\node [right] at (2,-2.236) {$U(o,x_{-n}) = U^o(g_0\cdots g_{-n})$};	
	\draw [fill] (2,-2.236) circle [radius = 0.03];	%S(xm)%
	\draw [thick] (0,0) -- (2,-2.236);
	\draw [fill] (1.7,-2) circle [radius = 0.02]; %xn1
	\draw [fill] (0.55,2.75) circle [radius = 0.03]; %xn+1
	\draw [fill] (1.2,-1.341) circle [radius = 0.03];%xn
	\node[left] at (1.2,-1.341) {$x_{-n}$}; %xn
	\draw (0,0) -- (.6,.18); %dist conjunto paralelo
	\node [right] at (3.5,0) {$\leq c'-c\log\min\{\sin\angle_o(U^o(g_0\cdots g_{-n}),S^o(g_{m-1}\cdots g_0))\}$};
	%\draw [->] (3,-2) to [out = 100,in = -10] (0.45,0.60);
	\draw [->] (3.5,0) to [out = 140,in = 60] (.3,.18);
	%\node [right] at (3.5,3) {$P(F,E)$};
	\end{tikzpicture}
	\caption{The flag at infinity associated to the Weyl cone $V_{\t_\scr C}(o,x_m)$ (resp.\ $V_{\ii\t_\scr C}(o,x_m)$) corresponds to the $S^o$ flag of a dominated sequence for $m\geq0$, resp.\ to the $U^o$ flag when $m\leq0$.}\label{Morse}
\end{figure}
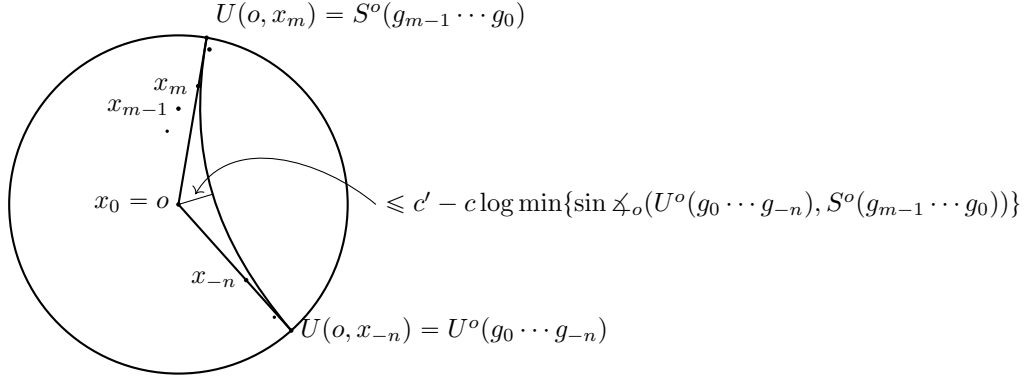

We now split between the three cases of the statement, in increasing difficulty. 

{\bf If $I=\Z,$} denote by $$F=E^{cu}(\{g_i\})=\lim_{n\to\infty} U(o,x_{-n})^{\ii\t_\scr C}$$ and by $$E=E^{cs}(\{g_i\})=\lim_{m\to\infty} U(o,x_{m})^{\t_\scr C}.$$ Recall that $E^{cu}\oplus E^{cs}$ is a dominated splitting for the cocycle $\vartheta$ (recall subsection \ref{ss.dom_sequences}), in particular it is equivariant. Hence, if $k>0$ then $$h_k^{-1} F=g_k\cdots g_0 E^{cu}(\{g_i\}_{i\in \Z})=E^{cu}(\{g_{i-k}\}_{i\in\Z})$$ and $h_k^{-1} E=g_k\cdots g_0 E^{cs}(\{g_{i-k}\}_{i\in\Z}).$ Consequently, the angle $$\angle_o(h_k^{-1} F,h_k^{-1} E)>\delta.$$ Thus $$d(h_k\cdot o,P(E,F))=d(o,P(h_k^{-1} F,h_k^{-1} E))<C.$$ The same argument works for $k<0$ and we have thus shown that $\{x_i\}_{i\in\Z}$ is at Hausdorff distance at most $C$ from the parallel set $P(F,E).$

{\bf If $I=\N$.} Let us assume for simplicity that $I=[-\ell_1,\infty)\cap\Z.$ The flag $U(o,x_m)^{\t_\scr C}$ is convergent as $m\to\infty$, to a flag $E\in\scr F_{\t_\scr C}.$ We have to show that, for all $k>0$ the $x_k$ angle $$\angle_{x_k}(U(x_k,o)^{\ii\t_\scr C},E)$$ is bounded below by a constant independent of $k.$

As before, we know that if $k\geq\ell_1 $ and $m\geq k+\ell_1$ then $$\angle_{x_k}(U(x_k,o)^{\ii\t_\scr C},U(x_k,x_m)^{\t_\scr C})>\delta.$$ For a homothety class $x\in X_d$ denote by $d_{x}$ the distance on $\scr F_{\t_\scr C}$ induced by an inner product in $x.$ We will show that $d_{x_k}\big(U(x_k,x_m)^{\t_\scr C},U(o,x_m)^{\t_\scr C}\big)$ goes to zero as $m\to\infty,$ and thus the angle $\angle_{x_k}( U(x_k,o)^{\ii\t_\scr C}, E)$ will be bounded below uniformly.

Recall that, by definition $U(x_k,x_m)^{\t_\scr C}=h_k U^o(h_k^{-1}h_m)^{\t_\scr C}$ hence \begin{multline*}d_{x_k}(U(x_k,x_m)^{\t_\scr C},U(o,x_m)^{\t_\scr C})=d_{h_k\cdot o}(h_k U^o(h_k^{-1}h_m)^{\t_\scr C},U^o(h_m)^{\t_\scr C})\\ =d_o(U^o(h_k^{-1}h_m)^{\t_\scr C},h_k^{-1}U^o(h_m)^{\t_\scr C}).\end{multline*} Since $k$ is fixed, Lemma \ref{l.domination_implies_slow_change}  implies that $d_o(U^o(h_k^{-1}h_m)^{\t_\scr C},h_k^{-1}U^o(h_m)^{\t_\scr C})\to0$ as $m\to\infty.$

Let us now show that given $\eps>0$ there exists $L\in\N$ (depending on the quasi-geodesic constants and $\scr C$) such that if $k\geq L$ one has $$d_{x_k}\big(U(x_k,x_{-\ell_1})^{\ii\t_\scr C},U(o,x_{-\ell_1})^{\ii\t_\scr C}\big)<\eps,$$ this will conclude the proof in this case since every $\{x_i\}_{i\in\N}$ will be at bounded distance from a Weyl cone pointing to $E$ in the parallel set $P(U(o,x_{-\ell_1})^{\ii\t_\scr C},E).$

One has \begin{multline*}d_{x_k}\big(U(x_k,x_{-\ell_1})^{\ii\t_\scr C},U(o,x_{-\ell_1})^{\ii\t_\scr C}\big)=d_{h_k\cdot o}\big(h_k U^o(h_k^{-1}h_{-\ell_1})^{\ii\t_\scr C},U^o(h_{-\ell_1})^{\ii\t_\scr C}\big)\\ =d_o\big(U^o(h_k^{-1}h_{-\ell_1})^{\ii\t_\scr C},h_k^{-1}U^o(h_{-\ell_1})^{\ii\t_\scr C}\big).\end{multline*} Since $$\angle_o(U^o(h_{-\ell_1})^{\ii\t_\scr C},S^o(h_k^{-1})^{\ii\t_\scr C})=\angle_o(U^o(h_{-\ell_1})^{\ii\t_\scr C},U^o(h_k)^{\t_\scr C})>\delta,$$ Lemmas~\ref{l.dominationattractor} and~\ref{l.Rafael} imply the desired conclusion.

{\bf If $I$ is finite,} notice that in the last paragraph we proved the following: given $\eps>0$ there exists $L\in\N$ such that for all $k\in I$ with $k-L\geq\min I$ one has $$d_{x_k}\big(U(x_k,x_{\min I})^{\ii\t_\scr C},U(x_{\min I +L},x_{\min I})^{\ii\t_\scr C}\big)<\eps.$$ Analogously, for every $k\in I$ with $k+L\leq\max I$ one has $$d_{x_k}\big(U(x_k,x_{\max I})^{\t_\scr C},U(x_{\max I -L},x_{\max I})^{\t_\scr C}\big)<\eps.$$

Hence, for every $k $ with $\min I+L<k<\max I-L$ one has \begin{multline*}\angle_{x_k}\big(U(x_{\min I +L},x_{\min I})^{\ii\t_\scr C},U(x_{\max I -L},x_{\max I})^{\t_\scr C}\big)\\ \geq \angle_{x_k}\big(U(x_k,x_{\min I})^{\ii\t_\scr C},U(x_k,x_{\max I})^{\t_\scr C}\big)-2\eps>\delta-2\eps.\end{multline*} Choosing $\eps$ such that $\delta-2\eps>0$ completes the proof.\end{proof}

%%%%%%%%%%%%%%%%%%%%%%%%%%%%%%%%%%%%%%%%%%%%%%%%%%%%%%%%%%%%%%%%

\section{When the target group is a semi-simple Lie group}\label{s.general}

The purpose of this section is to extend the main results in the previous sections to the situation where the target group is a real-algebraic semi-simple Lie group without compact factors. 

We will begin by recalling the general structure theory of these groups, needed to define concepts such as domination, Anosov representation, regular quasi-geodesic \dots  
This basic structure theory can be found in \cite{James}, \cite{Eberlein}, \cite{Helgason}.

We will then explain how the representation theory of these groups is used to reduce the general case to the $\PSL(d,\R)$ case, for a well chosen $d$. Section \ref{s.morse} mimics, for $\PSL(d,\R)$, the general structure presented here.

The first main goal is subsection \ref{generalcase?}, where Theorem \ref{teo:dominationimplieshyp} and Proposition \ref{p.DominatedImpliesAnosov} are extended to the general setting. This general case is reduced to the actual statement of \ref{teo:dominationimplieshyp} and \ref{p.DominatedImpliesAnosov} using Tits representations.

The remainder of the section is devoted to a new proof of the Morse Lemma of \cite{KLP2} for symmetric spaces of non-compact type. In contrast with subsection \ref{generalcase?}, a simple reduction to the $\PSL(d,\R)$-case is not sufficient, one needs to have a finer control of distances to parallel sets when embedding symmetric spaces. This is achieved in Corollary \ref{encaje1}.

If $G$ is a Lie group with Lie algebra $\frak g$, the \emph{Killing form} of $\frak g$ is the symmetric bilinear form defined by 
$$
\kappa(v,w) = \traza(\ad_v\ad_w).
$$
The group $G$ is \emph{semi-simple} if $\kappa$ is non-degenerate. 

We will assume from now on that $G$ is a semi-simple, real-algebraic (i.e.\ defined by polynomial equations with real coefficients) Lie group, without compact factors (i.e. there is no normal subgroup $H$ of $G$ such that $G/H$ is compact).

\subsection{Root system}

A \emph{Cartan involution} of $\frak g$ is a Lie algebra morphism $o:\frak g\to\frak g$ with $o^2 = 1$ and such that the bilinear form $(v,w)\mapsto-\kappa(v,o(w))$ is positive definite. The fixed point set 
$$
\frak k^{o} = \{v\in\frak g: o v = v\}
$$
is the Lie algebra of a maximal compact subgroup $K^o$. Consider $\frak p^o = \{v\in\frak g: o v = -v\}$ and note that 
$$
\frak g = \frak k^o\oplus\frak p^o.
$$

A computation shows that $[\frak p^o,\frak p^o]\en \frak k^o$ and hence any subalgebra of $\frak p^o$ is necessarily abelian. Let $\frak a\subset\frak p^o$ be a maximal abelian subalgebra.

Denote by $\E$ the set of \emph{restricted roots} of $\frak a$ on $\frak g$. By definition, 
$$
\E = \{\alpha\in\frak a^*-\{0\}:\frak g_\alpha\neq 0\}
$$
where 
$$
\frak g_\alpha = \{w\in\frak g:[a,w] = \alpha(a)w\ \forall a\in\frak a\}.
$$
The closure of a connected component of 
$$
\frak a -\bigcup_{\alpha\in\E} \ker\alpha
$$
is called \emph{a closed Weyl chamber}. Fix a closed Weyl chamber $\frak a^+$ and let $\E^+ = \{\alpha\in\E:\alpha|\frak a^+\geq0\}$ be the set of \emph{positive roots} associated to $\frak a^+$. The set $\E^+$ contains a subset $\Pi$ that verifies \begin{itemize}\item[-] $\Pi$ is a basis of $\frak a$ as a vector space, \item[-] every element of $\E^+$ has non-negative coefficients in the basis $\Pi$.\end{itemize} The set $\Pi$ is called the \emph{set of simple (restricted) roots} determined by $\E^+$, the sets $\ker\alpha\cap\frak a^+$ for $\alpha\in\Pi$, are the \emph{walls} of the chamber $\frak a^+$. 

The \emph{Weyl group} $W$ of $\E$ is defined as the group generated by the orthogonal reflections on the subspaces $\{\ker\alpha:\alpha\in\E\}$. 

The reflections associated to elements of $\Pi$ span $W$. With respect to the word-length on this generating set, there exists a unique longest element in $W$, denoted by $u_0:\frak a\to \frak a$. This is the unique element in $W$ that sends $\frak a^+$ to $-\frak a^+$. The \emph{opposition involution} $\ii:\frak a\to\frak a$ is defined by $\ii = -u_0$. If we denote by $(\mathord{\cdot},\mathord{\cdot})$ the bilinear form on $\frak a^*$ dual to the Killing form, define 
$$
\<\chi,\psi\> = \frac{2(\chi,\psi) }{(\psi,\psi)}
$$
for $\chi,\psi\in\frak a^*,$ and let $\{\om_\alpha\}_{\alpha\in\Pi}$ be the \emph{dual basis} of $\Pi$, i.e.\ $\<\om_\alpha,\beta\> = \delta_{\alpha\beta}$. The linear form $\om_\alpha$ is \emph{the fundamental weight} associated to $\alpha$. Note that for every $\chi\in\frak a^*$ one has 
\begin{equation}\label{coeficientes}\chi = \sum_{\alpha\in\Pi}\<\chi,\alpha\>\,\om_\alpha.
\end{equation}

Denote by $a = a_G:G\to \frak a^+$ the \emph{Cartan projection} of $G$. By definition, for every $g\in G$ one has $g\in K^o\exp a(g)\,K^o$ and $a(g^{-1}) = \ii a(g)$.

\subsection{Parabolic subgroups}\label{parabolic.groups}
Denote by $M$ the centralizer of $\exp\frak a$ in $K$ and by $N = \exp \frak n^+$ where $\frak n^+ = \bigoplus_{\alpha\in\E^+}\frak g_\alpha$. The group $P_\Pi = M\exp\frak a\,N$ is called a \emph{minimal parabolic subgroup} and its Lie algebra is $\frak p_\Pi = \frak g_0\oplus\frak n^+$. A \emph{parabolic subgroup} of $G$ is a subgroup that contains a conjugate of $P_\Pi$. Two parabolic subgroups are \emph{opposite} if their intersection is a reductive group.\footnote{Recall that a Lie group is \emph{reductive} if its Lie algebra splits as a semi-simple algebra and an abelian algebra.}

To each subset $\t$ of $\Pi$ one associates two opposite parabolic subgroups of $G$, $P_\t$ and $\wk{P_\t}$, whose Lie algebras are, by definition, 
$$
\frak p_\t = \frak g_0 \oplus\bigoplus_{\alpha\in\E^+}\frak g_\alpha\oplus \bigoplus_{\alpha\in \<\Pi-\t\>}\frak g_{-\alpha},
$$
and 
$$
\wk{\frak p}_\t = \frak g_0 \oplus\bigoplus_{\alpha\in\E^+}\frak g_{-\alpha}\oplus \bigoplus_{\alpha\in \<\Pi-\t\>}\frak g_{\alpha},
$$
where $\<\t\>$ is the set of positive roots generated by $\t$. Every pair of opposite parabolic subgroups of $G$ is conjugate to $(P_\t,\wk{P}_\t)$ for a unique $\t$, and every opposite parabolic subgroup of $P_\t$ is conjugate to $P_{\ii\t}:$ the parabolic group associated to 
$$
\ii\t = \{\alpha\circ\ii:\alpha\in\t\}.
$$

The quotient space $\scr F = \scr F_\Pi = G/P_\Pi$ is called \emph{the flag space} of $G$ and if $\t\subset\Pi$ then $\scr F_\t = G/P_\t$ is called \emph{a partial flag space} of $G$. Note that if $\t\en\t'\en\Pi$ one has $P_{\t'}\en P_\t$ and there is hence a canonical projection $\scr F_{\t'}\to \scr F_\t$, denoted by $x\mapsto x^\t.$

Finally, denote by $\posgen_\t\subset \scr F_\t\times\scr F_{\ii\t}$ the space of pairs of opposite parabolic subgroups (of type $\t$), this is the unique open $G$-orbit on $\scr F_\t\times\scr F_{\ii\t}.$
\subsection{Representations of $G$}\label{representaciones}

Let $\L:G\to\PSL(V)$ be a finite dimensional rational\footnote{i.e.\ a rational map between algebraic varieties.} irreducible representation and denote by $\phi_\L:\frak g\to\frak{sl}(V)$ the Lie algebra homomorphism associated to $\L$. Then $\chi\in\frak a^*$ is a \emph{restricted weight} of $\Lambda$ if the vector space 
$$
V_\chi = \{v\in V:\phi_\L(a) v = \chi(a) v\ \forall a\in\frak a\}
$$
is non zero. Theorem 7.2 of Tits \cite{Tits} states that the set of weights has a unique maximal element with respect to the order $\chi\geq\psi$ if $\chi-\psi$ is positive on $\frak a^+$. This is called \emph{the highest weight} of $\L$ and denoted by $\chi_\L$.

Note that if $\chi$ is a restricted weight and $v\in V_\chi$ then, for $n\in\frak g_\alpha$ with $\alpha\in\E$ one has that $\phi_\L(n)v$ is an eigenvector of $\phi_\L(\frak a)$ of eigenvalue $\chi+\alpha$, \footnote{Indeed, this follows from $\phi_\L([a,n])v = \alpha(a)\phi_\L(n)v$.} unless 
$\phi_\L(n)v = 0$. Since for all $\beta\in\E^+$ one has $\chi_\L+\beta\geq\chi_\L$ and $\chi_\L$ is maximal, one concludes that $\chi_\L+\beta$ is not a weight, i.e.\ for all $n\in\frak g_\beta$ and $v\in V_{\chi_\L}$ one has $\phi_\L(n)v = 0$.

Let $\t_\L\en\Pi$ be the set of simple roots $\alpha$ such that $\chi_\L-\alpha$ is still a weight of $\L$. 

\begin{obs}\label{igual0} The subset $\t_\L$ is the smallest subset of simple roots such that the following holds. Consider $\alpha\in\E^+$, $n\in\frak g_{-\alpha}$ and $v\in V_{\chi_\L}$, then $\phi_\L(n)v = 0$ if and only if $\alpha\in\<\Pi-\t_\L\>$. Equivalently, the smallest parabolic subgroup $P$ of $G$ stabilizing $V_{\chi_\L}$ is of type $\t_\L$.

\end{obs}

\begin{obs}\label{obs.modulo.irreducible}Since $V$ is irreducible, the $\phi_\L(\frak a)$-eigenspaces 
$$
\big(\prod_{\alpha\in\Pi}\phi_\L(n_\alpha^{i_\alpha})\big)V_{\chi_\L}
$$
span $V$ and thus any other weight of $\L$ is of the form 
\begin{equation}\label{pesos}
\chi_\L-\sum_{\alpha\in\Pi}k_\alpha \alpha,
\end{equation} 
where $k_\alpha\geq 0$ and $\sum_{\alpha\in\t_\L}k_\alpha\neq0$ (i.e.\ the numbers $k_\alpha$, for $\alpha\in\t_\L$, do not simultaneously vanish). 
\end{obs}

Consider an inner product on $V$ invariant under $\L K^o$ such that $\L \exp\frak a$ is symmetric, and denote by $\L o$ its homothety class. For the Euclidean norm $\|\ \|$ induced by this scalar product, one has 
\begin{equation}\label{eq:normayrep}\log\|\L g\| = \chi_\L(a(g)).
\end{equation} If $g = k(\exp a(g)) l$ with $k,l\in K^o$, then for all $v\in l^{-1}V_{\chi_\L}$ one has 
$$
\|\L g(v)\| = \|\L g\|\|v\|.
$$

If we denote by $\ell = \dim V_{\chi_\L}$ then Remark \ref{obs.modulo.irreducible} implies that for all $g\in G$ one has 
\begin{equation}\label{alpha}\alpha_\ell(a(\L g)) = a_\ell(\L g)-a_{\ell+1}(\L g) = \beta(a(g))
\end{equation} for some $\beta\in\t_\L$ (depending on $g$).

Denote by $W_{\chi_\L}$ the $\L(\exp\frak a)$-invariant complement of $V_{\chi_\L}$. The space $W_{\chi_\L}$ is $\L o$-orthogonal to $V_{\chi_\L}$. Its stabilizer (in $G$) is opposite to $P$, and hence conjugated to $P_{\ii\t_\L}$. Thus, one has a map of flag spaces 
\begin{equation}\label{maps}(\xi_\L,\xi^*_\L):\scr F_{\t_\L}^{(2)}\to \scr F_{\{\alpha_p\}}^{(2)}.
\end{equation} This is a proper embedding which is a homeomorphism onto its image.

\subsection{A set of representations defined by Tits}

One has the following proposition by Tits (see also Humphreys \cite[Chapter XI]{LAG}).

\begin{prop}[Tits \cite{Tits}]\label{prop:titss} For each $\alpha\in\Pi$ there exists a finite dimensional rational irreducible representation $\L_\alpha:G\to\PSL(V_\alpha)$, such that $\chi_{\L_\alpha}$ is an integer multiple of the fundamental weight $\om_\alpha$ and $\dim V_{\chi_{\L_\alpha}} = 1.$\end{prop}

Such a set of representations is not necessarily unique but will be fixed from now on, also we will say that $\L_\alpha$ is the \emph{Tits representation} of $G$ associated to $\alpha$. Observe that equation (\ref{alpha}) implies that for all $g\in G$ one has  
\begin{equation}\label{eq:relaciontits}
a_1(\L_\alpha g)-a_2(\L_\alpha g) = \alpha(a(g)).
\end{equation}

\subsection{Dominated representations: reduction to the \texorpdfstring{$\GL(d,\R)$}{GL(d,R)} case}\label{generalcase?} A representation $\rho:\G\to G$ is $\t$\emph{-dominated} if there exist positive constants $\mu$ and $c$ such that for all $\alpha\in\t$ and $\g\in\G$ one has 
$$
\alpha(a(\rho\g))\geq\mu|\g|-c.
$$

Assume that $\rho$ is a $\t$-dominated representation. Then for all $\alpha\in\t$, the representation $\L_\alpha\rho:\G\to\PSL(V_\alpha)$ is 1-dominated in the sense of Subsection~\ref{ss.dominatedrep}. Indeed, equation (\ref{eq:relaciontits}) implies that 
$$
\log\sigma_1(\L_\alpha\rho(\g))-\log\sigma_2(\L_\alpha\rho(\g)) = \alpha(a(\rho\g))\geq\mu|\g|-c,
$$
or equivalently 
$$
\frac{\sigma_2(\L_\alpha\rho(\g))}{\sigma_1(\L_\alpha\rho(\g))}\leq e^c e^{-\mu|\g|}.
$$
Thus, Theorem \ref{teo:dominationimplieshyp} implies that $\G$ is word-hyperbolic. Moreover, Proposition \ref{p.DominatedImpliesAnosov}, together with Guichard-Wienhard \cite[Proposition 4.3]{GuichardWienhard}, imply that $\rho$ is $P_\t$-Anosov\footnote{See \cite{Labourie-AnosovFlows} or \cite{GuichardWienhard} for a precise definition.}.

Thus, Theorem \ref{teo:dominationimplieshyp}, Proposition \ref{p.DominatedImpliesAnosov} together with \cite[Proposition 4.3]{GuichardWienhard} and Tits representations \ref{prop:titss} prove the following.

\begin{teo} Let $\rho:\G\to G$ be a $\t$-dominated representation, then $\G$ is word- hyperbolic and $\rho$ is $P_\t$-Anosov.
\end{teo}

\subsection{A Pl\"ucker representation}\label{plucker} Given $\t\subset \Pi$ one can construct a rational irreducible representation of $G$ such that $P_\t$ is the stabilizer of a line, and hence $\wk P_\t$ will be the stabilizer of a hyperplane.\footnote{There are actually infinitely many such representations.} More precisely, one has the following result from representation theory (see Guichard-Wienhard \cite[\S 4]{GuichardWienhard}).

\begin{prop}\label{prop:plucker} Given $\t\en\Pi$, there exists a finite dimensional rational irreducible representation $\plucker:G\to \PSL(V)$ and $\ell\in\P(V)$ such that 
$$
P_\t = \{g\in G:\plucker g(\ell) = \ell\}.
$$

\end{prop}

Such a representation can be defined as follows: if we denote by $k = \dim\frak p_\t$ then the composition $\L^k\Ad:G\to\PSL(\L^k\frak g)$ verifies the desired conditions, except (maybe) irreducibility, this is fixed by considering the vector space $V$ spanned by the $G$-orbit of the line $\ell = \L^k\frak p_\t$, 
$$
V = \<\L^k\Ad G\cdot \ell\>,
$$
and considering the restriction of $\L^k\Ad G$ to $V$.

A representation verifying the statement of Proposition \ref{prop:plucker} will be called a \emph{Pl\"ucker representation} of $G$ associated to $\t$. For such a representation one has a continuous equivariant map $(\xi,\xi^*):\posgen_\t\to\P^{(2)}(V)$ which is a homeomorphism onto its image.

\subsection{The symmetric space}

The \emph{symmetric space of $G$} is the space of Cartan involutions on $\frak g$, and is denoted by $X$. This a $G$-homogeneous space, and the stabilizer of $o\in X$ is the compact group $K^o$, whose lie algebra is $\frak k^o$. The tangent space ${\sf{T}}_o X$ is hence identified with $\frak p^o$. The $G$-invariant Riemannian metric on $X$ is the restriction of the Killing form $\kappa$ to $\frak p^o\times\frak p^o$.

If $d_X$ is the distance on $X$ induced by $\kappa|\frak p^o\times\frak p^o$, then the Euclidean norm $\|\ \|_o$ induced on $\frak a$ is invariant under the Weyl group, and for all $a\in\frak a$ one has $d_X(o,(\exp a)\cdot o) = \|a\|_o$. The Cartan decomposition of $G$ implies hence that for all $g\in G$ one has 
$$
d_X(o,g\cdot o) = \|a(g)\|_o.
$$

Consider the map $\aa:X\times X\to\frak a^+$ defined by $\aa(g\cdot o,h\cdot o) = a(g^{-1}h)$. Note that $\aa$ is $G$-invariant for the diagonal action of $G$ on $X\times X$, that 
\begin{equation}\label{distvectG}d_X(p,q) = \|\aa(p,q)\|_o
\end{equation} and that $\ii(\aa(p,q)) = \aa(q,p)$.

\subsection{Flats}

A \emph{parametrized flat} is a map $\p:\frak a\to X$ of the form $\p(v) = g\exp(v)\cdot o$ for some $g\in G$. Observe that $G$ acts transitively on the set of parametrized flats and that the stabilizer of $\p_0:v\mapsto \exp(v)\cdot o$ is the group $M$ of elements in $K$ commuting with $\exp(\frak a)$. We will hence identify the space of parametrized flats with $G/M$.

Two parametrized flats $\p,{\sf g}$ are \emph{equivalent} if the function $\frak a\to\R$, defined by 
$$
v\mapsto d_X(\p(v),{\sf g}(v)),
$$
is bounded on $\frak a^+$. The \emph{Furstenberg boundary} of $X$ is the space of equivalence classes of parametrized flats. It is a standard fact that this space is $G$-equivariantly identified with $\scr F = G/P_\Pi$, thus the Furstenberg boundary will also be denoted by $\scr F$. Denote by 
$$
\zz:\{\textrm{parametrized flats}\}\to\scr F
$$
the canonical projection and by $\wk\zz(\p) = \zz(\p\circ\ii)$. The pair $(\wk\zz(\p),\zz(\p))$ belongs to $\posgen$.

The following proposition is standard.

\begin{prop}[see {\cite[Chapter III]{bordeF}}]\item\begin{enumerate}\item A pair $(p,x)\in X\times\scr F$ determines a unique parametrized flat $\p$ such that $\p(0) = p$ and $\zz(\p) = x$. \item A point $(x,y)\in\posgen$ determines a unique maximal flat $\p_{xy}(\frak a)$ such that $\wk\zz(\p_{xy}) = x$ and $\zz(\p_{xy}) = y$. \end{enumerate}
\end{prop}

Given a subset of simple roots $\t\en\Pi$ and a pair of partial flags in general position $(x,y)\in\scr F_\t^{(2)}$ 
the \emph{parallel set} $P(x,y)$ from $x$ to $y$ is 
$$
\bigcup_\p \p(\frak a),
$$
where the union is indexed on all parametrized flats $\p$ with $\wk\zz(\p)^{\t} = x$ and $\zz(\p)^{\ii\t} = y$.

\subsection{Representations and distances to parallel sets} The purpose of this Subsection~are the following proposition and corollary. Only the statement of Corollary \ref{encaje1}  will be needed in the sequel. If $\L:G\to\PSL(V)$ is a finite dimensional rational irreducible representation, denote by $\L:X\to X_V$ the induced map between symmetric spaces and $D\L:{\sf{T}}X\to{\sf{T}}X_V$ be its differential mapping. The map $D_p\L$ is $\phi_\L|\frak p^p:\frak p^p\to\frak p^{\L p}$.

\begin{prop}\label{angulosTX} Let $\L:G\to\PSL(V)$ be a finite dimensional rational irreducible representation. Then, there exists a constant $\delta>0$ such that if $(x,y)\in\posgen_{\t_\L}$ and $p\in P(x,y)$ then 
$$
\angle \big(D_p\L\big(({\sf{T}}_pP(x,y))^\perp\big),{\sf{T}}_{\L p} P(\xi_\L x,\xi^*_\L y)\big)>\delta,
$$
where $\angle$ denotes the angle on ${\sf{T}}_{\L p}X_V$.
\end{prop}

\begin{cor}\label{encaje1} Let $\L:G\to\PSL(V)$ be an injective finite dimensional irreducible representation. Then there exists $c>0$ such that if $o\in X$ and $(x,y)\in\posgen_{\t_\L}$, then 
$$
\frac 1cd_{X_V}(\L o,P(\xi_\L x,\xi^*_\L y))\leq d_X(o,P(x,y))\leq cd_{X_V}(\L o,P(\xi_\L x,\xi^*_\L y)).
$$

\end{cor}

Let us show how Proposition \ref{angulosTX} implies the corollary.

\begin{proof} Since $G$ has no compact factors $X$ is non-positively curved. Hence, since $P(x,y)$ is totally geodesic, the distance from $o$ to $P(x,y)$ is attained at a unique point $p\in P(x,y)$. Moreover, the geodesic segment $\sigma:I\to X$, from $p$ to $o$ is orthogonal to $P(x,y)$ at $p$, i.e.\ $\dot\sigma(0)\in {\sf{T}}_pP(x,y)^\perp$. Similarly, the distance from $\L o$ to $P(\xi_\L x,\xi_\L^* y)$ is also attained at a unique point $q$ and the geodesic segment from $q$ to $\L o$ is perpendicular to $P(\xi_\L x,\xi^*_\L y)$.

Since $\L X$ is totally geodesic in $X_V$, we can estimate the angle at the vertex $\L p$ of the geodesic triangle $\{\L o,\L p, q\}$. Indeed, Proposition \ref{angulosTX} implies that this angle is bounded below by $\delta>0$, independently of $o$ and $(x,y)$. Since the angle at the vertex $q$ is $\pi/2$, a CAT(0) argument completes the proof. 
\end{proof}

The remainder of the subsection is devoted to the proof of Proposition \ref{angulosTX}.

\subsubsection*{Proof} Let $L(x,y)$ be the stabilizer in $G$ of $(x,y)\in\posgen_{\t_\L}$ and let $\frak l$ be its Lie algebra. Moreover, let $L(\xi_\L x,\xi^*_\L y)$ be the stabilizer in $\PSL(V)$ of $(\xi_\L x,\xi^*_\L y)$ and $\frak l'\en\frak{sl}(V)$ its Lie algebra.

If $\frak l^\perp$ denotes the orthogonal of $\frak l$ with respect to the Killing form $\kappa$, then we will show that 
\begin{equation}\label{vacio}\phi_\L(\frak l^\perp)\cap\frak l' = \{0\}.
\end{equation} This will imply the proposition since if $p\in P(x,y)$ then ${\sf{T}}_pP(x,y) = \frak p^p\cap\frak l$, its orthogonal in ${\sf{T}}_p X$ is hence $\frak p^p \cap\frak l^{\perp}$ and given $q\in P(x,y)$ there exists $g\in L(x,y)$ such that $gp = q$, thus the angle 
$$
\angle(\frak p^{\L p}\cap\phi_\L(\frak l^\perp),\frak p^{\L p}\cap \frak l')\geq\angle(\phi_\L(\frak p^p\cap\frak l^\perp),\frak p^{\L p}\cap \frak l')
$$
is independent of $p$ in $P(x,y)$.

We will hence show that equation (\ref{vacio}) holds. Note that, by homogeneity, we can assume that the stabilizer of $x$ is $P_{\t_\L}$ and that the stabilizer of $y$ is $\wk P_{\t_\L}$. The parallel set $P(x,y)$ is hence the orbit $L_{\t_\L}\cdot o$, where $L_{\t_\L}$ is the Levi group $P_\t\cap\wk P_\t$.

An explicit computation shows that for $\alpha,\beta\in\E$, the eigenspaces $\frak g_\alpha$ and $\frak g_\beta$ (recall Subsection~\ref{Cartan}) are orthogonal with respect to the Killing form $\kappa$, whenever $\alpha\neq-\beta$. If we denote by $\frak g_{\{\alpha,-\alpha\}} = \frak g_\alpha\oplus\frak g_{-\alpha}$ then the decomposition 
$$
\frak g = \frak g_0\oplus\bigoplus_{\alpha\in\E^+}\frak g_{\{\alpha,-\alpha\}}
$$
is orthogonal with respect to $\kappa$.

Note that the Cartan involution $o$ sends $\frak g_\alpha$ to $\frak g_{-\alpha}$ and hence $\frak g_{\{\alpha,-\alpha\}}$ is $o$-invariant. Since $-\kappa(\cdot,o(\cdot))$ is positive definite and $\kappa|\frak g_\alpha = 0$ one concludes that $\frak p^o\cap\frak g_{\{\alpha,-\alpha\}}\neq\{0\}$. One finds then a orthogonal decomposition of the tangent space to $X$ at $o$,
\begin{equation}\label{eq:ToXd} {\sf{T}}_oX = \frak p^o = \frak a\oplus\bigoplus_{\alpha\in\E^+}\frak p^o\cap\frak g_{\{\alpha,-\alpha\}}.
\end{equation}

The Lie algebra of $L_{\t_\L}$ is 
$$
\frak l_{\t_\L} = \frak g_0\oplus\bigoplus_{\alpha\in\<\Pi-\t_\L\>}\frak g_{\{\alpha,-\alpha\}}
$$
and hence the decomposition 
$$
\frak g = \frak l_{\t_\L}\oplus\bigoplus_{\alpha\in\E^+-\<\Pi-\t_\L\>}\frak g_{\{\alpha,-\alpha\}}
$$
is orthogonal with respect to $\kappa$. This is to say 
\begin{equation}\label{orto}\frak l_{\t_\L}^{\perp} = \bigoplus_{\alpha\in\E^+-\<\Pi-\t_\L\>}\frak g_{\{\alpha,-\alpha\}}.
\end{equation}

Denote by $\ell = \dim V_{\chi_\L}$ and note that $\frak l'$ is the Lie algebra of the stabilizer of $(V_{\chi_\L},W_{\chi_\L})\in\posgen_{\{\alpha_\ell\}}$. Choose a Cartan algebra $\frak a_{\cE}$ of $\frak{sl}(V)$ and a Weyl chamber $\frak a_{\cE}^+$ such that\footnote{It suffices to choose a $\L o$-orthogonal set of $\dim V$ lines $\cE$ such that $\cE\subset V_{\chi_\L}\cup W_{\chi_\L}$.} (recall section \ref{parabolic.groups}): \begin{itemize}\item[-] the Lie algebra $\frak p_{\{\alpha_\ell\}}$ is the Lie algebra of the parabolic group stabilizing $V_{\chi_\L},$\item[-] the Lie algebra $\wk{\frak p}_{\{\alpha_\ell\}}$ is the Lie algebra of the parabolic group stabilizing $W_{\chi_\L}$. \end{itemize}

The Lie algebra $\frak l'$ is thus $\frak p_{\{\alpha_\ell\}}\cap\wk{\frak p}_{\{\alpha_\ell\}}$. In order to show that $\phi_\L(\frak l_{\t_\L}^{\perp})\cap\frak l' = \{0\}$ the following lemma is sufficient.

\begin{lem} 
The subspace ${\displaystyle \phi_\L \left( \bigoplus_{\alpha\in\E^+-\<\Pi-\t_\L\>}\frak g_{-\alpha} \right)}$ has trivial intersection with $\frak p_{\{\alpha_\ell\}}$.
\end{lem}

\begin{proof} Recall that 
$$
\frak p_{\{\alpha_p\}} = \frak a_{\cE}\oplus\bigoplus_{\beta\in\E_V^+}\frak{sl}(V)_\beta\oplus\bigoplus_{\beta\in\<\Pi_V-\{\alpha_l\}\>}\frak{sl}(V)_{-\beta},
$$
where $\E_V^+$ and $\Pi_V$ are the set of positive, resp.\ simple, roots associated to the choice of $\frak a_{\cE}^+$.

Consider then $\alpha\in\E^+-\<\Pi-\t_\L\>$ and $n\in\frak g_{-\alpha}$. Remark \ref{igual0} implies that if $v\in V_{\chi_\L}$ then $\phi_\L(n)v\neq0$, Hence, 
$$
\phi_\L(n)\notin\frak a_{\cE}\oplus\bigoplus_{\beta\in\E_V^+}\frak{sl}(V)_\beta\oplus\bigoplus_{\beta\in\<\Pi_V-\{\alpha_l\}\>}\frak{sl}(V)_{-\beta}.
$$

Let $\{\gamma_i\}_1^k\en\E^+-\<\Pi-\t_\L\>$ be pairwise distinct and consider $n_{\gamma_i}\in\frak g_{-\gamma_i}-\{0\}$. Then $\phi_\L(n_{\gamma_i})v$ is a non-zero eigenvector of $\frak a$ of eigenvalue $\chi-\gamma_i$, hence 
$$
\phi_\L(n_{\gamma_1}+\cdots+ n_{\gamma_k})v = \phi_\L(n_{\gamma_1})v+\cdots+\phi_\L(n_{\gamma_k})v\neq0.
$$
This proves the lemma.
\end{proof}

As an example, consider the irreducible representation (unique up to conjugation) $\L:\PSL(2,\R)\to\PSL(3,\R)$. Explicitly $\L$ is the action of $\PSL(2,\R)$ on the symmetric power $\mathbf{S}^2(\R^2)$, which is a 3-dimensional space spanned by 
$$
\{e_1\otimes e_1,e_1\otimes e_2,e_2\otimes e_2\},
$$
where $e_1 = (1,0)$ and $e_2 = (0,1)$. 

The group $L(\< e_1\>,\< e_2\>)$ has Lie algebra $\frak l = \frak a = \<\left(\begin{smallmatrix} 1 & 0 \\ 0 & -1\end{smallmatrix}\right)\>$, its orthogonal with respect to the Killing form of $\PSL(2,\R)$ is 
$$
\frak a^\perp = \<\left(\begin{smallmatrix}0 & 1\\ 0 & 0\end{smallmatrix}\right),\left(\begin{smallmatrix}0 & 0\\ 1 & 0\end{smallmatrix}\right)\>.
$$

The group $L(\<e_1\otimes e_1\>,\< e_1\otimes e_2, e_2\otimes e_2\>)$ has Lie algebra 
$$
\frak l' = \left\<\left(\begin{smallmatrix} 1 & 0 & 0\\ 0 & -1 & 0 \\ 0& 0 & 0\end{smallmatrix}\right),\left(\begin{smallmatrix} 1 & 0 & 0\\ 0& 0 & 0\\ 0 & 0& -1 \end{smallmatrix}\right),\left(\begin{smallmatrix} 0 & 0 & 0\\ 0 & 1 & 0 \\ 0& 0 & -1\end{smallmatrix}\right),\left(\begin{smallmatrix} 0 & 0 & 0\\ 0 & 0& 1 \\ 0& 0 & 0\end{smallmatrix}\right), \left(\begin{smallmatrix} 0 & 0 & 0\\ 0 & 0 & 0 \\ 0& 1 & 0\end{smallmatrix}\right)\right\>.
$$
The orthogonal complement of $\frak l'$ with respect to the Killing form of $\PSL(3,\R)$ is 
$$
\frak l'^\perp = \left\<\left(\begin{smallmatrix} 0 & 1 & 0\\ 0 & 0 & 0 \\ 0& 0 & 0\end{smallmatrix}\right),\left(\begin{smallmatrix} 0 & 0 & 1\\ 0 & 0 & 0 \\ 0& 0 & 0\end{smallmatrix}\right),\left(\begin{smallmatrix} 0 & 0 & 0\\ 1 & 0 & 0 \\ 0& 0 & 0\end{smallmatrix}\right),\left(\begin{smallmatrix} 0 & 0 & 0\\ 0 & 0 & 0 \\ 1 & 0 & 0\end{smallmatrix}\right)\right\>.
$$

An explicit computation shows that $\L\left(\begin{smallmatrix} 1 & t \\ 0 & 1\end{smallmatrix}\right) = \left(\begin{smallmatrix} 1 & t & t^2 \\ 0 & 1 & t \\ 0 & 0 & 1\end{smallmatrix}\right)$, hence 
$$
\phi_\L\left(\begin{smallmatrix} 0 & 1 \\ 0 & 0\end{smallmatrix}\right) = \left(\begin{smallmatrix} 0 & 1 & 0 \\ 0 & 0 & 1 \\ 0 & 0 & 0\end{smallmatrix}\right)
$$
which does not %\marg{A:saqu\'e el underline} 
belong to $\frak l'^\perp$. Nevertheless, it does not belong to $\frak l'$ either, which gives the definite angle of Proposition \ref{angulosTX}.

\subsection{Gromov product and representations} Consider $\t\en\Pi$ and denote by 
$$
\frak a_\t = \bigcap_{\alpha\in\Pi-\t}\ker\alpha
$$
the Lie algebra of the center of the reductive group $L_\t = P_\t\cap\wk P_\t$. Recall from \cite{Sambarino-Orbitalcounting} that the \emph{Gromov product}\footnote{This is the negative of the one defined in \cite{Sambarino-Orbitalcounting}.} based on $o\in X$ is the map 
$$
(\mathord{\cdot}|\mathord{\cdot})_o:\posgen_\t\to\frak a_\t
$$
defined by the unique vector $(x|y)_o\in\frak a_\t$ such that 
$$
\om_\alpha((x|y)_o) = -\log\sin\angle_{\L_\alpha o}(\xi_{\L_\alpha}x,\xi^*_{\L_\alpha}y) = -\log\frac{|\varphi(v)|}{\|\varphi\|_{\L_\alpha o}\|v\|_{\L_\alpha o}}
$$
for all $\alpha\in\t$, where $\om_\alpha$ is the fundamental weight of $\L_\alpha$, $v\in\xi_{\L_\alpha}x-\{0\}$ and $\ker\varphi = \xi_{\L_\alpha}^*y$.

\begin{obs}Note that 
\begin{equation}\label{minangulo}\max_{\alpha\in\t}\om_\alpha((x|y)_o) = -\log\min_{\alpha\in\t}\sin\angle_{\L_\alpha o}(\xi_{\L_\alpha}x,\xi^*_{\L_\alpha}y).
\end{equation} Note also that, since $\{\om_\alpha|\frak a_\t\}_{\alpha\in\t}$ is a basis of $\frak a_\t$, the right hand side of equation (\ref{minangulo}) is comparable to the norm $\|(x|y)_o\|_o$. \end{obs}

As suggested by the notation, the Gromov product is independent on the choice of Tits's representations of $G$, moreover, it keeps track of Gromov products for all irreducible representations of $G$. Indeed one has the following consequence of equation (\ref{eq:normayrep}).

\begin{obs}\label{GromovyRep} Let $\L:G\to\PSL(V)$ be a finite dimensional rational irreducible representation with $\dim V_{\chi_\L} = 1$. If $(x,y)\in\posgen_{\t_\L}$ then\footnote{We have identified $\frak a_{\{\alpha_ 1\}}$ with $\R$ via $\omega_{\alpha_1}$.}
$$
(\xi_\L x |\xi_\L^* y)_{\L o} = \chi_\L((x|y)_o) = \sum_{\alpha\in\t_\L}\<\chi_\L,\alpha\>\,\om_\alpha((x|y)_o).
$$
Note that, by definition of $\t_\L$ (recall Subsection \ref{representaciones}), the coefficients in the last equation are all strictly positive. 
\end{obs}

\subsection{Gromov product and distances to parallel sets} 
The aim of this subsection is to prove the following.

\begin{prop}\label{angles} Given $\t\en\Pi$ there exist $c>1$ and $c'>0$ only depending on $G$ such that for all $(x,y)\in\posgen_\t$ one has 
$$
\frac1c\|(x|y)_o\|_o\leq d_X(o,P(x,y))\leq c\|(x|y)_o\|_o+c'.
$$

\end{prop}

\begin{proof} As for $\PSL(d,\R)$, the first inequality follows easily. Let us show the second one. Let $\L:G\to\PSL(V)$ be a Pl\"ucker representation of $G$ associated to the set $\t$, (recall Subsection~\ref{plucker}). Corollary \ref{encaje1} implies that there exists $c_0$ such that for all $o\in X$ and $(x,y)\in\posgen_{\t_\L}$ one has 
$$
d_X(o,P(x,y))\leq c_0d_{X_V}(\L o,P(\xi_\L x,\xi_\L^* y)).
$$
Note that $\xi_\L x\in\P(V) $ and that $\xi^*_\L y\in\P(V^*)$, Proposition \ref{cjtoparalelogeneral} applied to the set $\{\alpha_1\}\en\Pi_V$ implies that 
$$
d_{X_V}(\L o,P(\xi_\L x,\xi_\L^* y))\leq c(\xi_\L x|\xi_\L^* y)_{\L o}+c'.
$$
Finally, Remark \ref{GromovyRep} states that 
$$
(\xi_\L x|\xi_\L^* y)_{\L o} = \sum_{\alpha\in\t}\<\chi_\L,\alpha\>\,\om_\alpha((x|y)_o),
$$
where $\<\chi_\L,\alpha\>>0$ for all $\alpha\in\t$, thus, this last quantity is comparable to $\|(x|y)_o\|$. This completes the proof.
\end{proof}

\subsection{The Morse Lemma of Kapovich--Leeb--Porti} Let us begin with some definitions. Consider $p,q\in X$, $\t\en\Pi$ and $(x,y)\in\posgen_\t$.

\begin{itemize}\item[-] The \emph{Weyl cone} $V(p,x)$ through $p$ and $x$ is $\bigcup_\p \p(\frak a^+)$, where the union is indexed on all parametrized flats $\p$ with $\p(0) = p$ and $\zz(\p)^\t = x$.\item[-] Let 
$$
\t(p,q) = \{\alpha\in\Pi: \alpha(\aa(p,q))\neq0\}
$$
and let $U(p,q)\in \scr F_{\t(p,q)}$ be the flag determined by $\zz(\p )^{\t(p,q)} = U(p,q)$ for every parametrized flat with $\p(0) = p$ and $q\in\p(\frak a^+)$.

\item[-] If $\t\en \t_{(p,q)}$ the \emph{Weyl cone} $V_\t(p,q)$ through $p$ and $q$ (in that order) is 
$$
\bigcup_\p \p(\frak a^+),
$$
where the union is indexed on all parametrized flats $\p$ with $\p(0) = p$, and $\zz(\p)^\t = U(p,q)^{\t}$. \item[-] Finally, if $\t\en \t_{(p,q)}$ the $\t$-\emph{diamond} between $p$ and $q$ is the subset 
$$
\dd_{\t}(p,q) = V_{\t}(p,q)\cap V_{\ii\t}(q,p).
$$

\end{itemize}

If $\scr C\subset \frak a^+$ is a closed cone, consider the subset 
$$
\t_\scr C = \{\alpha\in\Pi: \ker\alpha\cap\scr C = \{0\}\}.
$$
Let $I\en \Z$ be an interval. Following \cite{KLP2} we will say that a quasi-geodesic segment $\{p_n\}_I\subset X$ is $\scr C$-\emph{regular} if for all $n<m$ one has 
$$
\aa(p_n,p_m)\in\scr C.
$$
One has the following version of the Morse Lemma.

\begin{teo}[{\cite[Theorem 1.3]{KLP2}}] Let $\mu,c$ be positive numbers and $\scr C\en\frak a^+$ a closed cone, then there exists $C>0$ such that if $\{p_n\}_{n\in I}$ is a $\scr C$-regular $(\mu,c)$-quasi-geodesic segment, then \begin{itemize}\item[-] If $I$ is finite then $\{p_n\}$ is at distance at most $C$ from the diamond $\dd_{\t_{\scr C}}(p_{\min I},p_{\max I})$. \item[-] If $I = \N$ then there exists $x\in\scr F_{\t_{\scr C}}$ such that $\{p_n\}$ is contained in a $C$-neighborhood from the Weyl cone $V_{\t_{\scr C}}(p_{\min I},x)$. \item[-] If $I = \Z$ then there exist two partial flags in general position $(x,y)\in\posgen_{\t_\scr C}$ such that $\{p_n\}$ is contained in a $C$-neighborhood from the union $V_{\t_{\scr C}}(z,x)\cup V_{\ii\t_{\scr C}}(z,y)$ for some $z\in P(x,y)$ at uniform distance from $\{p_n\}$. \end{itemize}
\end{teo}

\begin{proof} Let $\L:G\to\PSL(V)$ be a Pl\"ucker representation associated to $\t_{\scr C}$. If $\{p_n\}$ is a $\scr C$-regular quasi geodesic then $\{\L p_n\}$ is a $\L\scr C$-regular quasi-geodesic. Moreover, equation (\ref{alpha}) implies that the cone $\L \scr C$ does not intersect the wall $\ker\alpha_1$. 

The proof now follows the lines of the proof of Theorem~\ref{teo:MorseSLd}, with the use of Proposition~\ref{angles}.
\end{proof}

%%%%%%%%%%%%%%%%%%%%%%%%%%%%%%%%%%%%%%%%%%%%%%%%%%%%%%%%%%%%%%%%%
\appendix

\section{Auxiliary technical results}\label{s.appendix}

In this appendix we collect a number of lemmas that are used elsewhere in the paper.
These lemmas are either quantitative linear-algebraic facts or properties of dominated splittings.
Certainly many of these results are known, but they do not necessarily appear in the literature in the exact form or setting that we need; therefore we include proofs for the reader's convenience.

\subsection{Angles, Grassmannians}\label{ss.grass}

The \emph{angle} between non-zero vectors
$v$, $w\in\R^d$ is defined as the unique number $\angle(v,w)$ in $[0,\pi]$
whose cosine is $\langle v,w \rangle / (\|v\| \, \|w\|)$.
If $P$, $Q \subset \R^d$ are nonzero subspaces then we define their \emph{angle} as:
\begin{equation}\label{e.ang}
\angle(P,Q) \coloneqq \min_{v \in P^\times} \min_{w \in Q^\times} \angle(v,w) \, ,
\end{equation}
where $P^\times \coloneqq P \smallsetminus \{0\}$.
We also write $\angle (v, Q)$ instead of $\angle(\R v, Q)$, if $v$ is a nonzero vector.

Given integers $1 \leq p < d$, we let $\Gr_p(\R^d)$ be the \emph{Grassmaniann} formed by the $p$-dimensional subspaces of $\R^d$. 
Let us metrize $\Gr_p(\R^d)$ in a convenient way.
If $p = 1$, the sine of the angle defines a distance. (We leave for the reader to check the triangle inequality.)
In general, we may regard each element of $\Gr_p(\R^d)$ as a compact subset of $\Gr_1(\R^d)$, and then use the Hausdorff metric induced by the distance on $\Gr_1(\R^d)$.
In other words, we define, for $P$, $Q \in \Gr_p(\R^d)$:
\begin{equation}\label{e.distGr0}
d(P,Q) \coloneqq \max \left\{ 
\max_{v \in P^\times} \min_{w \in Q^\times} \sin \angle(v,w), \ 
\max_{w \in Q^\times} \min_{v \in P^\times} \sin \angle(v,w) 
\right\} \, .
\end{equation}
Note the following trivial bound:
\begin{equation}\label{e.stupid_bound}
d(P,Q) \ge \sin \angle(P,Q) \, .
\end{equation}
In fact, the following stronger fact holds: for any subspace $\tilde Q \subset \R^d$ (not necessarily of dimension $p$),
\begin{equation}\label{e.stupid_bound2}
\tilde Q \cap Q \neq \{0\} \quad \Rightarrow \quad
d(P,Q) \ge \sin \angle(P, \tilde Q) \, .
\end{equation}
Actually the two quantities between curly brackets in the formula \eqref{e.distGr0} coincide, that is,
\begin{equation}\label{e.distGr}
d(P,Q) = \max_{w \in Q^\times} \min_{v \in P^\times} \sin \angle(v,w) 
	   = \max_{w \in Q^\times} \sin \angle(w,P) \, . 
\end{equation}
This follows from the existence of a isometry of $\R^d$ that interchanges $P$ and $Q$ (see \cite[Theorem~2]{Wong}).

As the reader can easily check, relation \eqref{e.distGr} can be rewritten in the following ways:
\begin{align}
d(P,Q) 
	& = \max_{w \in Q^\times} \max_{u \in (P^\perp)^\times} \cos \angle(u,w) \, , \label{e.distGr2}
\\
   & = \max_{w \in Q^\times} \min_{v \in P} \frac{\|v - w\|}{\|w\|} \, ,
\label{e.distGr3}
\end{align}
where $P^\perp$ denotes the orthogonal complement of $P$.
As a consequence of \eqref{e.ang} and \eqref{e.distGr2}, we have:
\begin{equation}\label{e.distGr4}
d(P,Q) = \cos \angle(P^\perp, Q) \, .
\end{equation}
Another expression for the distance $d(P,Q)$ is given below in \eqref{e.sin_tan}.

A finer description of the relative position of a pair of elements in the Grassmannian is given by the next proposition:

\begin{prop}[Canonical angles between a pair of spaces] \label{p.canonical}
Let $P$, $Q \in \grass_p(\R^d)$.
Then there exist:
\begin{itemize} 
\item numbers $\beta_1 \ge \beta_2 \ge \cdots \ge \beta_p \in [0,\pi/2]$, called \emph{canonical angles}, and
\item orthonormal bases $\{v_1, \dots, v_p\}$ and $\{w_1, \dots, w_p\}$ of $P$ and $Q$ respectively,
\end{itemize}
such that, for every $i$, $j \in \{1,\dots,p\}$,
$$
\angle(v_i,w_j) =
\begin{cases}
	\beta_i	&\text{if $i=j$,} \\
	\pi/2	&\text{if $i\neq j$,}
\end{cases}
$$
and moreover:
\begin{itemize}
\item if $p>d/2$ then $\beta_{d-p+1} = \beta_{d-p+2} = \cdots = \beta_p = 0$;
\item $d(P,Q) = \sin \beta_1$.
\end{itemize}
\end{prop}

The proof of the proposition is a simple application of the singular value decomposition: see \cite[p.73]{Stewart}.

\subsection{More about singular values}

If $A \colon E \to F$ is a linear map between two inner product vector spaces then we define its \emph{norm} and its \emph{conorm} (or \emph{mininorm}) by:
$$
\|A\|  \coloneqq \max_{v \in E^\times} \frac{\|A v\|}{\|v\|} \, , \qquad
\mm(A) \coloneqq \min_{v \in E^\times} \frac{\|A v\|}{\|v\|} \, .
$$
The following properties hold whenever they make sense:
$$
\|AB\| \le \|A\| \, \|B\| \, , \qquad
\mm(AB) \ge \mm(A) \, \mm(B) \, , \qquad
\mm(A)  =  \| A^{-1}\|^{-1} \, .
$$
In terms of singular values, we have $\|A\|  =  \sigma_1(A)$ and $\mm(A)  =  \sigma_d(A)$, where $d  =  \dim E$.

For convenience, let us assume that $A$ is a linear map from $\R^d$ to $\R^d$. % . endowed with the usual Euclidian inner product.
We have the following useful ``minimax'' characterization of the singular values:
\begin{align}
\sigma_p(A)     & =  \max_{P \in \Gr_p(\R^d)} \mm(A|_P) \, ,      \label{e.maxmin}
\\ 
\sigma_{p+1}(A) & =  \min_{Q \in \Gr_{d-p}(\R^d)} \| A|_Q \| \, ; \label{e.minmax}
\end{align}
see \cite[Corol.~4.30]{Stewart}.
Moreover, if $A$ has a gap of index $p$ (that is, $\sigma_p(A) > \sigma_{p+1}(A)$)
then the maximum and the minimum above are respectively attained (uniquely) at the spaces $P  =  S_{d-p}(A)^\perp$ and $Q = S_{d-p}(A)$ (defined at \S~\ref{ss.dom_sing}).

\subsection{Linear-algebraic lemmas}\label{ss.la_lemmas}

In this subsection we collect a number of estimates that will be useful later.
Fix integers $1 \le p \le d$.

\begin{lem}\label{l.sing_value_change}
Let $A$, $B \in \GL(d,\R)$. 
Then:
$$
\max \big\{ \mm(A) \sigma_p(B), \sigma_p(A) \mm(B) \big\}
\le \sigma_p(AB) \le 
\min \big\{ \|A\| \sigma_p(B) , \sigma_p(A) \|B\|  \big\} \, .
$$
\end{lem}

\begin{proof}
The two inequalities follow from \eqref{e.maxmin} and \eqref{e.minmax}, respectively.
\end{proof}

\begin{lem}\label{l.pitagoras}
Let $A \in \GL(d,\R)$ have a gap of index $p$.
Then, for all unit vectors $v$, $w \in \R^d$, we have: 
\begin{align}
\| A v \|      &\ge \sigma_p(A) \sin \angle (v, S_{d-p}(A))      \, , \label{e.pit1} \\
\| A^{-1} w \| &\ge \sigma_{p+1}(A)^{-1} \sin \angle (w, U_p(A)) \, . \label{e.pit2}
\end{align}
Also, for all $Q \in \Gr_{d-p}(\R^d)$ and $P \in \Gr_p(\R^d)$, we have:
\begin{align}
\|A|_Q\|		&\ge \sigma_p(A)			\, d(Q, S_{d-p}(A)) \, , \label{e.pit3} \\
\|A^{-1}|_P\|	&\ge \sigma_{p+1}(A)^{-1}	\, d(P, U_p(A))     \, . \label{e.pit4}
\end{align}
\end{lem}

\begin{proof}
Given a unit vector $v \in \R^d$, decompose it as $v  =  s+u$ where $s \in S_{d-p}(A)$ and $u \in S_{d-p}(A)^\perp$;
Then $\|u\|  =   \sin \angle (v, S_{d-p}(A))$.
Moroever, since $As$ and $Au$ are orthogonal, we have $\|A v\| \ge \|A u\| \ge \sigma_p(A) \|u\|$.
This proves \eqref{e.pit1}, from which \eqref{e.pit3} follows.
Inequalities \eqref{e.pit2} and \eqref{e.pit4} follow from the previous ones applied to the matrix $A^{-1}$, which has a gap of index $d-p$.
\end{proof}

The next three lemmas should be thought of as follows:
if $A$ has a strong gap of index $p$ and $\|B^{\pm 1}\|$ are not too large, then 
then $U_p(AB)$ is close to $U_p(A)$, and $U_p(BA)$ is close to $B(U_p(A))$;
moreover, $A(P)$ is close to $U_p(A)$ for any $P \in \Gr_{p}(\R^d)$ whose angle with $S_{d-p}(A)$ is not too small.\footnote{We remark that angle estimates of this flavor appear in the usual proofs os Oseledets theorem; see e.g.\ \cite[p.\ 141--142]{Simon}. We also note that \cite[Lemma 5.8]{GGKW} contains a generalization of Lemmas \ref{l.nochangeright} and \ref{l.dominationattractor} to more general Lie groups.}

\begin{lem}\label{l.nochangeright}
Let $A$, $B \in \GL(d,\R)$.
If $A$ and $AB$ have gaps of index~$p$ then:
$$
d(U_p(A),U_p(AB)) \leq \|B\| \, \|B^{-1}\| \, \frac{\sigma_{p+1}(A)}{\sigma_p(A)}. 
$$
\end{lem}

For another way to estimate the distance of $U_p(A)$ and $U_p(AB)$ that does not rely on ``smallness'' of $B$, see Lemma~\ref{l.Rafael} below.

\begin{proof}[Proof of Lemma~\ref{l.nochangeright}]
We have:
\begin{alignat*}{2}
d(U_p(AB),U_p(A)) 
&\le \sigma_{p+1}(A) \, \|A^{-1}|_{U_p(AB)}\| 					&\quad&\text{(by \eqref{e.pit4})} \\
&\le \sigma_{p+1}(A) \, \|B\| \, \|B^{-1}A^{-1}|_{U_p(AB)} \| \\
& =    \sigma_{p+1}(A) \, \|B\| \, \sigma_p(AB)^{-1}  \\
&\le \sigma_{p+1}(A) \, \|B\| \, \|B^{-1}\| \, \sigma_p(A)^{-1}	&\quad&\text{(by Lemma~\ref{l.sing_value_change}).}
\xqedhere{9mm}{\qed}
\end{alignat*}
\renewcommand{\qed}{}
\end{proof}

\begin{lem}\label{l.domination_implies_slow_change} 
Let $A$, $B \in \GL(d,\R)$.
If $A$ and $BA$ have gaps of index $p$ then:
$$
d\big( B(U_p(A)), U_p(BA) \big) 
\le \|B\| \, \|B^{-1}\| \, \frac{\sigma_{p+1}(A)}{\sigma_p(A)} \, .
$$
\end{lem}

\begin{proof}
We have:
\begin{alignat*}{2}
d(B(U_p(A)), U_p(BA)) 
&\le \sigma_{p+1}(BA) \, \|A^{-1}B^{-1}|_{B(U_p(A))}\| 				 &\quad&\text{(by \eqref{e.pit4})} \\
&\le \|B\| \, \sigma_{p+1}(A) \, \|A^{-1}|_{U_p(A)} \| \, \|B^{-1}\| &\quad&\text{(by Lemma~\ref{l.sing_value_change})} \\
& =    \|B\| \, \sigma_{p+1}(A) \, \sigma_p(A)^{-1} \, \|B^{-1}\| \, . &&
\xqedhere{31mm}{\qed}
\end{alignat*}
\renewcommand{\qed}{}
\end{proof}

\begin{lem}\label{l.dominationattractor} 
Let $A \in \GL(d,\R)$ have a gap of index $p$.
Then, for all $P \in \Gr_{p}(\R^d)$ transverse to $S_{d-p}(A)$ we have:
$$
d(A(P), U_p(A)) \le \frac{\sigma_{p+1}(A)}{\sigma_p(A)} \, \frac{1}{\sin\angle(P,S_{d-p}(A))} \, .
$$
\end{lem}

\begin{proof}
By \eqref{e.pit4},
$$
d(A(P), U_p(A)) 
\le \sigma_{p+1}(A) \, \|A^{-1}|_{A(P)}\| 
 =    \frac{\sigma_{p+1}(A)}{\mm(A|_P)} \, . 
$$
By \eqref{e.pit1}, $\mm(A|_P) \ge \sigma_p(A) \sin \angle (P, S_{d-p}(A))$, 
so the lemma is proved.
\end{proof}

The next lemma implies that the singular values of a product of matrices are approximately the products of the singular values, provided that certain angles are not too small:

\begin{lem}\label{l.no_cancellation}
Let $A$, $B \in \GL(d,\R)$.
Suppose that $A$ and $AB$ have gaps of index $p$.
Let $\alpha \coloneqq \angle \big( U_p(B), S_{d-p}(A) \big)$.
Then:
\begin{align*}
\sigma_p(AB)     &\ge (\sin \alpha) \, \sigma_p(A) \, \sigma_p(B) \, ,  		\\
\sigma_{p+1}(AB) &\le (\sin \alpha)^{-1} \, \sigma_{p+1}(A) \, \sigma_{p+1}(B) \, . 
\end{align*}
\end{lem}

\begin{proof}
By \eqref{e.maxmin} we have: 
$$
\sigma_p(AB) \ge \mm(AB|_{B^{-1}(U_p(B))}) 
\ge \mm(A|_{U_p(B)}) \, \mm(B|_{B^{-1}(U_p(B))})
 =  \mm(A|_{U_p(B)}) \, \sigma_p(B) \, .
$$
On the other hand, inequality \eqref{e.pit1} yields $\mm(A|_{U_p(B)}) \ge (\sin \alpha) \, \sigma_p(A)$,
and so we obtain the first inequality in the lemma.
The second inequality follows from the first one, using the fact that $\sigma_{p+1}(A^{-1})  =  1 / \sigma_{d-p}(A)$.
\end{proof}

\subsection{A sketch of the proof of Theorem~\ref{t.BG}}\label{ss.sketch_BG}

Note that due to the uniqueness property of dominated splittings (Proposition~\ref{p.uniqueness}), it is sufficient to prove the theorem in the case $\mathbb{T} = \Z$, which is done in \cite{BG}. 
For the convenience of the reader, let us include here a summary of this proof, using some lemmas that we have already proved in \S~\ref{ss.la_lemmas}.

The ``only if'' part of the theorem is not difficult, so let us consider the ``if'' part.
So assume that the gap between the $p$-th and the $(p+1)$-th singular values of $\psi^n_x$ increases uniformly exponentially with time $n$.
Fix any $x\in X$, and consider the sequence of spaces $U_p \big( \psi^n_{\phi^{-n}(x)} \big)$ in $\Gr_p(E_x)$.
Using Lemma~\ref{l.nochangeright}, we see that the distance between consecutive elements of the sequence decreases exponentially fast, and in particular the sequence has a limit $E^\cu_x$. Uniform control on the speed of convergence yields that $E^\cu$ is a continuous subbundle of $E$.
Lemma~\ref{l.domination_implies_slow_change} implies that this subbundle is invariant.\footnote{This step is missing in \cite{BG}.}
Analogously we obtain the subbundle $E^\cs$.
 
To conclude the proof, we need to show that the bundles $E^\cu$ and $E^\cs$ are transverse, and that the resulting splitting is indeed dominated. Here, \cite{BG} uses an ergodic-theoretical argument: The gap between singular values implies that for any Lyapunov regular point $x$, the difference $\lambda_p(x) - \lambda_{p+1}(x)$ between the $p$-th and the $(p+1)$-th Lyapunov exponents is bigger than some constant $2\epsilon > 0$. Moreover, Oseledets theorem implies that $E^\cu_x$ and $E^\cs_x$ are sums of Oseledets spaces, corresponding to Lyapunov exponents $\ge \lambda_p(x)$ and $\le \lambda_{p+1}(x)$, respectively.
Bearing these facts in mind, assume for a contradiction that $E^\cu$ does not dominate $E^\cs$. Then there exist points $x_i \in X$, unit vectors $v_i \in E^\cs_{x_i}$ and $w_i \in E^\cu_{x_i}$, and times $n_i \to \infty$ such that 
$$
\frac{\|\psi^{n_i} (v_i) \|}{\|\psi^{n_i} (w_i) \|} > e^{-\epsilon n_i} \, .
$$
It follows from a Krylov--Bogoliubov argument\footnote{i.e.\ convergence (in the weak star topology) of measures supported in long segments of orbits to an invariant measure.} (making use of the continuity of the subbundles; see \cite{BG} for details) that there exists a Lyapunov regular point $x$ such that $\lambda_p(x) - \lambda_{p+1}(x) \le \epsilon$.
This contradiction establishes domination.

\subsection{Exterior powers and applications}\label{ss.ext}

We recall quantitative facts about exterior powers; all the necessary information can be found in \cite[\S~3.2.3]{Arnold}.

The space $\Wedge^p \R^d$ is endowed with the inner product defined on decomposable $p$-vectors as:
\begin{equation}\label{e.inner_pr_ext_pow}
\langle v_1 \wedge \cdots \wedge v_p ,  w_1 \wedge \cdots \wedge w_p \rangle 
\coloneq \det \big( \langle v_i , w_j \rangle \big)_{i,j=1,\dots,p}
\end{equation}
Geometrically (see \cite[\S~IX.5]{Gantmacher}), $\| v_1 \wedge \dots \wedge v_p \|$ is the $p$-volume of the parallelepiped with edges $v_1$, \dots, $v_p$.
Therefore, if those vectors span a $p$-dimensional space $P \subset \R^d$ then,
for any $A \in \GL(d,\R)$,
\begin{equation}\label{e.jac_norm}
\frac{\|(\Wedge^p A)( v_1 \wedge  \dots \wedge v_p )\|}{\| v_1 \wedge  \dots \wedge v_p \|} 
= \jac(A|_P),
\end{equation}
where $\jac(\mathord{\cdot})$ denotes the \emph{jacobian} of a linear map, i.e., the product of its singular values.

The \emph{Pl\"ucker embedding} is the map $\iota \colon \grass_p(\R^d) \to \grass_1(\Wedge^p \R^d)$ such if $P$ is spanned by vectors $v_1$, \dots, $v_p \in \R^d$ then $\iota(P)$ is spanned  by the vector $ v_1 \wedge  \dots \wedge v_p \in \Wedge^p \R^d$. We metrize $\grass_1(\Wedge^p \R^d)$ in the way described in Subsection~\ref{ss.grass}.

\begin{lem}\label{l.Plucker}
For all $P$, $Q \in \grass_p(\R^d)$ we have:
\begin{align}
d(\iota(P),\iota(Q)) &\ge d(P,Q) \, , \label{e.plucker_dist} \\
\sin \angle \big(\iota(P),(\iota(Q))^\perp \big) &\ge [\sin\angle(P,Q^\perp)]^{\min\{p,d-p\}} \, . \label{e.plucker_ang}
\end{align}
\end{lem}

\begin{proof}
Let $\beta$, $\hat \beta \in [0,\pi/2]$ be such that:
$$
\sin \beta = d(P,Q) \quad \text{and} \quad \sin \hat\beta = d(\iota(P) , \iota(Q)) = \angle (\iota(P) , \iota(Q)) \, .
$$
Recalling \eqref{e.distGr4}, we have:
$$
\cos \beta = \angle(P,Q^\perp) \quad \text{and} \quad \cos \hat\beta = \angle(\iota(P), \iota(Q)^\perp) \, .
$$
Associated to the pair $P$, $Q$, Proposition~\ref{p.canonical} provides canonical angles $\beta_1 \ge \cdots \ge \beta_p$, where $\beta_1=\beta$, and canonical orthonormal bases $\{v_1,\dots,v_p\}$ and $\{w_1,\dots,w_p\}$ for $P$ and $Q$, respectively.
Using \eqref{e.inner_pr_ext_pow} we compute:
$$
\cos \hat{\beta} = \cos \angle (\iota(P) , \iota(Q)) = 
\langle v_1 \wedge \cdots \wedge v_p ,  w_1 \wedge \cdots \wedge w_p \rangle
= \prod_{i=1}^p \cos \beta_i \, .
$$
Then, on one hand, $\cos \hat{\beta} \le \cos \beta$, that is, $\sin \hat\beta \ge \sin \beta$, which is estimate \eqref{e.plucker_dist}.
On the other hand, since at most $n \coloneqq \min\{p,d-p\}$ canonical angles are nonzero, we have $\cos \hat\beta \ge (\cos \beta)^n$, which is estimate \eqref{e.plucker_ang}.
\end{proof}

Given $A \in \GL(d,\R)$, the singular values of $\Wedge^p A$ are the numbers $(\sigma_{i_1}(A) \cdots \sigma_{i_p}(A))$, where $1 \le i_1 < \cdots < i_p \le d$; in particular,
\begin{align}
\label{e.ext_sigma1}
\sigma_1 (\Wedge^p A) &= \sigma_1(A) \cdots \sigma_{p-1}(A) \sigma_p(A) \, ,\\ 
\label{e.ext_sigma2}
\sigma_2 (\Wedge^p A) &= \sigma_1(A) \cdots \sigma_{p-1}(A) \sigma_{p+1}(A) \, .
\end{align}
So $A$ has a gap of index $p$ if and only if $\Wedge^p A$ has a gap of index $1$.
In this case, the corresponding singular spaces are related via the Pl\"ucker embedding  $\iota \colon \grass_p(\R^d) \to \grass_1(\Wedge^p \R^d)$ as follows:
\begin{align}
\label{e.ext_U}
U_1(\Wedge^p A) &= \iota(U_p(A)) \, ,
\\
\label{e.ext_S}
S_{\binom{d}{p}-1}(\Wedge^p A) &= \big( \iota(S_{d-p}(A)^\perp) \big)^\perp \, .
\end{align}

The following result is used in the end of Section~\ref{s.morse}; it should be compared to Lemma~\ref{l.nochangeright}.

\begin{lem}\label{l.Rafael}
Suppose that $A$, $B$, and $AB \in \GL(d,\R)$ have gaps of index $p$, and that 
the spaces $S_{d-p}(A)$ and $U_p(B)$ are transverse.
Then:
$$
d \big( U_p(AB) , U_p(A) \big) \le \frac{\sigma_{p+1}(A)}{\sigma_p(A)}  \, \big[\sin \angle \big( U_p(B), S_{d-p}(A) \big) \big]^{-\min\{p,d-p\}} \, .
$$
\end{lem}

\begin{proof}
First we consider the case $p=1$.
Let $v$ and $w$ be unit vectors spanning the spaces $B^{-1}A^{-1}U_1(AB)$ and $B^{-1}U_1(B)$, respectively.
By inequality \eqref{e.pit1},
$$
\frac{\|ABw\|}{\|Bw\|} 
\ge \sigma_1(A) \, \sin \angle (Bw , S_{d-1}(A))  
\, ,
$$
while by inequality \eqref{e.pit2},
$$
\frac{\|Bv\|}{\|ABv\|} 
= \frac{\|A^{-1}ABv\|}{\|ABv\|} 
\ge \sigma_2(A)^{-1} \, \sin \angle (ABv , U_1(B))  
\, .
$$
Multiplying the two inequalities, we have:
$$
\frac{\sigma_1(A)}{\sigma_2(A)} \, d (U_1(AB) , U_1(B))  \, \sin \angle (U_1(B) , S_{d-1}(A)) 
\le \frac{\|ABw\|}{\|ABv\|} \times \frac{\|Bv\|}{\|Bw\|} \le 1 \times 1 = 1 \, ,
$$
which is the desired inequality in the case $p=1$.

Now consider arbitrary $p$.
If $A$, $B$, and $AB$ have gaps of index $p$ then
the corresponding exterior powers have gaps of index $1$,
and by the previous case we have:
$$
d \big( U_1(\Wedge^p(AB)) , U_p(\Wedge^p A) \big) \le \frac{\sigma_2(\Wedge^p A)}{\sigma_1(\Wedge^p A)} \, \big[ \sin \angle\big( U_1(\Wedge^p B), S_{\binom{d}{p}-1}(\Wedge^p A) \big) \big]^{-1}
$$
By \eqref{e.ext_sigma1} and \eqref{e.ext_sigma2}, we have:
$$
\frac{\sigma_2(\Wedge^p A)}{\sigma_1(\Wedge^p A)} = \frac{\sigma_{p+1}(A)}{\sigma_p(A)} \, .
$$
By \eqref{e.ext_U} and \eqref{e.plucker_dist}, we have:
$$
d \big( U_1(\Wedge^p(AB)) , U_p(\Wedge^p A) \big) 
= d \big( \iota(U_p(AB)) , \iota(U_p(A)) \big)
\ge 
d \big( U_p(AB) , U_p(A) \big) \, ,
$$
On the other hand, by \eqref{e.ext_U}, \eqref{e.ext_S}, and \eqref{e.plucker_ang}, we have:
\begin{multline*}
\sin \angle\big( U_1(\Wedge^p B), S_{\binom{d}{p}-1}(\Wedge^p A) \big) =
\sin \angle\Big( \iota(U_p(B)) , \big( \iota(S_{d-p}(A)^\perp) \big)^\perp \Big) 
\\
\ge \big[\sin \angle \big( U_p(B), S_{d-p}(A) \big) \big]^{\min\{p,d-p\}} \, .
\end{multline*}
Combining the facts above we obtain the lemma.
\end{proof}

\subsection{Expansion on the Grassmannian}

The aim of this subsection is to prove the following lemma, which is used in Subsection~\ref{ss.coco}.

\begin{lem}\label{l.expand}
Given $\alpha > 0$, there exists $b>0$ with the following properties.
Let $A \in \GL(d,\R)$.
Suppose that $P \in \Gr_p(\R^d)$ and
$Q \in \Gr_{d-p}(\R^d)$ satisfy
\begin{equation}\label{e.bi_separation}
\min \{ \angle(P,Q), \angle(AP,AQ) \} \ge \alpha \, .
\end{equation}
Then there exists $\delta > 0$ such that if $P_1$, $P_2 \in \Gr_p(\R^d)$ are $\delta$-close to $P$
then
\begin{equation}\label{e.min_expansion}
d(A(P_1),A(P_2)) \ge b \, \frac{\mm(A|_Q)}{\| A|_P \|} \, d(P_1,P_2) \, .
\end{equation}
\end{lem}

Before proving this lemma, we need still another characterization of the distance \eqref{e.distGr} on the Grassmannian. 
Suppose that $P$, $Q \in \Gr_p(\R^d)$ satisfy $d(P,Q)<1$.
Then $Q \cap P^\perp  =  \{0\}$, and so there exists a unique linear map 
\begin{equation}\label{e.graph}
L_{Q,P} \colon P \to P^\perp
\quad \text{such that} \quad
Q  =  \big\{ v+ L_{Q,P}(v) \st v \in P \big\}\, .
\end{equation}
We have:
\begin{equation} \label{e.sin_tan}
\|L_{Q,P}\|  =  \frac{d(P,Q)}{\sqrt{1 - d(P,Q)^2}} \, ;
\end{equation}
indeed, letting $\theta \in [0,\pi/2)$ be such that $\sin \theta  =  d(P,Q)$, using \eqref{e.distGr3}
we conclude that $\|L_{Q,P}\|  =  \tan \theta$.

\begin{lem}\label{l.d_and_norm}
Let $P$, $P_1$, $P_2 \in \Gr_p(\R^d)$.
For each $i = 1$, $2$, assume that $d(P_i,P) < 1/\sqrt{2}$, and let $L_i  =  L_{P_i,P}$.
Then:
$$
d(P_1,P_2) \le \|L_1 - L_2\| \le 4 d(P_1,P_2) \, .
$$

\end{lem}

\begin{proof}
Consider an arbitrary $u_1 \in P_1^\times$,
and write it as $u_1  =  v + L_1 v$ for some $v \in P^\times$.
By orthogonality, $\|u_1\| \ge \|v\|$.
Letting $u_2 \coloneqq v + L_2 v$, we have:
$$
\frac{\|u_1 - u_2\|}{\|u_1\|} \le \frac{\|L_1 v - L_2 v\|}{\|v\|} \le \|L_1 - L_2\|.
$$
Using \eqref{e.distGr3} we conclude that $d(P_1,P_2) \le \|L_1 - L_2\|$, which is the first announced inequality.

For each $i = 1$, $2$, we have $\|L_i\| \le 1$, 
as a consequence of \eqref{e.sin_tan} and the hypothesis $d(P_i,P) < 1/\sqrt{2}$.
Consider an arbitrary unit vector $v_1 \in P$.
Let $w_1 \coloneqq v_1 + L_1 v_1$, so $\|w_1\| \le 2$.
By \eqref{e.distGr3}, there exists $w_2 \in P_2^\times$ such that 
$$
\frac{\|w_2 - w_1\|}{\|w_1\|} \le d(P_1,P_2) \, .
$$
Let $v_2 \in P$ be such that $w_2  =  v_2 + L_2 v_2$.
By orthogonality, 
$$
\|w_1 - w_2\| \ge \max \{ {\|v_1-v_2\|}, {\|L_1 v_1 - L_2 v_2\|} \} \, , 
$$
so
\begin{align*}
\|L_1 v_1 - L_2 v_1 \| 
&\le \|L_1 v_1 - L_2 v_2 \| + \|L_2 v_2 - L_2 v_1\| \\
&\le \|L_1 v_1 - L_2 v_2 \| + \|v_2 - v_1\| \\
&\le 2 \|w_1-w_2\| \\
&\le 2 \|w_1\| d(P_1,P_2) \\
&\le 4 d(P_1,P_2) \, .
\end{align*}
Taking sup over unit vectors $v_1 \in P_1$ we obtain $\|L_1 - L_2 \| \le 4 d(P_1,P_2)$.
\end{proof}

\begin{proof}[Proof of Lemma~\ref{l.expand}]
First consider the case $\alpha  =  \pi/2$, so \eqref{e.bi_separation} means that 
$Q  =  P^\perp$ and $AQ  =  (AP)^\perp$.
Assume that $P_1$, $P_2$ are $\delta$-close to $P$, for some small $\delta>0$ to be chosen later. 
Recall notation \eqref{e.graph} and, for each $i = 1$, $2$, consider the linear maps $L_i \coloneqq L_{P_i,P}$ and $M_i \coloneqq L_{A P_i,A P}$,
which are well-defined since $\delta<1$ guarantees that $P_i \cap Q  =  \{0\}$.
These maps are related by
$L_i  =  (A^{-1}|_{AQ}) \circ  M_i \circ (A|_P)$.
As a consequence,
\[
\|L_1 - L_2\|  =  \big\| (A^{-1}|_{AQ}) \circ (M_1-M_2) \circ (A|_P) \big\|
\le \frac{\| A|_P \|}{\mm(A|_Q)} \, \|M_1 - M_2\| \, .
\]
Lemma~\ref{l.d_and_norm} entails:
$$
\|L_1 - L_2\| \ge d(P_1,P_2) \, .
$$
On the other hand, by \eqref{e.sin_tan} we have $\|L_i\| \le d(P_i,P) < \delta$, 
and therefore $\|M_i\| \le  \|A^{-1}\| \, \|L_i\| \, \|A\| < 1$, provided $\delta$ is chosen sufficiently small (depending on $A$).
Using \eqref{e.sin_tan} again we guarantee that $d(A P_i, A P) < 1/\sqrt{2}$. 
This allows us to apply Lemma~\ref{l.d_and_norm} and obtain:
$$
\|M_1 - M_2\| \ge \frac{1}{4} d(AP_1, AP_2) \, .
$$ Putting these three estimates together, we conclude that the lemma holds with $b = 1/4$, provided $\alpha  =  \pi/2$.

The general case can be reduced to the previous one by changes of inner products, whose effect on all involved quantities can be bounded by a factor depending only on $\alpha$. 
\end{proof}

\subsection{Additional lemmas about dominated splittings}

\subsubsection{More about domination of sequences of matrices}

Recall from \S~\ref{ss.dom_sequences} the definition of the sets $\cD(K,p,\mu,c,I)$.

\begin{lem}\label{l.seq_convergence}
Given $K>1$, $\mu>0$, and $c>0$, there exists $\ell \in \N$ and $\tilde{c}>c$
such that if $I \subset \Z$ is an interval 
and $(A_i)_{i\in I}$  is an element of $\cD(K,p,\mu,c,I)$, 
then the following properties hold:
\begin{enumerate}
\item\label{i.seq_U}	
If $n' < n < k$ all belong to $I$ and $k-n \ge \ell$ then:
$$
d \big( U_p(A_{k-1} \cdots A_{n+1} A_n) , \, U_p(A_{k-1} \cdots A_{n'+1} A_{n'}) \big)
< \tilde{c} e^{-\mu (k-n)} \, .
$$
\item\label{i.seq_S}	
If $k < m < m'$ all belong to $I$ and $m-k \ge \ell$ then: 
$$
d \big( S_{d-p}(A_{m-1} \cdots A_{k+1} A_k) , \, S_{d-p}(A_{m'-1} \cdots A_{k+1} A_k) \big)
< \tilde{c} e^{-\mu (m-k)} \, .
$$
\end{enumerate}
\end{lem}

\begin{proof}
Given $K$, $\mu$, and $c$, let $\ell \in \N$ be such that $ce^{-\mu \ell} < 1$.
Fix $(A_n)_{n\in I} \in \cD(K,p,\mu,c,I)$ and $k \in I$.
If $n \in I$ satisfies $n \le k - \ell$ then the space $P_n \coloneqq U_p(A_{k-1} \cdots A_{n+1} A_n)$ is well-defined.
If $n-1 \in I$ then it follows from Lemma~\ref{l.nochangeright} 
that
$$
d(P_n, P_{n-1}) \le K^2 c e^{-\mu (k-n)} \, .
$$
Therefore, if $n' < n$ belongs to $I$ then
$$
d(P_n, P_{n'}) \le d(P_n, P_{n-1}) + d(P_{n-1}, P_{n-2}) + \cdots + d(P_{n'+1}, P_{n'}) 
\le \tilde{c} e^{-\mu (k-n)} \, .
$$
where $\tilde{c} \coloneqq K^2 c/(1-e^{-\mu})$.
This proves part~(\ref{i.seq_U}) of the lemma. The argument applied to $A^{-1}$ yields part~(\ref{i.seq_S}).
\end{proof}

The following ``extension lemma'' is useful to deduce one-sided results from two-sided ones.
% Let $\N$ denote the set of nonnegative integers.

\begin{lem}\label{l.seq_extension}
Given $K>1$, $\mu>0$, and $c>0$, there exists $c'\geq c$ such that every  
one-sided sequence $(A_n)_{n\in \N}$ in $\cD(K,p,\mu,c,\N)$ can be extended to a two-sided sequence $(A_n)_{n\in \Z}$ in $\cD(K,p,\mu,c',\Z)$. 
\end{lem}

\begin{proof}
Fix $(A_n)_{n\in \N}$ in $\cD(K,p,\mu,c,\N)$. 
Let $Q_n \coloneqq S_{d-p} \big(A_{n-1} A_{n-2} \cdots A_0)$, which is defined for sufficiently large $n$.
By Lemma~\ref{l.seq_convergence}, the spaces $Q_n$ converge to some $Q \in \Gr_p(\R^d)$;
moreover, we can find some $n_0$ depending only on the constants $K$, $\mu$, $c$ (and not on the sequence of matrices) with the property that for all $n \ge n_0$
we have $d(Q_n,Q) < 1/\sqrt{2}$ or equivalently, by \eqref{e.distGr4}, $\angle(Q_n, Q^\perp) < \pi/4$.

Let $B$ be a matrix satisfying the following conditions:
$$
\|B^{\pm 1}\| \le K \, , \quad
\frac{\sigma_{p+1}}{\sigma_p}(B) < e^{-\mu} \, , \quad
B(Q)  =  Q  =  S_{d-p}(B) \, , \quad 
B(Q^\perp)  =  Q^\perp  =  U_p(B) \, .
$$
In particular, for all $m \ge 0$ we have $\frac{\sigma_{p+1}}{\sigma_p}(B^m) < e^{-\mu m}$.
Then, for all $n\ge n_0$, Lemma~\ref{l.no_cancellation} yields:  
$$
\frac{\sigma_{p+1}}{\sigma_p} (A_{n-1}\cdots A_0 B^m) \le 2 c e^{-\mu(n+m)} \, .
$$
As a consequence of Lemma~\ref{l.sing_value_change}, a similar estimate holds for $0 \le n<n_0$ with another constant replacing $2c$.
So if we set $A_n  =  B$ for all $n<0$ then the extended sequence $(A_n)_{n\in \Z}$ belongs to $\cD(K,p,\mu,c',\Z)$ for some suitable $c'$ depending only on $\mu$, $c$, and $n_0$.
\end{proof}

Combining the previous lemma with Proposition~\ref{p.BG_sequences}, we obtain:

\begin{cor}\label{c.seq_complement}
Given $K>1$, $\mu>0$, and $c>0$, there exist $\tilde c > 0$, $\tilde \mu > 0$, and $\alpha>0$ with the following properties.
For every one-sided sequence $(A_n)_{n\in \N}$ in $\cD(K,p,\mu,c,\N)$, 
the following limit exists:
$$
Q  =  \lim_{n\to \infty} S_{d-p}(A_{n-1}\cdots A_0) \, ,
$$
and moreover there exists $\tilde Q \in \Gr_{p}(\R^d)$ 
such that for every $n \ge 0$  
we have:
\begin{gather*}
\angle \big(A_{n-1}\cdots A_0 (\tilde Q) , A_{n-1}\cdots A_0 (Q) \big) \ge \alpha \, , \\
\frac{\| A_{n-1}\cdots A_0|_{Q} \|}{\mm (A_{n-1}\cdots A_0|_{\tilde Q})} < \tilde{c} e^{-\tilde{\mu} n} \, .
\end{gather*}
\end{cor}

\subsubsection{H\"older continuity of the bundles} 

\begin{teo}\label{teo-holder}   
Let $\phi^t : X \to X$ a Lipschitz flow on a compact metric space $X$ with $t\in \mathbb{T}$ and $E$ a vector bundle over $X$. Consider $\psi^t$ a be a $\beta$-H\"older linear flows over $\phi^t$ which admits a dominated splitting with constants $C, \lambda>0$. Then, if $\alpha < \beta$ and $e^{-\lambda} K^{\alpha} < 1$ where $K$ is a Lipschitz constant for $\phi^{1}$ and $\phi^{-1}$, then the maps $x \mapsto E^{\cu}(x)$ and $x \mapsto E^{\cs}(x)$ are $\alpha$-H\"older. 
\end{teo}

\begin{proof}[Sketch of the proof] 
Choose Lipschitz approximations $\hat E^{\cs}$ and $\hat E^{\cu}$ of $E^{\cs}$ and $E^{\cu}$ respectively. One can define then the bundle $\cE$ over $X$ corresponding to the linear maps from $\hat E^{\cu}(x)$ to $\hat E^{\cs}(x)$, that is $\cE(x) =  \mathrm{Hom}(\hat E^{\cu}(x),\hat E^{\cs}(x))$ . Let $\cT \subset \cE$ given as $\cT =  \{ (x,L) \in \cE \st \ \|L\|\leq 1 \}$. We consider the standard graph transform $H: \cT \to \cE$ given as:  
$$ H(x,L) =  (\phi^1(x),H_{x}( L)) $$ 
defined so that if $(w,v) \in \hat E^{\cs}(x) \oplus \hat E^{\cu}(x)$ is in the graph of $L$ (i.e.\ $w = Lv$) then one has that $(\psi^1 w, \psi^1 v)$ is in the graph of $H_{x}(L)$. It follows that the map $H$ is $\beta$-H\"older and a standard computation shows that some iterate leaves invariant the set $\cT$.

Given two sections $\sigma_0$ and $\sigma_1$ of $\cE$ such that $\sigma_i(x) \in \cT$ for every $x\in X$, one shows that the $\alpha$-H\"older distance of $H\circ \sigma_0$ and $H\circ \sigma_1$ is uniformly contracted\footnote{To show that this metric is contracted, one needs to assume that the constant $C$ appearing in the dominated splitting is equal to $1$. Otherwise, one can argue for an iterate and the same will hold.}  if $\alpha<\beta$ and $e^{-\lambda} k^{-\alpha}<1$ where $k \coloneqq \min\limits_{x \neq y} \frac{d(\phi^1(x),\phi^1(y))}{d(x,y)}$. Indeed, the graphs are getting contracted at a rate similar to $\lambda$ while points cannot approach faster than $k$ which gives that the $\alpha$-H\"older distance contracts\footnote{The need for $\alpha < \beta$ is evident, since the section cannot be more regular than the cocycle. In the computation this appears because an error term of the form $d(x,y)^{\beta-\alpha}$ appears which will then be negligible as $d(x,y)\to 0$ and gives the desired statement. See \cite[Section 4.4]{CroPo} for a similar computation.}.

As a consequence, one obtains that there is a unique $H$-invariant section $\sigma$ of this bundle which, moreover, it is $\alpha$-H\"older. 

It is direct to show that this section corresponds to the bundle $E^{\cu}$. A symmetric argument shows that $E^{\cs}$ is $\alpha$-H\"older, proving the result. 
\end{proof}

\begin{cor}\label{cor-holder} 
If $u \mapsto \psi^t_u$ is a $\beta$-H\"older family of linear flows over $\phi^t$ a Lipschitz flow and $\psi_0^t$ admits a dominated splitting, then there exists a neighborhood $D$ of $0$  such that the maps $(u,x) \mapsto E^{\cs}_u(x)$ and $(u,x) \mapsto E^{\cu}_u(x)$ are $\alpha$-H\"older.
\end{cor}

\begin{proof} 
Fix $\alpha$ as in the previous theorem. There exists a neighborhood $D$ of $0$ for which $\lambda K^\alpha <1$ (here $\lambda$ denotes the strength of the domination for $\psi^t_u$ with $u \in D$). 

Now one applies the previous theorem to the linear flow $\hat \psi^t$ over $\phi^t \times \mathrm{id} : X \times D \to X\times D$. Note that the Lipschitz constants of $\phi^1 \times \mathrm{id}$ and $\phi^{-1} \times \mathrm{id}$ are the same as the ones of $\phi^1$ and $\phi^{-1}$. The result follows. 
\end{proof}

%%%%%%%%%%%%%%%%%%%%%%%%%%%%%%%%%%%%%%%%%%%%%%%%%%%%%%%%%%%%%%%%%

%\marg{J: Actualize las referencias que ya estan publicadas}

\medskip

\author{\vbox{\footnotesize\noindent 
	Jairo Bochi\\ 
	Facultad de Matem\'aticas\\ Pontificia Universidad Cat\'olica de Chile \\ 
	Av.\ Vicu\~na Mackenna 4860 Santiago Chile \\
	\texttt{jairo.bochi@mat.uc.cl}
\bigskip}}

\author{\vbox{\footnotesize\noindent 
	Rafael Potrie\\ 
	CMAT Facultad de Ciencias\\ 
	Universidad de la Rep\'ublica\\ Igu\'a 4225 Montevideo Uruguay\\ 
	\texttt{rpotrie@cmat.edu.uy}
\bigskip}}

\author{\vbox{\footnotesize\noindent 
	Andr\'es Sambarino\\
	IMJ-PRG (CNRS UMR 7586)\\ Universit\'e Pierre et Marie Curie (Paris VI) \\
	4 place Jussieu 75005 Paris France\\
	\texttt{andres.sambarino@imj-prg.fr}
}}

\end{document}